\documentclass[11pt]{article}
\usepackage[utf8]{inputenc}
\usepackage[OT1,T1]{fontenc}
\usepackage{amsmath,amssymb}
\usepackage{mathrsfs}
\usepackage{bm}
\usepackage{upgreek}
\usepackage[all]{xy}
\SelectTips{cm}{10}
\usepackage[colorlinks,citecolor=black,linkcolor=black,urlcolor=black,
  pdfpagemode=None,pdfstartview=FitH]{hyperref}
\usepackage{tikz}
\usepackage{ctable,array}

\newtheorem{theorem}{Theorem}[section]
\newtheorem{lemma}[theorem]{Lemma}
\newtheorem{proposition}[theorem]{Proposition}
\newtheorem{corollary}[theorem]{Corollary}
\newenvironment{proof}{\trivlist
  \item[\hskip\labelsep{\itshape Proof.}]\upshape}{\nobreak\noindent
  $\square$\endtrivlist}
\newenvironment{other}[1]{\refstepcounter{theorem}\trivlist
  \item[\hskip\labelsep{\itshape #1~\thesection.\arabic{theorem}.}]
  \upshape}{\endtrivlist\bigbreak}

\renewenvironment{description}{\list{}
  {\setlength\leftmargin{2.5em}\setlength\labelwidth{0pt}
  }}
  {\endlist}

\DeclareMathOperator\Hom{Hom}
\DeclareMathOperator\End{End}
\DeclareMathOperator\Aut{Aut}
\DeclareMathOperator\wt{wt}
\DeclareMathOperator\Irr{Irr}
\DeclareMathOperator\Rep{Rep}
\DeclareMathOperator\Card{Card}
\DeclareMathOperator\rank{rank}
\DeclareMathOperator\supp{supp}
\DeclareMathOperator\linspan{span}
\DeclareMathOperator\res{res}
\DeclareMathOperator\IC{IC}
\newcommand\id{\mathord{\mathrm{id}}}
\newcommand\SL{\mathord{\mathrm{SL}}}
\newcommand\GL{\mathord{\mathrm{GL}}}
\newcommand\PGL{\mathord{\mathrm{PGL}}}
\newcommand\SO{\mathord{\mathrm{SO}}}
\newcommand\Gr{\mathord{\mathrm{Gr}}}
\newcommand\Perv{\mathord{\mathrm{Perv}}}
\newcommand\BM{{\mathrm{BM}}}
\newcommand\BDGr[1]{\mathord{\mathcal Gr^{\mathrm{BD}}_{#1}}}
\newcommand\BDConv{\mathord{\mathcal Gr}}
\newcommand\PACS{\Phi^{\vee,+}_a}
\newcommand\caltimes{\mathbin{\propto}}
\newcommand\dbllangle{\langle\!\langle}
\newcommand\dblrangle{\rangle\!\rangle}

\begin{document}
\title{Bases of tensor products and geometric Satake correspondence}
\author{Pierre Baumann, Stéphane Gaussent and Peter Littelmann}
\date{}
\maketitle

\begin{center}
\textit{To the memory of C.\ S.\ Seshadri}
\end{center}

\begin{abstract}
\noindent
The geometric Satake correspondence can be regarded as a geometric
construction of the rational representations of a complex connected
reductive group $G$. In their study of this correspondence,
Mirković and Vilonen introduced algebraic cycles that provide a linear
basis in each irreducible representation. Generalizing this construction,
Goncharov and Shen define a linear basis in each tensor product of
irreducible representations. We investigate these bases and show that
they share many properties with the dual canonical bases of Lusztig.
\end{abstract}

\section{Introduction}
\label{se:Intro}
Let $G$ be a connected reductive group over the field of complex
numbers, endowed with a Borel subgroup and a maximal torus. Let
$\Lambda^+$ be the set of dominant weights relative to these data.
Given $\lambda\in\Lambda^+$, denote by $V(\lambda)$ the irreducible
rational representation of $G$ with highest weight $\lambda$. Given
a finite sequence $\bm\lambda=(\lambda_1,\ldots,\lambda_n)$ in
$\Lambda^+$, define
$$V(\bm\lambda)=V(\lambda_1)\otimes\cdots\otimes V(\lambda_n).$$
A construction due to Mirković and Vilonen~\cite{MirkovicVilonen} in
the context of the geometric Satake correspondence endows the spaces
$V(\lambda)$ with linear bases. Specifically, $V(\lambda)$ is
identified with the intersection homology of a parabolic affine
Schubert variety, and Mirković and Vilonen's basis is the datum of
a family $\mathscr Z(\lambda)$ of algebraic cycles in this variety.
In~\cite{GoncharovShen}, Goncharov and Shen extend this construction
to the tensor products $V(\bm\lambda)$. In the present paper, we
investigate the properties of these bases, which we call MV bases.

We show that MV bases are compatible with the isotypical filtration
of the representations $V(\bm\lambda)$. They are also compatible with
the action of the Chevalley generators of the Lie algebra $\mathfrak g$
of $G$, in the sense that the leading terms of this action define on
$\mathscr Z(\bm\lambda)$ the structure of a crystal in the sense of
Kashiwara; this result is due to Braverman and Gaitsgory
\cite{BravermanGaitsgory01} in the case $n=1$. With the help of the
path model~\cite{Littelmann95}, we prove that there is a natural
crystal isomorphism
$$\mathscr Z(\bm\lambda)\cong\mathscr Z(\lambda_1)\otimes\cdots
\otimes\mathscr Z(\lambda_n)$$
(Braverman and Gaitsgory implicitly treated the case $n=2$).

We study the transition matrix between the MV basis of a tensor product
$V(\bm\lambda)$ and the tensor product of the MV bases of the different
factors $V(\lambda_1)$, \dots, $V(\lambda_n)$. Using the fusion product
in the sense of Beilinson and Drinfeld~\cite{BeilinsonDrinfeld}, we
explain that the entries of this transition matrix can be computed
as intersection multiplicities. As a consequence, the transition matrix
is unitriangular with nonnegative integral entries.

The properties stated above are analogues of results obtained by Lusztig
about the dual canonical bases. To be specific, Lusztig~\cite{Lusztig93}
defines a notion of based module (module endowed with a basis enjoying
certain specific properties) over the quantized enveloping algebra
$U_v(\mathfrak g)$ and shows the following facts:
\begin{itemize}
\item
A simple module over $U_v(\mathfrak g)$, endowed with its canonical
basis, is a based module.
\item
The tensor product of finitely many based modules can be endowed with
a basis that makes it a based module. This basis is constructed from
the tensor product of the bases of the factors by adding corrective
terms in an unitriangular fashion.
\item
The basis of a based module is compatible with the decreasing isotypical
filtration of the module. Each subquotient in this filtration, endowed
with the induced basis, is isomorphic as a based module to the direct
sum of copies of a simple module endowed with its canonical basis.
\end{itemize}
The dual canonical bases for the representations $V(\bm\lambda)$
(see~\cite{FrenkelKhovanov}) can then be defined by dualizing Lusztig's
construction and specializing the quantum parameter at~$v=1$.

Because of its compatibility with the isotypical filtration, the dual
canonical basis of a tensor product $V(\bm\lambda)$ yields a linear
basis of the invariant subspace $V(\bm\lambda)^G$. Just as well, the
MV basis provides a linear basis (sometimes called the Satake basis)
of $V(\bm\lambda)^G$. The Satake basis and the dual canonical basis
of $V(\bm\lambda)^G$ generally differ (the paper
\cite{FontaineKamnitzerKuperberg} provides a counterexample);
nonetheless we show that the Satake basis enjoys the first two items in
Khovanov and Kuperberg's list of properties for the dual canonical basis
\cite{KhovanovKuperberg}. In particular, after restriction to the
invariant subspaces, the signed permutation
$$V(\lambda_1)\otimes V(\lambda_2)\otimes\cdots\otimes V(\lambda_n)
\xrightarrow\simeq
V(\lambda_2)\otimes\cdots\otimes V(\lambda_n)\otimes V(\lambda_1)$$
maps the Satake basis of the domain to that of the codomain.

As explained in \cite{BaumannKamnitzerKnutson}, the MV bases of the
irreducible representations $V(\lambda)$ can be glued together to produce
a basis of the algebra $\mathbb C[N]$ of regular functions on the
unipotent radical $N$ of~$B$. Of particular interest would be any
relation with the cluster algebra structure on $\mathbb C[N]$
\cite{BerensteinFominZelevinsky,GeissLeclercSchroer07}. The methods
developed in the present paper allow for explicit computations;
for instance we show that the cluster monomials attached to certain seeds
belong to the MV basis. However $\mathbb C[N]$ is not of finite cluster
type in general; in other words cluster monomials do not span the whole
space. In type $D_4$ we compare the MV basis element with the dual
canonical and the dual semicanonical elements at the first position not
covered by cluster monomials.

\textit{Acknowledgements.}
P.B.\ and S.G.'s research is in part supported by the ANR project
GeoLie, ANR-15-CE40-0012. P.L.'s research is in part supported by the
SFB/TRR 191 ‘Symplectic Structures in Geometry, Algebra and Dynamics’.
P.B.\ thanks Joel Kamnitzer for several discussions about the fusion
product, which led to the main result of sect.~\ref{se:CompTPB}. The
authors warmly thank Simon Riche for help with several technicalities.
P.B.\ also gratefully acknowledges the hospitality of the Centre
de Recherches Mathématiques in Montreal, where research reported
in sect.~7.1 was carried out.

\section{Mirković-Vilonen cycles and bases}
\label{se:MVCycBas}
In the whole paper $G$ is a connected reductive group over $\mathbb C$,
endowed with a Borel subgroup $B$ and a maximal torus $T\subseteq B$.
We denote by $\Lambda$ the character lattice of $T$, by
$\Phi\subseteq\Lambda$ the root system of $(G,T)$, by $\Phi^\vee$ the
coroot system, and by $W$ the Weyl group. The datum of $B$ determines
a set of positive root in $\Phi$; we endow $\Lambda$ with the dominance
order $\leq$ relative to this and denote the cone of dominant weights
by $\Lambda^+$. We denote by $\rho:\Lambda\to\mathbb Q$ the half-sum
of the positive coroots; so $\langle\rho,\alpha\rangle=1$ for each
simple root $\alpha$. The Langlands dual of $G$ is the connected
reductive group $G^\vee$ over $\mathbb C$ built from the dual torus
$T^\vee=\Lambda\otimes_{\mathbb Z}\mathbb G_m$ and the root system
$\Phi^\vee$; the positive coroots define a Borel subgroup
$B^\vee\subseteq G^\vee$.

\subsection{Recollection on the geometric Satake correspondence}
\label{ss:RecGeomSat}
The geometric Satake correspondence was devised by Lusztig
\cite{Lusztig81} and given its definitive shape by Beilinson and
Drinfeld~\cite{BeilinsonDrinfeld} and Mirković and Vilonen
\cite{MirkovicVilonen}. Additional references for the material
exposed in this section are \cite{Zhu} and \cite{BaumannRiche}.

As recalled in the introduction of \cite{PappasRapoport}, loop
groups appear under several guises across mathematics: there is
the differential-geometric variant, the algebraic-geometric one, etc.
We adopt the framework of Lie theory and Kac-Moody groups~\cite{Kumar}.
For instance, though the affine Grassmannian is a (generally
non-reduced) ind-scheme, we will only look at its complex ind-variety
structure.

Let $\mathcal O=\mathbb C[\![z]\!]$ be the ring of formal
power series in $z$ with complex coefficients and let $\mathcal K=
\mathbb C(\!(z)\!)$ be the fraction field of $\mathcal O$.
The affine Grassmannian of the Langlands dual $G^\vee$ is the
homogeneous space $\Gr=G^\vee(\mathcal K)/G^\vee(\mathcal O)$.
This space, a partial flag variety for an affine Kac-Moody group,
is endowed with the structure of an ind-variety.

Each weight $\lambda\in\Lambda$ gives a point $z^\lambda$ in
$T^\vee(\mathcal K)$, whose image in $\Gr$ is denoted by $L_\lambda$.
The $G^\vee(\mathcal O)$-orbit through $L_\lambda$ in $\Gr$ is denoted
by $\Gr^\lambda$; this is a smooth connected simply-connected variety
of dimension $2\rho(\lambda)$. The Cartan decomposition implies that
$$\Gr=\bigsqcup_{\lambda\in\Lambda^+}\Gr^\lambda;\quad\text{moreover}\quad
\overline{\Gr^\lambda}=\bigsqcup_{\substack{\mu\in\Lambda^+\\[2pt]
\mu\leq\lambda}}\Gr^\mu.$$

Let $\Perv(\Gr)$ be the category of $G^\vee(\mathcal O)$-equivariant
perverse sheaves on $\Gr$ (for the middle perversity) supported on
finitely many $G^\vee(\mathcal O)$-orbits, with coefficients in
$\mathbb C$. This is an abelian semisimple category; the simple
objects in $\Perv(\Gr)$ are the intersection cohomology sheaves
$$\mathscr I_\lambda=\IC\Bigl(\overline{\Gr^\lambda},\,
\underline{\mathbb C}\Bigr).$$
(By convention, IC sheaves are shifted so as to be perverse.)

Let $\theta\in\Lambda$ be a dominant and regular weight. The embedding
$$\mathbb C^\times\xrightarrow\theta T^\vee(\mathbb C)\to
G^\vee(\mathcal K)$$
gives rise to an action of $\mathbb C^\times$ on $\Gr$. For each
$\mu\in\Lambda$, the point $L_\mu$ is fixed by this action; we denote its
stable and unstable sets by
$$S_\mu=\Bigl\{x\in\Gr\Bigm|\lim_{c\to0}\theta(c)\cdot x=L_\mu\Bigr\}
\quad\text{and}\quad
T_\mu=\Bigl\{x\in\Gr\Bigm|\lim_{c\to\infty}\theta(c)\cdot x=L_\mu\Bigr\}$$
and denote the inclusion maps by $s_\mu:S_\mu\to\Gr$ and $t_\mu:T_\mu\to\Gr$.

Given $\mu\in\Lambda$ and $\mathscr A\in\Perv(\Gr)$, Mirković and
Vilonen identify the homology groups
$$H_c\bigl(S_\mu,\,(s_\mu)^*\mathscr A\bigr)\quad\text{and}\quad
H\bigl(T_\mu,\,(t_\mu)^!\mathscr A\bigr)$$
via Braden's hyperbolic localization, show that these groups are
concentrated in degree $2\rho(\mu)$, and define
$$F_\mu(\mathscr A)=H^{2\rho(\mu)}\bigl(T_\mu,\,(t_\mu)^!\mathscr A\bigr)
\quad\text{and}\quad
F(\mathscr A)=\bigoplus_{\mu\in\Lambda}F_\mu(\mathscr A).$$

Then $F$ is an exact and faithful functor from $\Perv(\Gr)$ to the
category of finite dimensional $\Lambda$-graded $\mathbb C$-vector
spaces. Mirković and Vilonen prove that $F$ induces an equivalence
$\overline F$ from $\Perv(\Gr)$ to the category $\Rep(G)$ of finite
dimensional rational representations of $G$, the $\Lambda$-graduation
on $F(\mathscr A)$ giving rise to the decomposition of
$\overline F(\mathscr A)$ into weight subspaces. In the course of the
proof, it is shown that~$\overline F$ maps $\mathscr I_\lambda$ to the
irreducible highest weight representation~$V(\lambda)$.

The map $G^\vee(\mathcal K)\to\Gr$ is a principal
$G^\vee(\mathcal O)$-bundle. From the $G^\vee(\mathcal O)$-space $\Gr$,
we form the associated bundle
$$\Gr_2=G^\vee(\mathcal K)\,\times^{G^\vee(\mathcal O)}\;\Gr.$$
This space is called the $2$-fold convolution variety and has the
structure of an ind-variety. The action of $G^\vee(\mathcal K)$ on
$\Gr$ defines a map $m:\Gr_2\to\Gr$ of ind-varieties. Let
$p:G^\vee(\mathcal K)\to\Gr$ and $q:G^\vee(\mathcal K)\times\Gr\to\Gr_2$
be the quotient maps. Given two equivariant perverse sheaves
$\mathscr A_1$ and $\mathscr A_2$ on $\Gr$, there is a unique
equivariant perverse sheaf $\mathscr A_1\,\widetilde\boxtimes\,\mathscr A_2$
on $\Gr_2$ such that
$$p^*\mathscr A_1\boxtimes\mathscr A_2=
q^*\bigl(\mathscr A_1\,\widetilde\boxtimes\,\mathscr A_2\bigr)$$
in the equivariant derived category of constructible sheaves on
$G^\vee(\mathcal K)\times\Gr$. We then define the convolution product
of $\mathscr A_1$ and $\mathscr A_2$ to be
$$\mathscr A_1*\mathscr A_2=m_*\bigl(\mathscr A_1\,\widetilde\boxtimes\,
\mathscr A_2\bigr).$$
Using Beilinson and Drinfeld's fusion product, one defines
associativity and commutativity constraints and obtains a monoidal
structure on $\Perv(\Gr)$. Then $F$ is a tensor functor; in particular,
the fusion product imparts an explicit identification of $\Lambda$-graded
vector spaces
$$F(\mathscr A_1*\mathscr A_2)\cong F(\mathscr A_1)\otimes
F(\mathscr A_2)$$
for any $(\mathscr A_1,\mathscr A_2)\in\Perv(\Gr)^2$.

\subsection{Mirković-Vilonen cycles}
\label{ss:MVCyc}
In this paper we study tensor products
$V(\bm\lambda)=V(\lambda_1)\otimes\cdots\otimes V(\lambda_n)$, where
$\bm\lambda=(\lambda_1,\ldots,\lambda_n)$ is a sequence of dominant
weights. Accordingly, we want to consider convolution products
$$\mathscr I_{\bm\lambda}=\mathscr I_{\lambda_1}*\cdots*\mathscr
I_{\lambda_n}$$
and therefore need the $n$-fold convolution variety
$$\Gr_n=\underbrace{G^\vee(\mathcal K)\,\times^{G^\vee(\mathcal O)}\;\cdots\,
\times^{G^\vee(\mathcal O)}\;G^\vee(\mathcal K)}_{n\ \text{factors}\
G^\vee(\mathcal K)}\,/\,G^\vee(\mathcal O).$$
As customary, we denote elements in $\Gr_n$ as classes $[g_1,\ldots,g_n]$
of tuples of elements in $G^\vee(\mathcal K)$. Plainly $\Gr_1=\Gr$;
in consequence, we write the quotient map $G^\vee(\mathcal K)\to\Gr$
as $g\mapsto[g]$. We define a map $m_n:\Gr_n\to\Gr$ by setting
$m_n([g_1,\ldots,g_n])=[g_1\ldots g_n]$.

Given $G^\vee(\mathcal O)$-stable subsets $K_1$, \dots, $K_n$ of $\Gr$,
we define
$$K_1\,\widetilde\times\cdots\widetilde\times\,K_n=\bigl\{[g_1,\ldots,g_n]
\in\Gr_n\bigm|[g_1]\in K_1,\ \ldots,\ [g_n]\in K_n\bigr\}.$$
Alternatively, $K_1\,\widetilde\times\cdots\widetilde\times\,K_n$
can be defined as
$$\widehat K_1\,\times^{G^\vee(\mathcal O)}\;\cdots\,
\times^{G^\vee(\mathcal O)}\;\widehat K_n\,/\,G^\vee(\mathcal O)$$
where each $\widehat K_j$ is the preimage of $K_j$ under the quotient
map $G^\vee(\mathcal K)\to\Gr$.

For $\bm\lambda=(\lambda_1,\ldots,\lambda_n)$ in $(\Lambda^+)^n$, we set
$$\Gr_n^{\bm\lambda}=\Gr^{\lambda_1}\,\widetilde\times\cdots
\widetilde\times\,\Gr^{\lambda_n}\quad\text{and}\quad
|\bm\lambda|=\lambda_1+\cdots+\lambda_n.$$
Viewing $\Gr_n^{\bm\lambda}$ as an iterated fibration with base
$\Gr^{\lambda_1}$ and successive fibers $\Gr^{\lambda_2}$, \dots,
$\Gr^{\lambda_n}$, we infer that it is a smooth connected
simply-connected variety of dimension $2\rho(|\bm\lambda|)$; also
$$\Gr_n=\bigsqcup_{\bm\lambda\in(\Lambda^+)^n}\Gr_n^{\bm\lambda}
\quad\text{and}\quad\overline{\Gr_n^{\bm\lambda}}=
\overline{\Gr^{\lambda_1}}\,\widetilde\times\cdots\widetilde\times\,
\overline{\Gr^{\lambda_n}}=\bigsqcup_{\substack{\bm\mu\in(\Lambda^+)^n\\[2pt]
\mu_1\leq\lambda_1,\,\ldots,\,\mu_n\leq\lambda_n}}\Gr_n^{\bm\mu}.$$

\begin{proposition}
\label{pr:ICnFold}
Let $\bm\lambda\in(\Lambda^+)^n$. Then $\mathscr I_{\bm\lambda}$ is
the direct image by $m_n$ of the intersection cohomology sheaf
of $\overline{\Gr_n^{\bm\lambda}}$ with trivial local system, to wit
$$\mathscr I_{\bm\lambda}=(m_n)_*\IC\Bigl(
\overline{\Gr_n^{\bm\lambda}},\,\underline{\mathbb C}\Bigr),$$
and the cohomology sheaves $\mathscr H^k\IC\Bigl(
\overline{\Gr_n^{\bm\lambda}},\,\underline{\mathbb C}\Bigr)$
vanish unless $k$ and $2\rho(|\bm\lambda|)$ have the same parity.
\end{proposition}
\begin{proof}
We content ourselves with the case $n=2$; the proof is the same in
the general case but requires more notation. Working out the
technicalities explained in \cite{BaumannRiche}, \S1.16.4, we get
$$\IC\Bigl(\overline{\Gr_2^{(\lambda_1,\lambda_2)}},\,
\underline{\mathbb C}\Bigr)=\IC\Bigl(\overline{\Gr^{\lambda_1}},\,
\underline{\mathbb C}\Bigr)\,\widetilde\boxtimes\,
\IC\Bigl(\overline{\Gr^{\lambda_2}},\,\underline{\mathbb C}\Bigr).$$
Applying $(m_2)_*$ then gives the announced equality, while the
parity property follows from \cite{Lusztig81}, sect.~11.
\end{proof}

A key argument in Mirković and Vilonen's proof of the geometric
Satake correspondence is the fact that for any $\lambda\in\Lambda^+$
and $\mu\in\Lambda$, all the irreducible components of
$\overline{\Gr^\lambda}\cap S_\mu$ (respectively,
$\overline{\Gr^\lambda}\cap T_\mu$) have dimension
$\rho(\lambda+\mu)$ (respectively, $\rho(\lambda-\mu)$)
(\cite{MirkovicVilonen}, Theorem~3.2). We need a similar result
for the intersections $\overline{\Gr_n^{\bm\lambda}}
\cap(m_n)^{-1}(S_\mu)$ and $\overline{\Gr_n^{\bm\lambda}}\cap
(m_n)^{-1}(T_\mu)$ inside the $n$-fold convolution variety.

Let $N^\vee$ be the unipotent radical of $B^\vee$. Then $S_\mu$ is
the $N^\vee(\mathcal K)$-orbit through $L_\mu$; this well-known fact
follows from the easily proved inclusion
$N^\vee(\mathcal K)\,L_\mu\subseteq S_\mu$ and the Iwasawa decomposition
\begin{equation}
\label{eq:IwasDec}
G^\vee(\mathcal K)=\bigsqcup_{\mu\in\Lambda}N^\vee(\mathcal K)\,
z^\mu\,G^\vee(\mathcal O).
\end{equation}
We record that $S_\mu=(N^\vee(\mathcal K)\,z^\mu)\,/\,N^\vee(\mathcal O)$
and that for each $\lambda\in\Lambda^+$, the action of the connected
subgroup $N^\vee(\mathcal O)$ leaves stable the intersection
$\overline{\Gr^\lambda}\cap S_\mu$, hence leaves stable each irreducible
component of this intersection.

The construction of the $n$-fold convolution variety is functorial in
the group $G^\vee$; applied to the inclusion $B^\vee\to G^\vee$, this
remark provides a natural map
$$\Psi:\bigsqcup_{(\mu_1,\ldots,\mu_n)\in\Lambda^n}
\bigl(N^\vee(\mathcal K)\,z^{\mu_1}\bigr)\,\times^{N^\vee(\mathcal O)}\;
\cdots\,\times^{N^\vee(\mathcal O)}\;
\bigl(N^\vee(\mathcal K)\,z^{\mu_n}\bigr)\,/\,N^\vee(\mathcal O)
\to\Gr_n.$$
Using \eqref{eq:IwasDec}, we easily see that $\Psi$ is bijective.

Given weights $\mu_1$, \dots, $\mu_n$ and $N^\vee(\mathcal O)$-stable
subsets $Z_1\subseteq S_{\mu_1}$, \dots, $Z_n\subseteq S_{\mu_n}$, we
define
$$Z_1\ltimes\cdots\ltimes Z_n=\widetilde Z_1\,\times^{N^\vee(\mathcal O)}\;
\cdots\,\times^{N^\vee(\mathcal O)}\;\widetilde Z_n\,/\,N^\vee(\mathcal O)$$
where each $\widetilde Z_j$ is the preimage of $Z_j$ under
the quotient map $N^\vee(\mathcal K)\,z^{\mu_j}\to S_{\mu_j}$.
With this notation, the bijectivity of $\Psi$ implies that
$$\Gr_n=\bigsqcup_{(\mu_1,\ldots,\mu_n)\in\Lambda^n}
\Psi\bigl(S_{\mu_1}\ltimes\cdots\ltimes S_{\mu_n}\bigr).$$
If $Z_1$, \dots, $Z_n$ are varieties, then $Z_1\ltimes\cdots\ltimes Z_n$
is an iterated fibration with basis $Z_1$ and successive fibers $Z_2$,
\dots, $Z_n$ and $\Psi$ induces an homeomorphism from
$Z_1\ltimes\cdots\ltimes Z_n$ onto its image.

We use the symbol $\Irr(-)$ to designate the set of irreducible
components of its argument. For $\lambda\in\Lambda^+$,
$\bm\lambda\in(\Lambda^+)^n$, and $\mu\in\Lambda$, we define
$${}_*\mathscr Z(\lambda)_\mu=\Irr\Bigl(\overline{\Gr^\lambda}\cap
S_\mu\Bigr)\quad\text{and}\quad{}_*\mathscr Z(\bm\lambda)_\mu=
\Irr\Bigl(\overline{\Gr_n^{\bm\lambda}}\cap(m_n)^{-1}(S_\mu)\Bigr).$$

\begin{proposition}
\label{pr:MVCyc}
Let $\bm\lambda=(\lambda_1,\ldots,\lambda_n)$ in $(\Lambda^+)^n$
and let $\mu\in\Lambda$.
\begin{enumerate}
\item
\label{it:PrMVCa}
All the irreducible components of $\overline{\Gr_n^{\bm\lambda}}\cap
(m_n)^{-1}(S_\mu)$ have dimension $\rho(|\bm\lambda|+\mu)$.
\item
\label{it:PrMVCb}
The map $(Z_1,\ldots,Z_n)\mapsto\overline{\Psi(Z_1\ltimes\cdots\ltimes Z_n)}$
induces a bijection
$$\bigsqcup_{\substack{(\mu_1,\ldots,\mu_n)\in\Lambda^n\\[2pt]
\mu_1+\cdots+\mu_n=\mu}}{}_*\mathscr Z(\lambda_1)_{\mu_1}
\times\cdots\times\,{}_*\mathscr Z(\lambda_n)_{\mu_n}
\xrightarrow\simeq\,{}_*\mathscr Z(\bm\lambda)_\mu.$$
(The bar above $\Psi(Z_1\ltimes\cdots\ltimes Z_n)$ means closure
in $(m_n)^{-1}(S_\mu)$.)
\end{enumerate}
\end{proposition}
\begin{proof}
One easily checks that $(m_n\circ\Psi)\bigl(S_{\mu_1}\ltimes\cdots\ltimes
S_{\mu_n}\bigr)\subseteq S_{\mu_1+\cdots+\mu_n}$, whence
$$(m_n)^{-1}(S_\mu)=\bigsqcup_{\substack{(\mu_1,\ldots,\mu_n)\in
\Lambda^n\\[2pt]
\mu_1+\cdots+\mu_n=\mu}}\Psi\bigl(S_{\mu_1}\ltimes\cdots\ltimes
S_{\mu_n}\bigr)$$
for any $\mu\in\Lambda$. Adding $\bm\lambda=(\lambda_1,\ldots,\lambda_n)$
to the mix, we see that
$$\overline{\Gr_n^{\bm\lambda}}\cap(m_n)^{-1}(S_\mu)=
\bigsqcup_{\substack{(\mu_1,\ldots,\mu_n)\in\Lambda^n\\[2pt]
\mu_1+\cdots+\mu_n=\mu}}
\Psi\Bigl(\Bigl(\overline{\Gr^{\lambda_1}}\cap S_{\mu_1}\Bigr)
\ltimes\cdots\ltimes
\Bigl(\overline{\Gr^{\lambda_n}}\cap S_{\mu_n}\Bigr)\Bigr)$$
is the disjoint union over $(\mu_1,\ldots,\mu_n)$ of an iterated
fibration with base $\overline{\Gr^{\lambda_1}}\cap S_{\mu_1}$
and successive fibers $\overline{\Gr^{\lambda_2}}\cap S_{\mu_2}$,
\dots, $\overline{\Gr^{\lambda_n}}\cap S_{\mu_n}$. Statement
\ref{it:PrMVCa} follows directly from this observation and Mirković
and Vilonen's dimension estimates. Statement~\ref{it:PrMVCb}
requires the additional argument that the spaces
$\overline{\Gr^{\lambda_j}}\cap S_{\mu_j}$ are contractible
(because of the $\mathbb C^\times$-action), so the monodromy is
necessarily trivial.
\end{proof}

For $\lambda\in\Lambda^+$, $\bm\lambda\in(\Lambda^+)^n$, and
$\mu\in\Lambda$, we similarly define
$$\mathscr Z(\lambda)_\mu=\Irr\Bigl(\overline{\Gr^\lambda}\cap
T_\mu\Bigr)\quad\text{and}\quad\mathscr Z(\bm\lambda)_\mu=
\Irr\Bigl(\overline{\Gr_n^{\bm\lambda}}\cap(m_n)^{-1}(T_\mu)\Bigr).$$
Then all cycles in $\mathscr Z(\bm\lambda)_\mu$ have dimension
$\rho(|\bm\lambda|-\mu)$ and there is a natural bijection
\begin{equation}
\label{eq:BijMVCyc}
\bigsqcup_{\substack{(\mu_1,\ldots,\mu_n)\in\Lambda^n\\[2pt]
\mu_1+\cdots+\mu_n=\mu}}\mathscr Z(\lambda_1)_{\mu_1}
\times\cdots\times\,\mathscr Z(\lambda_n)_{\mu_n}
\xrightarrow\simeq\,\mathscr Z(\bm\lambda)_\mu.
\end{equation}

Elements in ${}_*\mathscr Z(\lambda)_\mu$, $\mathscr Z(\lambda)_\mu$,
${}_*\mathscr Z(\bm\lambda)_\mu$ or $\mathscr Z(\bm\lambda)_\mu$ are
called Mirković-Vilonen (MV) cycles. For future use, we note that the
map $Z\mapsto Z\cap\Gr_n^{\bm\lambda}$ provides a bijection from
$\mathscr Z(\bm\lambda)_\mu$ onto the set of irreducible components
of $\Gr_n^{\bm\lambda}\cap(m_n)^{-1}(T_\mu)$. (Each
$Z\in\mathscr Z(\bm\lambda)_\mu$ meets the open subset
$\Gr_n^{\bm\lambda}$ of $\overline{\Gr_n^{\bm\lambda}}$, because
the dimension of $\bigl(\overline{\Gr_n^{\bm\lambda}}\setminus
\Gr_n^{\bm\lambda}\bigr)\cap(m_n)^{-1}(T_\mu)$ is smaller than the
dimension of~$Z$.)

\subsection{Mirković-Vilonen bases}
\label{ss:MVBas}
Following Goncharov and Shen (\cite{GoncharovShen}, sect.~2.4),
we now define the MV basis of the tensor product representations
$$V(\bm\lambda)=F(\mathscr I_{\bm\lambda})=\bigoplus_{\mu\in\Lambda}
F_\mu(\mathscr I_{\bm\lambda}).$$

Let $\bm\lambda\in(\Lambda^+)^n$ and let $\mu\in\Lambda$.
By base change in the Cartesian square
$$\xymatrix@R=26pt@C=36pt{\overline{\Gr_n^{\bm\lambda}}\cap
(m_n)^{-1}(T_\mu)\ar[r]^-f\ar[d]&\overline{\Gr_n^{\bm\lambda}}
\ar[d]^{m_n}\\T_\mu\ar[r]^{t_\mu}&\Gr}$$
we compute
$$F_\mu(\mathscr I_{\bm\lambda})
=H^{2\rho(\mu)}\Bigl(T_\mu,(t_\mu)^!(m_n)_*\IC\Bigl(
\overline{\Gr_n^{\bm\lambda}},\,\underline{\mathbb C}\Bigr)\Bigr)
=H^{2\rho(\mu)}\Bigl(\overline{\Gr_n^{\bm\lambda}}\cap
(m_n)^{-1}(T_\mu),\,f^!\IC\Bigl(\overline{\Gr_n^{\bm\lambda}},
\underline{\mathbb C}\Bigr)\Bigr).$$

Let $j:\Gr_n^{\bm\lambda}\to\overline{\Gr_n^{\bm\lambda}}$ and
$g:\Gr_n^{\bm\lambda}\cap(m_n)^{-1}(T_\mu)\to\Gr_n^{\bm\lambda}$
be the inclusion maps. We can then look at the sequence of maps
\begin{align*}
F_\mu(\mathscr I_{\bm\lambda})
&\to H^{2\rho(\mu)}\Bigl(\overline{\Gr_n^{\bm\lambda}}\cap
(m_n)^{-1}(T_\mu),\,f^!j_*j^*\IC\Bigl(\overline{\Gr_n^{\bm\lambda}},
\underline{\mathbb C}\Bigr)\Bigr)\\
&=H^{2\rho(\mu)}\Bigl(\Gr_n^{\bm\lambda}\cap(m_n)^{-1}(T_\mu),\,
g^!\,\underline{\mathbb C}_{\Gr_n^{\bm\lambda}}
\bigl[\dim\Gr_n^{\bm\lambda}\bigr]\Bigr)\\
&\xrightarrow{\cap[\Gr_n^{\bm\lambda}]}
H^\BM_{2\rho(|\bm\lambda|-\mu)}
\Bigl(\Gr_n^{\bm\lambda}\cap(m_n)^{-1}(T_\mu)\Bigr).
\end{align*}
Here the first two maps carry out the restriction to $\Gr_n^{\bm\lambda}$
(technically, an adjunction followed by a base change) and
the last map is the Alexander duality\footnote{Specifically, the
generalization presented in~\cite{Fulton}, sect.~19.1, equation~(3)
or in~\cite{Iversen}, Theorem~IX.4.7.}.

We claim that these maps are isomorphisms. For the Alexander
duality, this comes from the smoothness of $\Gr_n^{\bm\lambda}$.
For the restriction, consider a stratum
$\Gr_n^{\bm\eta}\subseteq\overline{\Gr_n^{\bm\lambda}}$
with $\bm\eta\neq\bm\lambda$; denoting the inclusion map by $i$
and using the perversity condition, the parity property in
Proposition~\ref{pr:ICnFold}, and the dimension estimate for
$\overline{\Gr_n^{\bm\eta}}\cap(m_n)^{-1}(T_\mu)$, one checks that
$$H^k\Bigl(\overline{\Gr_n^{\bm\lambda}}\cap
(m_n)^{-1}(T_\mu),\,f^!i_*i^!\IC\Bigl(\overline{\Gr_n^{\bm\lambda}},
\underline{\mathbb C}\Bigr)\Bigr)$$
vanishes if $k<2\rho(\mu)+2$; therefore the stratum $\Gr_n^{\bm\eta}$
does not contribute to $F_\mu(\mathscr I_{\bm\lambda})$.

To sum up there is a natural isomorphism
\begin{equation}
\label{eq:DefMVBasis}
F_\mu(\mathscr I_{\bm\lambda})\xrightarrow\simeq
H^\BM_{2\rho(|\bm\lambda|-\mu)}
\Bigl(\Gr_n^{\bm\lambda}\cap(m_n)^{-1}(T_\mu)\Bigr).
\end{equation}
The irreducible components $\Gr_n^{\bm\lambda}\cap(m_n)^{-1}(T_\mu)$
have all dimension $\rho(|\bm\lambda|-\mu)$ and their fundamental
classes provide a basis of the Borel-Moore homology group above.
Gathering these bases for all weights $\mu\in\Lambda$ produces what we
call the MV basis of $V(\bm\lambda)$.

\section{$L$-perfect bases}
\label{se:LPerf}
In this section we consider a general setup, which captures properties
shared by both the MV bases and the dual canonical bases. As before,
$G$ is a connected reductive group over $\mathbb C$ endowed with a Borel
subgroup $B$ and a maximal torus $T\subseteq B$, $\Lambda$ is the character
lattice of $T$, and $\Phi$ and $\Phi^\vee$ are the root and coroot
systems of $(G,T)$. We denote by $\{\alpha_i\mid i\in I\}$ the set of
simple roots defined by $B$ and by $\{\alpha_i^\vee\mid i\in I\}$ the
set of simple coroots. We endow $\Lambda$ with the dominance order $\leq$
and denote the cone of dominant weights by $\Lambda^+$. We regard the Weyl
group $W$ as a subgroup of $\Aut(\Lambda)$; for $i\in I$, we denote by
$s_i$ the simple reflection along the root $\alpha_i$. When needed, we
choose simple root vectors $e_i$ and $f_i$ of weights $\pm\alpha_i$ in
the Lie algebra of $G$ in such a way that $[e_i,f_i]=-\alpha_i^\vee$.

\subsection{Semi-normal crystals}
\label{ss:Crystals}
We start by recalling the following definitions due to Kashiwara
\cite{Kashiwara95}. A semi-normal crystal is a set $B$ endowed with
a map $\wt:B\to\Lambda$ and, for each ``color'' $i\in I$, with a partition
into a collection of finite oriented strings. This latter structure is
recorded by the datum of operators
$$\tilde e_i:B\to B\sqcup\{0\}\quad\text{and}\quad
\tilde f_i:B\to B\sqcup\{0\}$$
which move an element of $b\in B$ upwards and downwards, respectively,
along the string of color~$i$ to which $b$ belongs. The special value
$0$ is assigned to $\tilde e_i(b)$ or $\tilde f_i(b)$ when $b$ is at
the upper or lower end of a string of color $i$. For convenience one
usually further sets $\tilde e_i(0)=\tilde f_i(0)=0$. The position of
$b$ in its string of color $i$ is recorded by functions
$\varepsilon_i$ and $\varphi_i$ defined as follows:
$$\varepsilon_i(b)=\max\bigl\{n\geq0\bigm|\tilde e_i(b)\neq0\bigr\},
\qquad\varphi_i(b)=\max\bigl\{n\geq0\bigm|\tilde f_i(b)\neq0\bigr\}.$$

Two compatibility conditions between the weight map $\wt$ and the datum
of the partitions into oriented strings are required: first, $\wt(b)$
increases by $\alpha_i$ when $b$ moves upwards the string of color $i$ to
which it belongs; second,
$\varphi_i(b)-\varepsilon_i(b)=\langle\alpha_i^\vee,\wt(b)\rangle$
for any $b\in B$ and any $i\in I$. These conditions imply that the
image of a string of color $i$ by the map $\wt$ is stable under the
action of the simple reflection $s_i$. As a consequence, the set
$\{\wt(b)\mid b\in B\}$ is stable under the action of the Weyl group $W$.

The direct sum of two semi-normal crystals $B_1$ and $B_2$ is defined to
be just the disjoint union of the underlying sets. The tensor product
$B_1\otimes B_2$ is defined to be the Cartesian product of the sets
endowed with the maps given in \cite{Kashiwara95}, \S7.3. Notably, for
each color, the strings in $B_1\otimes B_2$ are created from the strings
contained in $B_1$ and in $B_2$ by the process illustrated by the
picture below.

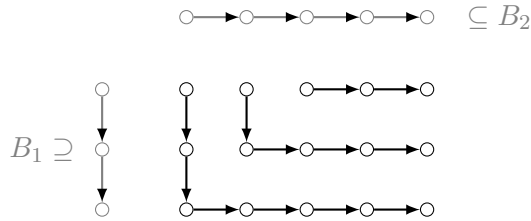
\begin{center}
\begin{tikzpicture}
  [scale=0.8,
   bvertex/.style={circle,draw=black,thin,inner sep=0pt,minimum size=5pt},
   barrow/.style={->,draw=black,thick,>=latex},
   gvertex/.style={circle,draw=black!50,thin,inner sep=0pt,minimum size=5pt},
   garrow/.style={->,draw=black!50,thick,>=latex},
   gtext/.style={black!50}]
  \path (0,0) node (t00) [bvertex] {};
  \path (1,0) node (t10) [bvertex] {};
  \path (2,0) node (t20) [bvertex] {};
  \path (3,0) node (t30) [bvertex] {};
  \path (4,0) node (t40) [bvertex] {};
  \path (0,-1) node (t01) [bvertex] {};
  \path (1,-1) node (t11) [bvertex] {};
  \path (2,-1) node (t21) [bvertex] {};
  \path (3,-1) node (t31) [bvertex] {};
  \path (4,-1) node (t41) [bvertex] {};
  \path (0,-2) node (t02) [bvertex] {};
  \path (1,-2) node (t12) [bvertex] {};
  \path (2,-2) node (t22) [bvertex] {};
  \path (3,-2) node (t32) [bvertex] {};
  \path (4,-2) node (t42) [bvertex] {};
  \path (-1.4,-0) node (tb0) [gvertex] {};
  \path (-1.4,-1) node (tb1) [gvertex] {};
  \path (-1.4,-2) node (tb2) [gvertex] {};
  \path (0,1.2) node (t0b) [gvertex] {};
  \path (1,1.2) node (t1b) [gvertex] {};
  \path (2,1.2) node (t2b) [gvertex] {};
  \path (3,1.2) node (t3b) [gvertex] {};
  \path (4,1.2) node (t4b) [gvertex] {};
  \draw (t00) edge [barrow] (t01);
  \draw (t01) edge [barrow] (t02);
  \draw (t02) edge [barrow] (t12);
  \draw (t12) edge [barrow] (t22);
  \draw (t22) edge [barrow] (t32);
  \draw (t32) edge [barrow] (t42);
  \draw (t10) edge [barrow] (t11);
  \draw (t11) edge [barrow] (t21);
  \draw (t21) edge [barrow] (t31);
  \draw (t31) edge [barrow] (t41);
  \draw (t20) edge [barrow] (t30);
  \draw (t30) edge [barrow] (t40);
  \draw (tb0) edge [garrow] (tb1);
  \draw (tb1) edge [garrow] (tb2);
  \draw (t0b) edge [garrow] (t1b);
  \draw (t1b) edge [garrow] (t2b);
  \draw (t2b) edge [garrow] (t3b);
  \draw (t3b) edge [garrow] (t4b);
  \node at (-2.4,-1) [gtext] {$B_1\supseteq$};
  \node at (5.2,1.2) [gtext] {$\subseteq B_2$};
\end{tikzpicture}
\end{center}

\subsection{$L$-perfect bases}
\label{ss:LPerfBas}
To a subset $J\subseteq I$ we attach the standard Levi subgroup $M_J$
of $G$, the cone
$$\Lambda_J^+=\bigl\{\lambda\in\Lambda\bigm|\forall j\in J,\;
\langle\alpha_j^\vee,\lambda\rangle\geq0\bigr\}$$
of dominant weights for $M_J$, and the $J$-dominance order $\leq^{}_J$
on $\Lambda$ defined by
$$\mu\leq^{}_J\lambda\;\ \Longleftrightarrow\;\ \lambda-\mu\in
\linspan_{\mathbb N}\{\alpha_j\mid j\in J\}.$$
Given $\lambda\in\Lambda_J^+$ we denote by $V_J(\lambda)$ the irreducible
rational representation of $M_J$ with highest weight $\lambda$.
Given a finite sequence $\bm\lambda=(\lambda_1,\ldots,\lambda_n)$ in
$\Lambda_J^+$ we define
$$V_J(\bm\lambda)=V_J(\lambda_1)\otimes\cdots\otimes V_J(\lambda_n).$$

For $J=I$ we recover the conventions previously used by dropping the
decoration $J$ in the notation $\Lambda_J^+$, $\leq^{}_J$ or
$V_J(\bm\lambda)$.

Let $V$ be a rational representation of $G$. With respect to the action
of $M_J$ the space $V$ can be uniquely written as a direct sum of
isotypical components
$$V=\bigoplus_{\mu\in\Lambda_J^+}V_{J,\mu}$$
where $V_{J,\mu}$ is the sum of all subrepresentations of
$\res_{M_J}^G(V)$ isomorphic to $V_J(\mu)$. We define
$$V_{J,\leq^{}_J\mu}=\bigoplus_{\substack{\nu\in\Lambda_J^+\\[2pt]
\nu\leq^{}_J\mu}}V_{J,\nu}.$$

We say that a linear basis $B$ of $V$ is $L$-perfect\footnote{$L$
stands for Levi. This notion of $L$-perfect basis appears unnamed
(and in a dual form) in Braverman and Gaitsgory's paper
(\cite{BravermanGaitsgory01}, sect.~4.3).} if for each $J\subseteq I$
and each $\mu\in\Lambda_J^+$:
\begin{description}
\item[(P1)]
The subspace $V_{J,\leq^{}_J\mu}$ is spanned by a subset of $B$.
\item[(P2)]
The induced basis on the quotient $V_{J,\leq^{}_J\mu}/V_{J,<^{}_J\mu}
\cong V_{J,\mu}$ is compatible with a decomposition of the isotypical
component as a direct sum of irreducible representations.
\end{description}

Taking $J=\varnothing$, we see that an $L$-perfect basis of $V$
consists of weight vectors (note that $\leq_\varnothing$ is the
trivial order on $\Lambda$). Now let $i\in I$, and for each
nonnegative integer $\ell$, define
$$V_{\{i\},\leq\ell}=\bigoplus_{\substack{\mu\in\Lambda\\[2pt]
0\leq\langle\alpha_i^\vee,\mu\rangle\leq\ell}}V_{\{i\},\mu},$$
the sum of all irreducible subrepresentations of
$\res_{M_{\{i\}}}^G(V)$ of dimension at most $\ell+1$.
If $B$ satisfies the conditions (P1) and (P2) for $J=\{i\}$,
then $V_{\{i\},\leq\ell}$ is spanned by $B\cap V_{\{i\},\leq\ell}$
and the induced basis on the quotient
$V_{\{i\},\leq\ell}/V_{\{i\},\leq\ell-1}$ is compatible with a
decomposition as a direct sum of irreducible representations.
Therefore $(B\cap V_{\{i\},\leq\ell})
\setminus(B\cap V_{\{i\},\leq\ell-1})$ decomposes as the disjoint
union of oriented strings of length $\ell$, in such a way that the
simple root vector $e_i$ or $f_i$ acts on a basis vector of
$V_{\{i\},\leq\ell}/V_{\{i\},\leq\ell-1}$ by moving it upwards
or downwards along the string that contains it, up to a scalar.

We can sum up the discussion in the previous paragraph as follows:
if $B$ satisfies (P1) and (P2) for all $J$ of cardinality $\leq1$,
then $B$ is endowed with the structure of crystal and is perfect
in the sense of Berenstein and Kazhdan
(\cite{BerensteinKazhdan}, Definition~5.30).

\begin{lemma}
\label{le:HeredLPerf}
Let $B$ be an $L$-perfect basis of a rational representation $V$ of $G$
and let $B'\subseteq B$. Assume that the linear space $V'$ spanned by $B'$
is a subrepresentation of $V$. Then $B'$ is an $L$-perfect basis of $V'$
and (the image of) $B\setminus B'$ is an $L$-perfect basis of the quotient
$V/V'$.
\end{lemma}
\begin{proof}
Let $J\subseteq I$ and $\mu\in\Lambda_J^+$. Then
$V'_{J,\mu}=V'\cap V_{J,\mu}$ and
$V'_{J,\leq^{}_J\mu}=V'\cap V_{J,\leq^{}_J\mu}$. Now both spaces $V'$
and $V_{J,\leq^{}_J\mu}$ are spanned by a subset of the basis $B$, so
their intersection $V'_{J,\leq^{}_J\mu}$ is spanned by a subset of $B$,
namely $B'\cap V_{J,\leq^{}_J\mu}$.

Let $C$ be the image of
$(B\cap V_{J,\leq^{}_J\mu})\setminus(B\cap V_{J,<^{}_J\mu})$
in the quotient $V_{J,\leq^{}_J\mu}/V_{J,<^{}_J\mu}\cong V_{J,\mu}$.
Then $C$ can be viewed as a basis of $V_{J,\mu}$ and it can be partitioned
into disjoint subsets $C_1$, \dots, $C_n$ so that each $C_k$ spans an
irreducible subrepresentation. By construction, the subspace~$V'_{J,\mu}$
is spanned by a subset $C'\subseteq C$. Each subset $C_k$ can be either
contained in $C'$ or disjoint from~$C'$, depending whether the
subrepresentation that it spans is contained in $V'_{J,\mu}$ or
meets trivially~$V'_{J,\mu}$. Therefore $C'$ is the disjoint union of
some $C_k$; in other words $C'$ is compatible with a decomposition of
$V'_{J,\mu}$ as a direct sum of irreducible subrepresentations.

Thus, $B'$ satisfies both conditions (P1) and (P2), and is therefore an
$L$-perfect basis of $V'$. The proof that $B\setminus B'$ yields an
$L$-perfect basis of the quotient $V/V'$ rests on similar arguments and
is left to the reader.
\end{proof}

Under the assumptions of Lemma~\ref{le:HeredLPerf}, the subset $B'$ is
a subcrystal of $B$; in other words, the crystal structure on $B$ is
the direct sum of the crystal structures on $B'$ and $B\setminus B'$.

\begin{proposition}
\label{pr:UniqCrys}
Let $V$ be a rational representation of $G$. Up to isomorphism, the
crystal of an $L$-perfect basis of $V$ depends only on $V$, and
not on the basis.
\end{proposition}
\begin{proof}
Let $B$ be an $L$-perfect basis of $V$. The conditions imposed on
$B$ with the choice $J=I$ imply the existence in $V$ of a composition
series compatible with $B$. By Lemma~\ref{le:HeredLPerf}, the crystal
$B$ is the direct sum of the crystals of the $L$-perfect bases
induced by $B$ on the subquotients. It thus suffices to prove the
desired uniqueness property in the particular case where $V$ is an
irreducible representation, which in fact is just Theorem~5.37
in~\cite{BerensteinKazhdan}.
\end{proof}
In particular, the crystal of an $L$-perfect basis of an irreducible
representation $V(\lambda)$ is unique. We use henceforth the notation
$B(\lambda)$ for the associated crystal.

\begin{other}{Remark}
The crystal $B(\lambda)$ of an irreducible representation $V(\lambda)$
was introduced by Kashiwara in the context of representations
of quantum groups. The definition via crystallization at $q=0$ and
the definition via the combinatorics of $L$-perfect bases yield the
same crystal; this follows from \cite{Kashiwara93a}, sect.~5.
\end{other}

Fortunately this nice little theory is not empty. As mentioned in the
introduction, any tensor product of irreducible representations has
an $L$-perfect basis, namely its dual canonical basis.
Another example for a $L$-perfect basis: it can be shown, in the case
where $G$ is simply laced, that the dual semicanonical basis of an
irreducible representation is $L$-perfect.

\begin{theorem}
\label{th:MVBasLPerf}
The MV basis of a tensor product of irreducible representations is
$L$-perfect.
\end{theorem}

The end of sect.~\ref{se:LPerf} is devoted to the proof of this result.
The case of an irreducible representation is Proposition~4.1 in
\cite{BravermanGaitsgory01}. The proof for an arbitrary tensor product
follows the same lines; it is only sketched in \textit{loc.~cit.}, and
we add quite a few details to Braverman and Gaitsgory's exposition.

\subsection{Geometric Satake and restriction to a standard Levi subgroup}
\label{ss:SatResLevi}
Consider a subset $J\subseteq I$. In this section we recall Beilinson
and Drinfeld's geometric construction of the restriction functor
$\res_{M_J}^G$ (\cite{BeilinsonDrinfeld}, sect.~5.3).
Additional details can be found in \cite{Lysenko}, sect.~8.6 and
\cite{BaumannRiche}, sect.~1.15.

Define the root and coroot systems
$$\Phi_J=\Phi\cap\linspan_{\mathbb Z}\{\alpha_j\mid j\in J\}
\quad\text{and}\quad\Phi_J^\vee=\Phi^\vee\cap
\linspan_{\mathbb Z}\{\alpha_j^\vee\mid j\in J\}$$
and denote by $\rho_J:\Lambda\to\mathbb Q$ the half-sum of the positive
coroots in $\Phi_J^\vee$. Then $\rho-\rho_J$ vanishes
on $\mathbb Z\Phi_J$ so induces a linear form
$\rho_{I,J}:\Lambda/\mathbb Z\Phi_J\to\mathbb Q$.

To $J$ we also attach the standard Levi subgroup $M_J^\vee$ of $G^\vee$.
Choose a dominant $\theta_J\in\Lambda$ such that
$\langle\alpha_j^\vee,\theta_J\rangle=0$ for each $j\in J$ and
$\langle\alpha_i^\vee,\theta_J\rangle>0$ for each $i\in I\setminus J$.
The embedding
$$\mathbb C^\times\xrightarrow{\theta_J}T^\vee(\mathbb C)\to
G^\vee(\mathcal K)$$
gives rise to an action of $\mathbb C^\times$ on $\Gr$. Then the set
$\Gr_J$ of fixed points under this action can be identified with the
affine Grassmannian for $M_J^\vee$. We denote by $\Perv(\Gr_J)$ the
category of $M_J^\vee(\mathcal O)$-equivariant sheaves on $\Gr_J$
supported on finitely many $M_J^\vee(\mathcal O)$-orbits.

Let $\zeta\in\Lambda/\mathbb Z\Phi_J$ be a coset; then all the points
$L_\mu$ for $\mu\in\zeta$ belong to the same connected component
of $\Gr_J$, which we denote by $\Gr_{J,\zeta}$. The map
$\zeta\mapsto\Gr_{J,\zeta}$ is a bijection from $\Lambda/\mathbb Z\Phi_J$
onto $\pi_0(\Gr_J)$. We denote the stable and unstable sets of
$\Gr_{J,\zeta}$ with respect to the $\mathbb C^\times$-action by
$\Gr_{J,\zeta}^+$ and $\Gr_{J,\zeta}^-$ and form the diagram
\begin{equation}
\label{eq:ParabRetrac}
\begin{split}
\xymatrix@R=20pt@C=10pt{&\Gr&\\\Gr_{J,\zeta}^+\ar[ur]^{s_\zeta^J}
\ar[dr]_{p_{J,\zeta}}&&\Gr_{J,\zeta}^-\ar[ul]_{t_\zeta^J}
\ar[dl]^{q_{J,\zeta}}\\&\Gr_{J,\zeta}&}
\end{split}
\end{equation}
where $s_\zeta^J$ and $t_\zeta^J$ are the inclusion maps and the maps
$p_{J,\zeta}$ and $q_{J,\zeta}$ are defined by
$$p_{J,\zeta}(x)=\lim_{c\to0}\theta_J(c)\cdot x
\quad\text{and}\quad
q_{J,\zeta}(x)=\lim_{c\to\infty}\theta_J(c)\cdot x.$$

Given $\zeta\in\Lambda/\mathbb Z\Phi_J$ and $\mathscr A\in\Perv(\Gr)$,
Beilinson and Drinfeld identify the two sheaves
$$(q_{J,\zeta})_*\,(t_\zeta^J)^!\,\mathscr A\quad\text{and}
\quad(p_{J,\zeta})_!\,(s_\zeta^J)^*\,\mathscr A$$
on $\Gr_{J,\zeta}$ via Braden's hyperbolic localization and show that
they live in perverse degree $2\rho_{I,J}(\zeta)$. Then they define a
functor $r_J^I:\Perv(\Gr)\to\Perv(\Gr_J)$ by
$$r_J^I(\mathscr A)=\bigoplus_{\zeta\in\Lambda/\mathbb Z\Phi_J}
(q_{J,\zeta})_*\,(t_\zeta^J)^!\,\mathscr A[2\rho_{I,J}(\zeta)].$$

For $\mu\in\Lambda$, let $T_{J,\mu}$ be the analog of the unstable
subset $T_\mu$ for the affine Grassmannian $\Gr_J$. Let $\zeta$ be
the coset of $\mu$ modulo $\mathbb Z\Phi_J$ and let
$t_{J,\mu}:T_{J,\mu}\to\Gr_{J,\zeta}$ be the inclusion map.
Using the Iwasawa decomposition, one checks that
$T_\mu\subseteq\Gr_{J,\zeta}^-$ and
\begin{equation}
\label{eq:ParSemiinf}
T_\mu=(q_{J,\zeta})^{-1}(T_{J,\mu}).
\end{equation}
Performing base change in the Cartesian square
$$\xymatrix@R=26pt@C=26pt{T_\mu\ar[r]\ar[r]\ar@/^3ex/[rr]^{t_\mu}\ar[d]
&\Gr_{J,\zeta}^-\ar[r]_{t_\zeta^J}\ar[d]^{q_{J,\zeta}}&\Gr\\
T_{J,\mu}\ar[r]_{t_{J,\mu}}&\Gr_{J,\zeta}&}$$
we obtain, for any sheaf $\mathscr A\in\Perv(\Gr)$, a canonical isomorphism
\begin{equation}
\label{eq:TransitBD}
H^{2\rho(\mu)}\bigl(T_\mu,\,(t_\mu)^!\mathscr A\bigr)\cong
H^{2\rho_J(\mu)}\bigl(T_{J,\mu},\,(t_{J,\mu})^!\;r_J^I(\mathscr A)\bigr).
\end{equation}

For $\mathscr B\in\Perv(\Gr_J)$, define
$$F_{J,\mu}(\mathscr B)=H^{2\rho_J(\mu)}\bigl(T_{J,\mu},\,
(t_{J,\mu})^!\mathscr B\bigr)\quad\text{and}\quad
F_J(\mathscr B)=\bigoplus_{\mu\in\Lambda}F_{J,\mu}(\mathscr B).$$
Then \eqref{eq:TransitBD} can be rewritten as
$F_\mu=F_{J,\mu}\circ r_J^I$. This equality can be refined in the
following statement: the functor $F_J$ induces an equivalence
$\overline{F_J}$ from $\Perv(\Gr_J)$ to the category $\Rep(M_J)$
of finite dimensional rational representations of $M_J$ and the
following diagram commutes.
$$\xymatrix@R=30pt@C=30pt{\Perv(\Gr)\ar[r]^{\overline F}\ar[d]_{r_J^I}
&\Rep(G)\ar[d]^{\res_{M_J}^G}\\\Perv(\Gr_J)\ar[r]^{\overline{F_J}}
&\Rep(M_J)}$$

\subsection{The $J$-decomposition of an MV cycle}
\label{ss:JDecMVCyc}
We fix a subset $J\subseteq I$. We denote by $P_J^{-,\vee}$ the
parabolic subgroup of $G^\vee$ containing $M_J^\vee$ and the
negative root subgroups.

The group $P_J^{-,\vee}(\mathcal K)$ certainly acts on
$\Gr$; it also acts on $\Gr_J$ via the quotient morphism
$P_J^{-,\vee}(\mathcal K)\to M_J^\vee(\mathcal K)$. Given
$\mu\in\Lambda_J^+$, we denote by $\Gr_J^\mu$ the orbit of
$L_\mu$ under the action of $M_J^\vee(\mathcal O)$
(or $P_J^{-,\vee}(\mathcal O)$) on $\Gr_J$. Noting that
$$\lim_{a\to\infty}\theta_J(a)\,g\,\theta_J(a)^{-1}=1$$
for all $g$ in the unipotent radical of $P_J^{-,\vee}$,
we see that for any $\zeta\in\Lambda/\mathbb Z\Phi_J$, the
connected component $\Gr_{J,\zeta}$ of $\Gr_J$ and the unstable
subset $\Gr_{J,\zeta}^-$ in $\Gr$ are stable under the action
of $P_J^{-,\vee}(\mathcal O)$ and that the map $q_{J,\zeta}$ is
equivariant.

Let $\bm\lambda\in(\Lambda^+)^n$, let $\mu\in\Lambda_J^+$ and let
$\zeta$ be the coset of $\mu$ modulo $\mathbb Z\Phi_J$. We consider
the following diagram.
$$\xymatrix@R=26pt@!C=60pt{(m_n)^{-1}\bigl(\Gr_{J,\zeta}^-\bigr)
\ar[r]^-{m_n}\ar@{^(->}[d]&\Gr_{J,\zeta}^-\ar[r]^{q_{J,\zeta}}
\ar[d]^{t_\zeta^J}&\Gr_{J,\zeta}\\\Gr_n\ar[r]^{m_n}&\Gr&}$$
The group $G^\vee(\mathcal K)$ acts on $\Gr_n$ by left multiplication
on the first factor and the action of the subgroup $G^\vee(\mathcal O)$
leaves $\overline{\Gr_n^{\bm\lambda}}$ stable. Let $H$ be the stabilizer
of $L_\mu$ with respect to the action of $P_J^{-,\vee}(\mathcal O)$
on $\Gr_J$; it acts on
$E=\overline{\Gr_n^{\bm\lambda}}\cap(q_{J,\zeta}\circ m_n)^{-1}(L_\mu)$.
Since $q_{J,\zeta}\circ m_n$ is equivariant under the action of
$P_J^{-,\vee}(\mathcal O)$, we can make the identification
$$\xymatrix@C=36pt{P_J^{-,\vee}(\mathcal O)\times^HE\ar[r]^-\cong
\ar[d]&\overline{\Gr_n^{\bm\lambda}}\cap(q_{J,\zeta}\circ m_n)^{-1}
(\Gr_J^\mu)\ar[d]^{q_{J,\zeta}\circ m_n}\\P_J^{-,\vee}(\mathcal O)/H
\ar[r]^-\cong&\Gr_J^\mu}$$
where the left vertical arrow is the projection along the first
factor. We thereby see that the right vertical arrow is a locally
trivial fibration.

In particular, all the fibers $\overline{\Gr_n^{\bm\lambda}}\cap
(q_{J,\zeta}\circ m_n)^{-1}(x)$ with $x\in\Gr_J^\mu$ are isomorphic
varieties. Remembering that $(q_{J,\zeta})^{-1}(L_\mu)\subseteq T_\mu$,
we find the following bound for their dimension:
$$\dim\Bigl(\overline{\Gr_n^{\bm\lambda}}\cap(q_{J,\zeta}\circ
m_n)^{-1}(x)\Bigr)=\dim E\leq\dim\Bigl(\overline{\Gr_n^{\bm\lambda}}
\cap(m_n)^{-1}(T_\mu)\Bigr)=\rho(|\bm\lambda|-\mu).$$
Therefore
\begin{equation}
\label{eq:ParabDimEst}
\dim\Bigl(\overline{\Gr_n^{\bm\lambda}}\cap(q_{J,\zeta}\circ
m_n)^{-1}(\Gr_J^\mu)\Bigr)\leq\dim\Gr_J^\mu+\rho(|\bm\lambda|-\mu)
=2\rho_J(\mu)+\rho(|\bm\lambda|-\mu).
\end{equation}
Since $\Gr_J^\mu$ is connected and simply-connected, the fibration
induces a bijection between the set of irreducible components of
$\overline{\Gr_n^{\bm\lambda}}\cap(q_{J,\zeta}\circ m_n)^{-1}(\Gr_J^\mu)$
and the set of irreducible components of any fiber
$\overline{\Gr_n^{\bm\lambda}}\cap(q_{J,\zeta}\circ m_n)^{-1}(x)$.

We define
$$\mathscr Z^J(\bm\lambda)_\mu=\Bigl\{Z\in\Irr\Bigl(
\overline{\Gr_n^{\bm\lambda}}\cap(q_{J,\zeta}\circ
m_n)^{-1}(\Gr_J^\mu)\Bigr)\Bigm|\dim Z=2\rho_J(\mu)+
\rho(|\bm\lambda|-\mu)\Bigr\}.$$
For $\nu\in\Lambda$, we define
$$\mathscr Z_J(\mu)_\nu=\Irr\Bigl(\overline{\Gr_J^\mu}\cap
T_{J,\nu}\Bigr);$$
as we saw in sect.~\ref{ss:MVCyc}, the map $Z\mapsto Z\cap\Gr_J^\mu$
is a bijection from $\mathscr Z_J(\mu)_\nu$ onto the set of
irreducible components of $\Gr_J^\mu\cap T_{J,\nu}$.

Fix now $\bm\lambda\in(\Lambda^+)^n$ and $\nu\in\Lambda$. Following
Braverman and Gaitsgory's method, we define a bijection
$$\mathscr Z(\bm\lambda)_\nu\cong\bigsqcup_{\mu\in\Lambda_J^+}
\mathscr Z^J(\bm\lambda)_\mu\times\mathscr Z_J(\mu)_\nu.$$
The union above is in fact restricted to those weights $\mu$ such that
$\mu-\nu\in\mathbb Z\Phi_J$, for otherwise $\mathscr Z_J(\mu)_\nu$
is empty. Let $\zeta$ denote the coset of $\nu$ modulo $\mathbb Z\Phi_J$.

First choose $\mu\in\Lambda_J^+\cap\zeta$ and a pair
$(Z^J,Z_J)\in\mathscr Z^J(\bm\lambda)_\mu\times\mathscr Z_J(\mu)_\nu$.
Using \eqref{eq:ParSemiinf} and the fibration above, we see that
$Z^J\cap(q_{J,\zeta}\circ m_n)^{-1}(Z_J\cap\Gr_J^\mu)$ is an irreducible
subset of
$$\overline{\Gr_n^{\bm\lambda}}\cap(q_{J,\zeta}\circ m_n)^{-1}
(T_{J,\nu})=\overline{\Gr_n^{\bm\lambda}}\cap(m_n)^{-1}(T_\nu)$$
of dimension
$$\dim Z^J-\dim\Gr_J^\mu+\dim Z_J=\rho(|\bm\lambda|-\mu)+\rho_J(\mu-\nu)
=\rho(|\bm\lambda|-\nu).$$
Therefore there is a unique $Z\in\mathscr Z(\bm\lambda)_\nu$ that
contains $Z^J\cap(q_{J,\zeta}\circ m_n)^{-1}(Z_J\cap\Gr_J^\mu)$ as
a dense subset.

Conversely, start from $Z\in\mathscr Z(\bm\lambda)_\nu$.
Then $Z\subseteq T_\nu\subseteq\Gr_{J,\zeta}^-$. We can thus partition
$Z$ into locally closed subsets as follows.
$$Z=\bigsqcup_{\mu\in\Lambda_J^+\cap\zeta}\Bigl(Z\cap
(q_{J,\zeta}\circ m_n)^{-1}(\Gr_J^\mu)\Bigr)$$
Since $Z$ is irreducible, there is a unique $\mu\in\Lambda_J^+\cap\zeta$
such that $Z\cap(q_{J,\zeta}\circ m_n)^{-1}(\Gr_J^\mu)$ is open
dense in $Z$. That subset is certainly irreducible, hence
contained in an irreducible component $Z^J$ of
$\overline{\Gr_n^{\bm\lambda}}\cap(q_{J,\zeta}\circ
m_n)^{-1}(\Gr_J^\mu)$. Also, $q_{J,\zeta}\circ m_n$ maps
$Z\cap(q_{J,\zeta}\circ m_n)^{-1}(\Gr_J^\mu)$ to an irreducible
subset of $\Gr_J^\mu\cap T_{J,\nu}$, which in turn is contained
in an irreducible component $Z_J\in\mathscr Z_J(\mu)_\nu$. Then
$$Z\cap(q_{J,\zeta}\circ m_n)^{-1}(\Gr_J^\mu)\subseteq
Z^J\cap(q_{J,\zeta}\circ m_n)^{-1}(Z_J).$$
The left-hand side has dimension $\rho(|\bm\lambda|-\nu)$ and the
right-hand side has dimension
$$\dim Z^J-2\rho_J(\mu)+\dim Z_J=\dim Z^J-2\rho_J(\mu)+\rho(\mu-\nu);$$
combining with the bound \eqref{eq:ParabDimEst} we get
$$\dim Z^J=\rho(|\bm\lambda|-\mu)+2\rho_J(\mu)$$
and therefore $Z^J\in\mathscr Z^J(\bm\lambda)_\mu$.

These two constructions define mutually inverse bijections; in particular,
$$Z^J\cap(q_{J,\zeta}\circ m_n)^{-1}(Z_J\cap\Gr_J^\mu)
=Z\cap(q_{J,\zeta}\circ m_n)^{-1}(\Gr_J^\mu).$$

We record that to each MV cycle $Z\in\mathscr Z(\bm\lambda)_\nu$ is
assigned a weight $\mu\in\Lambda_J^+$ characterized by the conditions
$$(q_{J,\zeta}\circ m_n)(Z)\subseteq\overline{\Gr_J^\mu}\quad\text{and}
\quad(q_{J,\zeta}\circ m_n)(Z)\cap\Gr_J^\mu\neq\varnothing;$$
in the sequel this weight will be denoted by $\mu_J(Z)$.

\subsection{MV bases are $L$-perfect}
\label{ss:MVLPerf}
We now give the proof of Theorem~\ref{th:MVBasLPerf}, properly
speaking. We fix a positive integer $n$ and a tuple
$\bm\lambda\in(\Lambda^+)^n$. We need two ingredients besides the
constructions explained in sects.~\ref{ss:MVBas}
and~\ref{ss:SatResLevi}.

(A) Take $\mathscr A\in\Perv(\Gr)$ and write the sheaf
$$\mathscr B=r_J^I(\mathscr A)=
\bigoplus_{\zeta\in\Lambda/\mathbb Z\Phi_J}(q_{J,\zeta})_*\,
(t_\zeta^J)^!\,\mathscr A[2\rho_{I,J}(\zeta)]$$
in $\Perv(\Gr_J)$ as a direct sum of isotypical components
\begin{equation}
\label{eq:DecResSheaf}
\mathscr B=\bigoplus_{\mu\in\Lambda_J^+}\IC\Bigl(\overline{\Gr_J^\mu},\,
\mathscr L_\mu\Bigr).
\end{equation}
The local systems $\mathscr L_\mu$ on $\Gr_J^\mu$ that appear
in~\eqref{eq:DecResSheaf} can be expressed as
$\mathscr L_\mu=\mathscr H^kh^!\mathscr B$
where $h:\Gr_J^\mu\to\Gr_J$ is the inclusion map and
$k=-\dim\Gr_J^\mu=-2\rho_J(\mu)$. With $e:\{x\}\to\Gr_J^\mu$ the
inclusion of a point and $\zeta$ the coset of $\mu$ modulo
$\mathbb Z\Phi_J$, the fiber of $\mathscr L_\mu$ is
$$(\mathscr L_\mu)_x\cong e^!\mathscr L_\mu\bigl[2\dim\Gr_J^\mu\bigr]
\cong H^{2\rho_J(\mu)}\bigl(\{x\},\,e^!h^!\mathscr B\bigr).$$

For the specific case
$$\mathscr A=\mathscr I_{\bm\lambda}=(m_n)_*\IC\Bigl(
\overline{\Gr_n^{\bm\lambda}},\,\underline{\mathbb C}\Bigr),$$
noting the equality $\rho_J(\mu)+\rho_{I,J}(\zeta)=\rho(\mu)$, we get
$$(\mathscr L_\mu)_x\cong H^{2\rho(\mu)}
\Bigl((q_{J,\zeta}\circ m_n)^{-1}(x),\,f^!
\IC\Bigl(\overline{\Gr_n^{\bm\lambda}},\,
\underline{\mathbb C}\Bigr)\Bigr)$$
where $f$ is the injection map depicted in the Cartesian diagram below.
$$\xymatrix@R=22pt@!C=40pt{(q_{J,\zeta}\circ m_n)^{-1}(x)\ar[rrr]^f
\ar[d]&&&\Gr_n\ar[d]^{m_n}\\(q_{J,\zeta})^{-1}(x)\ar[rr]\ar[d]&&
\Gr_{J,\zeta}^-\ar[r]^{t_\zeta^J}\ar[d]^{q_{J,\zeta}}&\Gr\\\{x\}
\ar[r]^e&\Gr_J^\mu\ar[r]^h&\Gr_{J,\zeta}&}$$
The same reasoning as in sect.~\ref{ss:MVBas} proves that only the
stratum $\Gr_n^{\bm\lambda}$ contributes to this cohomology group.
Denoting by $g:\Gr_n^{\bm\lambda}\cap(q_{J,\zeta}\circ m_n)^{-1}(x)\to
\Gr_n^{\bm\lambda}$ the inclusion map, this observation leads to
an isomorphism
\begin{multline*}
(\mathscr L_\mu)_x
=H^{2\rho_J(\mu)+2\rho_{I,J}(\zeta)}
\Bigl(\Gr_n^{\bm\lambda}\cap(q_{J,\zeta}\circ m_n)^{-1}(x),\,
g^!\,\underline{\mathbb C}_{\Gr_n^{\bm\lambda}}
\bigl[\dim\Gr_n^{\bm\lambda}\bigr]\Bigr)\\
\xrightarrow{\cap[\Gr_n^{\bm\lambda}]}
H^\BM_{2\rho(|\bm\lambda|-\mu)}\Bigl(\Gr_n^{\bm\lambda}\cap
(q_{J,\zeta}\circ m_n)^{-1}(x)\Bigr).
\end{multline*}
Thus the local systems $\mathscr L_\mu$ appearing in
\eqref{eq:DecResSheaf} have a natural basis, namely the set
$\mathscr Z^J(\bm\lambda)_\mu$.

We record the following consequence of this discussion: given
$(\bm\lambda,\mu,\nu)\in(\Lambda^+)^n\times\Lambda_J^+\times\Lambda$
such that $\mu-\nu\in\mathbb Z\Phi_J$, we have
\begin{align}
\notag
\dim H^{2\rho_J(\nu)}\Bigl(T_{J,\nu},\,(t_{J,\nu})^!
\IC\Bigl(\overline{\Gr_J^\mu},\mathscr L_\mu\Bigr)\Bigr)
&=\rank\mathscr L_\mu\times
\dim H^{2\rho_J(\nu)}\Bigl(T_{J,\nu},\,(t_{J,\nu})^!
\IC\Bigl(\overline{\Gr_J^\mu},\underline{\mathbb C}\Bigr)\Bigr)\\
\notag
&=\Card\mathscr Z^J(\bm\lambda)_\mu\times\Card\mathscr Z_J(\mu)_\nu\\
\label{eq:CountMVCyc}
&=\Card\bigl\{Z\in\mathscr Z(\bm\lambda)_\nu\bigm|\mu_J(Z)=\mu\bigr\}.
\end{align}

(B) Now let us start with a sheaf $\mathscr B$ in $\Perv(\Gr_J)$
and a weight $\mu\in\Lambda_J^+$. Let us denote by
$i:\overline{\Gr_J^\mu}\to\Gr_J$ the inclusion map.
By \cite{BeilinsonBernsteinDeligneGabber}, Amplification~1.4.17.1,
the largest subobject of $\mathscr B$ in $\Perv(\Gr_J)$ supported
on $\overline{\Gr_J^\mu}$ is
$\mathscr B_{\leq^{}_J\mu}={}^p\tau_{\leq0}\;i_*i^!\,\mathscr B$,
where ${}^p\tau_{\leq0}$ is the truncation functor for the perverse
$t$-structure. From the distinguished triangle
$${}^p\tau_{\leq0}\;i_*i^!\,\mathscr B\to i_*i^!\,\mathscr B\to
{}^p\tau_{>0}\;i_*i^!\,\mathscr B\xrightarrow+$$
in the bounded derived category of constructible sheaves on $\Gr_J$,
we deduce the long exact sequence
\begin{multline*}
H^{2\rho_J(\nu)-1}\bigl(T_{J,\nu},\,
(t_{J,\nu})^!\,{}^p\tau_{>0}\;i_*i^!\,\mathscr B\bigr)\to
H^{2\rho_J(\nu)}\bigl(T_{J,\nu},\,
(t_{J,\nu})^!\,{}^p\tau_{\leq0}\;i_*i^!\,\mathscr B\bigr)\\[2pt]
\to H^{2\rho_J(\nu)}\bigl(T_{J,\nu},\,(t_{J,\nu})^!\,i_*i^!\,
\mathscr B\bigr)\to
H^{2\rho_J(\nu)}\bigl(T_{J,\nu},\,
(t_{J,\nu})^!\,{}^p\tau_{>0}\;i_*i^!\,\mathscr B\bigr).
\end{multline*}
Theorem~3.5 in \cite{MirkovicVilonen} implies that the two extrem
terms vanish, and therefore
$$F_{J,\nu}\bigl(\mathscr B_{\leq^{}_J\mu}\bigr)=H^{2\rho_J(\nu)}
\bigl(T_{J,\nu},\,(t_{J,\nu})^!\,i_*i^!\,\mathscr B\bigr).$$

Let us patch all these pieces together. We take
$(\bm\lambda,\mu,\nu)\in(\Lambda^+)^n\times\Lambda_J^+\times\Lambda$
such that $\mu$ and $\nu$ belong to the same coset
$\zeta\in Z/\mathbb Z\Phi_J$ and we consider
$$\mathscr I_{\bm\lambda}=(m_n)_*\IC\Bigl(
\overline{\Gr_n^{\bm\lambda}},\underline{\mathbb C}\Bigr)
\quad\text{and}\quad\mathscr B=r_J^I(\mathscr I_{\bm\lambda}).$$
Composing the isomorphisms given in \eqref{eq:TransitBD} and
\eqref{eq:DefMVBasis}, we get
\begin{equation}
\label{eq:MVBasLPerf1}
H^{2\rho_J(\nu)}\bigl(T_{J,\nu},\,(t_{J,\nu})^!\;\mathscr B\bigr)\cong
H^{2\rho(\nu)}\bigl(T_\nu,\,(t_\nu)^!\mathscr I_{\bm\lambda}\bigr)
\cong H^\BM_{2\rho(|\bm\lambda|-\nu)}
\Bigl(\Gr_n^{\bm\lambda}\cap(m_n)^{-1}(T_\nu)\Bigr).
\end{equation}

To save place we set
$S=(q_{J,\zeta})^{-1}\bigl(\overline{\Gr_J^\mu}\bigr)\cap T_\nu$
and denote by $s:S\to T_\nu$ the inclusion map. Chasing in the
three-dimensional figure
\begin{equation}
\label{eq:3DChase}
\begin{split}
\kern-8pt
\xymatrix@!C=44pt@R=14pt{&&&\Gr&\Gr_n\ar[l]_{m_n}&\Gr_n^{\bm\lambda}
\ar[l]\\&&&&&\\&\Gr_{J,\zeta}&&\Gr_{J,\zeta}^-\ar[ll]_{q_{J,\zeta}}
\ar[uu]^{t_\zeta^J}&&\Gr_n^{\bm\lambda}\cap(m_n)^{-1}
\bigl(\Gr_{J,\zeta}^-\bigr)\ar[ll]\ar[uu]\\\overline{\Gr_J^\mu}
\ar[ur]^i&&(q_{J,\zeta})^{-1}\bigl(\overline{\Gr_J^\mu}\bigr)\ar[ll]
\ar[ur]&&\Gr_n^{\bm\lambda}\cap(q_{J,\zeta}\circ m_n)^{-1}
\bigl(\overline{\Gr_J^\mu}\bigr)\ar[ll]\ar[ur]&\\&T_{J,\nu}\ar'[u][uu]
&&T_\nu\ar'[l][ll]\ar'[u][uu]&&\Gr_n^{\bm\lambda}\cap(m_n)^{-1}(T_\nu)
\ar'[l][ll]\ar[uu]\\\overline{\Gr_J^\mu}\cap T_{J,\nu}\ar[uu]\ar[ur]&&
S\ar[ll]\ar[uu]\ar[ur]^(.44)s&&\Gr_n^{\bm\lambda}\cap(m_n)^{-1}(S)
\ar[ll]\ar[uu]\ar[ur]&}
\end{split}
\end{equation}
we complete \eqref{eq:MVBasLPerf1} in the following commutative diagram.
$$\xymatrix@C=26pt{H^{2\rho_J(\nu)}\bigl(T_{J,\nu},\,
(t_{J,\nu})^!\,i_*i^!\,\mathscr B\bigr)\ar[r]^-\simeq\ar[d]
&H^{2\rho(\nu)}\bigl(S,\,(t_\nu s)^!\mathscr I_{\bm\lambda}\bigr)
\ar[r]^-\simeq\ar[d]
&H^\BM_{2\rho(|\bm\lambda|-\nu)}\Bigl(\Gr_n^{\bm\lambda}
\cap(m_n)^{-1}(S)\Bigr)\ar[d]\\
H^{2\rho_J(\nu)}\bigl(T_{J,\nu},\,
(t_{J,\nu})^!\,\mathscr B\bigr)\ar[r]^-\simeq
&H^{2\rho(\nu)}\bigl(T_\nu,\,(t_\nu)^!\mathscr I_{\bm\lambda}\bigr)
\ar[r]^-\simeq
&H^\BM_{2\rho(|\bm\lambda|-\nu)}\Bigl(\Gr_n^{\bm\lambda}
\cap(m_n)^{-1}(T_\nu)\Bigr)}$$
As explained in (B), the left vertical arrow of this diagram is the
inclusion map
$$F_{J,\nu}(\mathscr B_{\leq^{}_J\mu})\to F_{J,\nu}(\mathscr B).$$

If an MV cycle $Z\in\mathscr Z(\bm\lambda)_\nu$ satisfies
$\mu_J(Z)\leq^{}_J\mu$, then it is contained in $(m_n)^{-1}(S)$,
so the fundamental class of $Z\cap\Gr_n^{\bm\lambda}$ belongs to
$F_{J,\nu}(\mathscr B_{\leq^{}_J\mu})$.
Looking at equation \eqref{eq:CountMVCyc}, we see that there are
just enough such MV cycles to span this subspace. Going through the
geometric Satake correspondence, we conclude that the MV basis of
$V(\bm\lambda)$ satisfies the condition (P1) for being $L$-perfect.

Eying now to the condition (P2), we consider the diagram below,
consisting of inclusion maps.
$$\xymatrix{\overline{\Gr_J^\mu}\setminus\Gr_J^\mu
\ar[r]^-f&\overline{\Gr_J^\mu}\ar[d]_i&\Gr_J^\mu
\ar[l]_-g\ar[ld]^-h\\&\Gr_{J,\zeta}&}$$
For any sheaf $\mathscr F\in\Perv(\Gr_J)$ supported on
$\overline{\Gr_J^\mu}\setminus\Gr_J^\mu$, we have
$$\Hom\bigl(\mathscr F,\,{}^p\tau_{\leq0}\bigl(h_*h^!\,
\mathscr B\bigr)\bigr)
=\Hom\bigl(\mathscr F,\,h_*h^!\,\mathscr B\bigr)
=\Hom\bigl(h^*\,\mathscr F,\,h^!\,\mathscr B\bigr)=0$$
in the bounded derived category of constructible sheaves over $\Gr_J$
(the first two equalities by adjunction, the last one
because $h^*\mathscr F=0$). Since $h_*h^!\,\mathscr B$ is
concentrated in nonnegative perverse degrees
(\cite{BeilinsonBernsteinDeligneGabber}, Proposition~1.4.16),
the sheaf ${}^p\tau_{\leq0}\bigl(h_*h^!\mathscr B\bigr)$ is
perverse, and from the semisimplicity of $\Perv(\Gr_J)$ we
conclude that
$$\Hom\bigl(\,{}^p\tau_{\leq0}\bigl(h_*h^!\,
\mathscr B\bigr),\mathscr F\bigr)=0.$$
Again, in the distinguished triangle
$$i_*f_*f^!i^!\,\mathscr B\to i_*i^!\,\mathscr B\to i_*g_*g^!i^!\,
\mathscr B\xrightarrow+$$
all sheaves are concentrated in nonnegative perverse degrees.
Denoting by ${}^pH^1$ the first homology group for the perverse
$t$-structure, we obtain the exact sequence
$$0\to{}^p\tau_{\leq0}\bigl(i_*f_*f^!i^!\,\mathscr B\bigr)
\to{}^p\tau_{\leq0}\bigl(i_*i^!\,\mathscr B\bigr)
\to{}^p\tau_{\leq0}\bigl(h_*h^!\,\mathscr B\bigr)
\to{}^pH^1\bigl(i_*f_*f^!i^!\,\mathscr B\bigr).$$
The perverse sheaf on the right is supported on
$\overline{\Gr_J^\mu}\setminus\Gr_J^\mu$, so the right arrow is zero
by the previous step. The resulting short exact sequence can be
identified with
$$0\to\mathscr B_{<^{}_J\mu}\to\mathscr B_{\leq^{}_J\mu}
\to\mathscr B_{\leq^{}_J\mu}/\mathscr B_{<^{}_J\mu}\to0.$$
With the same arguments as in the point (B) above, we deduce that
$$F_{J,\nu}\bigl(\mathscr B_{\leq^{}_J\mu}/\mathscr B_{<^{}_J\mu}\bigr)
=F_{J,\nu}\bigl({}^p\tau_{\leq0}\;h_*h^!\,\mathscr B\bigr)
=H^{2\rho_J(\nu)}\bigl(T_{J,\nu},\,
(t_{J,\nu})^!\,h_*h^!\,\mathscr B\bigr).$$

In \eqref{eq:3DChase}, we replace $\overline{\Gr_J^\mu}$ by
$\Gr_J^\mu$; the same chasing as before now leads to the isomorphism
\begin{equation}
\label{eq:MVBasLPerf2}
H^{2\rho_J(\nu)}\bigl(T_{J,\nu},\,(t_{J,\nu})^!\,h_*h^!\,
\mathscr B\bigr)\xrightarrow\simeq H^\BM_{2\rho(|\bm\lambda|-\nu)}
\Bigl(\Gr_n^{\bm\lambda}\cap(q_{J,\zeta}\circ m_n)^{-1}(\Gr_J^\mu)
\cap(m_n)^{-1}(T_\nu)\Bigr).
\end{equation}
Now the point (A) at the beginning of this section explains that
$\mathscr H^kh^!\,\mathscr B$, where $k=-2\rho_J(\mu)$, is the local
system $\mathscr L_\mu$ and that it comes with a natural basis, namely
$\mathscr Z^J(\bm\lambda)_\mu$. This basis induces a decomposition of
$\mathscr B_{\leq^{}_J\mu}/\mathscr B_{<^{}_J\mu}$ into a sum of
simple objects in $\Perv(\Gr_J)$. On the one hand, this decomposition
can be followed through the geometric Satake correspondence, where
it gives a decomposition of the subquotient of the isotypical
filtration of $\res_{M_J}^GV(\bm\lambda)$ into a direct sum of
irreducible representations. On the other hand, it can also be
tracked through the isomorphism~\eqref{eq:MVBasLPerf2}:
\begin{equation}
\label{eq:MVBasLPerf3}
F_{J,\nu}\bigl(\mathscr B_{\leq^{}_J\mu}/\mathscr B_{<^{}_J\mu}\bigr)
\cong\bigoplus_{Y\in\mathscr Z^J(\bm\lambda)_\mu}
H^\BM_{2\rho(|\bm\lambda|-\nu)}
\Bigl(\Gr_n^{\bm\lambda}\cap Y\cap(m_n)^{-1}(T_\nu)\Bigr).
\end{equation}
From sect.~\ref{ss:JDecMVCyc}, we see that the irreducible components of
$$\overline{\Gr_n^{\bm\lambda}}\cap(q_{J,\zeta}\circ m_n)^{-1}(\Gr_J^\mu)
\cap(m_n)^{-1}(T_\nu)$$
of dimension $\rho(|\bm\lambda|-\nu)$ are the cycles
$Z^J\cap(q_{J,\zeta}\circ m_n)^{-1}(Z_J\cap\Gr_J^\mu)$, with
$(Z^J,Z_J)\in\mathscr Z^J(\bm\lambda)_\mu\times\mathscr Z_J(\mu)_\nu$.
The basis of the right-hand side of~\eqref{eq:MVBasLPerf2} afforded
by the fundamental classes of these irreducible components is thus
compatible with the decomposition~\eqref{eq:MVBasLPerf3}. Therefore,
the MV basis of $V(\bm\lambda)$ satisfies the condition (P2) for being
$L$-perfect.

The proof of Theorem~\ref{th:MVBasLPerf} is now complete.

\begin{other}{Remark}
\label{rk:StringP2}
The proof establishes that the MV basis of $V(\bm\lambda)$ satisfies a
stronger property than (P2): there exists an isomorphism of the isotypical
component $V(\bm\lambda)_{J,\mu}$ with a direct sum of copies of the
irreducible representation $V_J(\mu)$ such that the induced basis on
$V(\bm\lambda)_{J,\mu}$ matches the direct sum of the MV bases of the
summands.
\end{other}

\subsection{Crystal structure on MV cycles}
\label{ss:CrysMV}
Let $\bm\lambda\in(\Lambda^+)^n$. The MV basis of $V(\bm\lambda)$
defined in sect.~\ref{ss:MVBas} is indexed by
$$\mathscr Z(\bm\lambda)=\bigsqcup_{\nu\in\Lambda}\mathscr Z(\bm\lambda)_\nu$$
and is $L$-perfect. Thus, the set $\mathscr Z(\bm\lambda)$ is endowed
with the structure of a crystal, as explained in sect.~\ref{ss:LPerfBas}.
Obviously the weight of an MV cycle $Z\in\mathscr Z(\bm\lambda)_\nu$
is simply $\wt(Z)=\nu$. The aim of this section is to characterize the
action on $\mathscr Z(\bm\lambda)$ of the operators $\tilde e_i$ and
$\tilde f_i$.

In semisimple rank~$1$, one can give an explicit analytical description
of the MV cycles, as follows.
\begin{proposition}
\label{pr:MVRk1}
Assume that $G$ has semisimple rank $1$ and denote by $\alpha$ and
$\alpha^\vee$ the positive root and coroot. Let $y:\mathbb G_a\to G^\vee$
be the additive one-parameter subgroup for the root $-\alpha^\vee$.
Let $(\mu,\nu)\in\Lambda^+\times\Lambda$ and set
$r=\langle\alpha^\vee,\mu\rangle$. Then $\overline{\Gr^\mu}\cap T_\nu$
is nonempty if and only if there exists $p\in\{0,1,\ldots,r\}$ such
that $\nu=\mu-p\alpha$; in this case, the map
$a\mapsto y(az^{p-r})\,L_\nu$ induces an isomorphism of algebraic varieties
$$\mathcal O/z^p\mathcal O\xrightarrow\simeq\overline{\Gr^\mu}\cap T_\nu$$
so $\overline{\Gr^\mu}\cap T_\nu$ is an affine space of dimension $p$
and $\mathscr Z(\mu)_\nu$ is a singleton.
\end{proposition}

We skip the proof since this proposition is well-known; compare for
instance with~\cite{BaumannGaussent}, Proposition~3.10. We can now
describe the crystal structure on $\mathscr Z(\bm\lambda)$, which
extends \cite{BaumannGaussent}, Proposition~4.2.

\begin{proposition}
\label{pr:CharacCrysMV}
Let $(\bm\lambda,\nu)\in(\Lambda^+)^n\times\Lambda$, let $i\in I$ and
let $Z\in\mathscr Z(\bm\lambda)_\nu$.
\begin{enumerate}
\item
\label{it:PrCCMVa}
We have $\wt(Z)=\nu$, $\varepsilon_i(Z)=\frac12\bigl\langle
\alpha_i^\vee,\mu_{\{i\}}(Z)-\nu\bigr\rangle$ and\/
$\varphi_i(Z)=\frac12\bigl\langle\alpha_i^\vee,
\mu_{\{i\}}(Z)+\nu\bigr\rangle$.
\item
\label{it:PrCCMVb}
Let $Y\in\mathscr Z(\bm\lambda)_{\nu+\alpha_i}$.
Then $Y=\tilde e_iZ$ if and only if\/ $Y\subseteq\overline Z$
and $\mu_{\{i\}}(Y)=\mu_{\{i\}}(Z)$.
\end{enumerate}
\end{proposition}
\begin{proof}
Let $\bm\lambda$, $\nu$, $i$, $Z$ as in the statement and set
$\mu=\mu_{\{i\}}(Z)$. By definition, the MV cycles $\tilde e_iZ$ and
$\tilde f_iZ$ (if nonzero) are obtained by letting the Chevalley
generators $e_i$ and $f_i$ act on (the basis element indexed by) $Z$
in the appropriate subquotient of the isotypical filtration of
$\res_{M_{\{i\}}}^GV(\bm\lambda)$. According to \eqref{eq:MVBasLPerf3},
this entails that
$$\mu_{\{i\}}(Z)=\mu_{\{i\}}(\tilde e_iZ)=\mu_{\{i\}}(\tilde f_iZ)
\quad\text{and}\quad
Z^{\{i\}}=(\tilde e_iZ)^{\{i\}}=(\tilde f_iZ)^{\{i\}}.$$
In addition, $\mathscr Z_{\{i\}}(\mu)_{\nu+\alpha_i}$ and
$\mathscr Z_{\{i\}}(\mu)_{\nu-\alpha_i}$ are empty or singletons,
and the MV cycles $(\tilde e_iZ)_{\{i\}}$ and $(\tilde f_iZ)_{\{i\}}$
in the affine Grassmannian $\Gr_{\{i\}}$ are uniquely determined by
weight considerations.

The statements can be deduced from these facts by using the explicit
description provided by Proposition~\ref{pr:MVRk1} and the construction
of the map $(Z^J,Z_J)\mapsto Z$ in sect.~\ref{ss:JDecMVCyc}.
\end{proof}

\section{The path model and MV cycles}
\label{se:PathsAndMV}
In the previous section, we defined the structure of a crystal
on the set $\mathscr Z(\bm\lambda)$. In this section, we turn to
Littelmann's path model \cite{Littelmann95} to study this structure.
This combinatorial device can be used to effectively assemble
MV cycles. Our construction is inspired by the results presented
in~\cite{GaussentLittelmann} but is more flexible, for it relaxes
the restriction to minimal galleries.

In this paper, piecewise linear means continuous piecewise linear.
We keep the notation set up in the header of sect.~\ref{se:LPerf}.

\subsection{Recollections on the path model}
\label{ss:RecPathMod}
Let $\Lambda_{\mathbb R}=\Lambda\otimes_{\mathbb Z}\mathbb R$
be the real space spanned by the weight lattice and let $\Lambda_{\mathbb R}^+$ 
be the dominant cone inside $\Lambda_{\mathbb R}$.

A path is a piecewise linear map $\pi:[0,1]\to\Lambda_{\mathbb R}$
such that $\pi(0)=0$ and $\pi(1)\in\Lambda$.  We denote by $\widetilde\Pi$ the set of all paths.
The concatenation $\pi*\eta$ of two paths is defined in the usual way:
$\pi*\eta(t)=\pi(2t)$ for $0\le t\le\frac{1}{2}$, and $\pi*\eta(t)=\eta(2t-1)+\pi(1)$ for $\frac{1}{2} \le t\le 1$.

In~\cite{Littelmann95}, the third author associates to each simple
root $\alpha$ of $\Phi$ a pair $(e_\alpha,f_\alpha)$ of ``root
operators'' from $\widetilde\Pi$ to $\widetilde\Pi\sqcup\{0\}$ and
shows that the construction yields a semi-normal crystal structure
on $\widetilde\Pi$. Here the weight map is given by $\wt(\pi)=\pi(1)$.
To agree with the notation in sect.~\ref{ss:Crystals}, we will write
$\tilde e_i$ and $\tilde f_i$ instead of $e_{\alpha_i}$ and
$f_{\alpha_i}$ for each $i\in I$.

Let $\ell:[0,1]\to\mathbb R$ be a piecewise linear function. We say that
$p\in\mathbb R$ is a local absolute minimum of $\ell$ if there exists 
a compact interval $[a,b]\subseteq[0,1]$ over which $\ell$ takes the value $p$,
and there exists an $\epsilon > 0$ such that $\ell(x)>p$ for all 
$x\in (a-\epsilon,a)\cap [0,1]$ and all $x\in (b,b+\epsilon)\cap [0,1]$.

Given $\pi\in\widetilde\Pi$, we denote by $\mathcal A\pi$ the set of
all paths $\eta\in\widetilde\Pi$ that can be obtained from $\pi$ by
applying a finite sequence of root operators $\tilde e_i$ or
$\tilde f_i$. We say that a path $\pi\in\widetilde\Pi$ is integral if
for each $\eta\in\mathcal A\pi$ and each $i\in I$, all local absolute minima of
the function $t\mapsto\langle\alpha_i^\vee,\eta(t)\rangle$ are integers.

We denote the set of all integral paths by $\Pi$. Obviously, $\Pi$
is a subcrystal of $\widetilde\Pi$. Moreover, the root operators
have a simpler form on $\Pi$, since integral paths need only to be cut
into three parts: the initial part is left invariant, the second part
is reflected, and the third part is translated. Specifically, given
$(\pi,\eta)\in\Pi^2$ and $i\in I$, we have $\eta=\tilde e_i\pi$ if
and only if there exist a negative integer $p\in\mathbb Z$ and two reals $a$
and $b$ with $0\leq a<b\leq1$, such that the function 
$t\mapsto \langle\alpha_i^\vee,\pi(t)\rangle$ is weekly decreasing on $[a,b]$,
and for each $t\in[0,1]$:
\begin{itemize}
\item
if $t\leq a$, then
$\langle\alpha_i^\vee,\pi(t)\rangle\geq p+1$ and
$\eta(t)=\pi(t)$;
\item
if $t=a$, then
$\langle\alpha_i^\vee,\pi(t)\rangle=p+1$;
\item
if $a<t<b$, then
$p\leq\langle\alpha_i^\vee,\pi(t)\rangle<p+1$ and
$\eta(t)=\pi(t)-(\langle\alpha_i^\vee,\pi(t)\rangle-p-1)\,\alpha_i$;
\item
if $t=b$, then
$\langle\alpha_i^\vee,\pi(t)\rangle=p$;
\item
if $t\geq b$, then
$\langle\alpha_i^\vee,\pi(t)\rangle\geq p$ and
$\eta(t)=\pi(t)+\alpha_i$.
\end{itemize}
We say that an integral path $\pi\in\Pi$ is dominant if its image is
contained in $\Lambda_{\mathbb R}^+$.

\begin{other}{Remark}
Let $\Gamma$ be the group of all strictly increasing piecewise
linear maps from $[0,1]$ onto itself, the product being the
composition of functions. The group acts on the set of all paths by right composition:
$\pi \mapsto \pi\circ\gamma$ for a path $\pi$ and $\gamma\in \Gamma$.
We say that $\pi\circ\gamma$ is obtained from $\pi$ by a piecewise linear reparameterization. 
Visibly, the set $\Pi$ of integral paths is stable under this action, the weight map $\textrm{wt}$ is
invariant, and the root operators are equivariant. We can thus safely consider all our previous
constructions modulo this action. In the sequel we sometimes implicitly assume that this quotient has
been performed, i.e.\ among the possible parameterizations we choose one which is appropriate for the application in view. 
\end{other}
The first two items in the following proposition ensure that there is
an abundance of integral paths.
\begin{proposition}
\label{pr:PathModel}
\begin{enumerate}
\item
\label{it:PrPMa}
A dominant path $\pi$ is integral as soon as for each $i\in I$, the
function $t\mapsto\langle\alpha_i^\vee,\pi(t)\rangle$ is weakly increasing.
\item
\label{it:PrPMb}
The set $\Pi$ is stable under concatenation of paths and the map
$\pi\otimes\eta\mapsto\pi*\eta$ is a strict morphism of crystals
from $\Pi\otimes\Pi$ to $\Pi$.
\item
\label{it:PrPMc}
Let $\pi\in\Pi$. Then $\mathcal A\pi$ contains a unique dominant
path $\eta$ and is isomorphic as a crystal to $B(\wt(\eta))$.
\end{enumerate}
\end{proposition}
\begin{proof}
For $\nu\in\Lambda_\mathbb R$ let $\pi_{\nu}$ be the map $[0,1]\rightarrow \Lambda_{\mathbb R}$, $t\mapsto t\nu$.
A path in Proposition~\ref{pr:PathModel}~\ref{it:PrPMa} is of the form $\pi=\pi_{\nu_1}*\ldots*\pi_{\nu_s}$, where 
$\nu_1,\ldots,\nu_s\in\Lambda_{\mathbb R}^+$ are dominant and $\nu_1+\cdots+\nu_s\in\Lambda^+$. 
Such a path can be approximated by a rational dominant path, i.e.\ a path $\eta=\pi_{\mu_1}*\ldots*\pi_{\mu_s}$
such that $\mu_1,\ldots,\mu_s\in\Lambda^+_{\mathbb Q}$ and $\mu_1+\cdots+\mu_s=\nu$. By Lemma~\ref{locally integral concatenation}
below, a rational dominant path can be approximated by a locally integral concatenation (\cite{Littelmann97}, Definition~5.3).
These paths are integral by \textit{loc.~cit.}, Lemma~5.6 and Proposition~5.9. The integrality property in~\ref{it:PrPMa} 
follows now by the continuity of the root operators \textit{loc.~cit.}, property (v)  \textsc{continuity}.

It remains to prove the other two statements. 
The endpoint of a path is by definition an element of the lattice, so the concatenation of integral paths
is an integral path. Moreover, by Lemma 6.12 in \textit{loc.~cit.},
concatenation defines a strict morphism of crystals $\Pi\otimes\Pi\to\Pi$. This shows~\ref{it:PrPMb}. Statement~\ref{it:PrPMc} follows  
by Lemma~6.11 in \textit{loc.~cit.}
\end{proof}

For a rational dominant path $\pi=\pi_{\mu_1}*\cdots*\pi_{\mu_s}$, the property of being a 
locally integral concatenation is equivalent to: 
\begin{description}
\item[($*$)]
For all $j=1,\ldots,s$,  the affine
line passing through $\mu_1+\cdots+\mu_{j-1}$ and $\mu_1+\cdots+\mu_{j}$
meets at least two lattice points.
\end{description}
An equivalent formulation: the affine line meets at least one rational point and one lattice point.

To prove the approximation property used in the proof above, we need the following 
elementary geometric construction. We fix a scalar  product $(\cdot,\cdot)$ on 
$\Lambda_{\mathbb R}$ and let $d(\cdot,\cdot)$ be the corresponding distance function. 
We fix a basis $\mathbb B$ of $\Lambda$ and let $\mathbb L_1\subset \Lambda_{\mathbb R}$ be the associated unit
cube, i.e.\ the set of points in  $\Lambda_{\mathbb R}$ which can be written
as a linear combination of $\mathbb B$ with coefficients in the interval $[0,1]$.
Let $M$ be the maximal distance between two points in $\mathbb L_1$.
Let $P\in \Lambda^+_{\mathbb Q}$ be a dominant rational point and let $S(P,1)$
be the sphere with center $P$ and radius $1$. Let $g$ be a ray starting in $P$
and let $g_1$ be the intersection point of this ray with the sphere $S(P,1)$. 
\begin{lemma}\label{loc_int_directions}
One can find for any $\epsilon>0$ 
a ray $f$ starting in $P$ such that $(f\setminus\{P\})\cap \Lambda\not=\varnothing$, and for 
$\{f_1\}=f\cap S(P,1)$ we have $d(g_1,f_1)<\epsilon$.
\end{lemma}
\begin{proof}
Parametrize $g$ by $g(t) = P+t(g_1-P)$ for $t \ge 0$.
Choose $t_2\gg 0$ and pick $\lambda\in \Lambda$ such that 
$g(t_2)\in \lambda+\mathbb L_1$.
Let $f$ be the ray starting in $P$ passing through $\lambda$. Let $f_1$ be the
intersection point of this ray with $S(P,1)$; then $f(t)=P+t(f_1-P)$ for $t \ge 0$ is a parameterization of $f$.
Set $g_2=g(t_2)$ and  $f_2=f(t_2)$.
Noting that $d(P,f_2)=t_2=d(P,g_2)$ and using the triangular inequality,
we get
$$
d(f_2,\lambda)=|d(P,f_2)-d(P,\lambda)|=|d(P,g_2)-d(P,\lambda)|
\leq d(g_2,\lambda)\leq M
$$
whence $d(g_2,f_2)\leq2M$. By the intercept theorem 
$d(f_1,g_1)/d(g_2,f_2) = 1/ t_2$, and hence  
$d(f_1,g_1) \le  (2M) / t_2$. For $t_2$ large enough we obtain $d(g_1,f_1)<\epsilon$.
\end{proof}
\begin{lemma}\label{locally integral concatenation}
A dominant rational path can be approximated by a locally integral concatenation.
\end{lemma}
\begin{proof}
Let $\pi=\pi_{\lambda_1}*\cdots*\pi_{\lambda_s}$ be a dominant rational path ending in $\lambda\in\Lambda^+$. 
We define the support of an element $\mu\in\Lambda_\mathbb R$ as the subset of simple roots such that $\langle\alpha^\vee,\mu\rangle\not=0$.
We can assume that the support of each $\lambda_j$ is the same as the support of $\lambda$; otherwise
we approximate $\pi$ by a path we get by slightly perturbing $\lambda_1,\ldots, \lambda_s$, for instance by replacing $\lambda_i$ by
$\lambda_i+\epsilon(\lambda/s-\lambda_i)$ for some rational $0<\epsilon\ll 1$.
We can also assume the support of $\lambda$ is $I$, 
otherwise we work within the subspace $\bigcap_{i\in I\setminus\supp(\lambda)}(\ker\alpha_i^\vee)$.

Under these assumptions, small perturbations $\lambda'_1,\ldots, \lambda'_{s-1}$ of the directions $\lambda_1,\ldots, \lambda_{s-1}$ remain dominant, and so does  $\lambda'_s=\lambda-(\lambda'_1+\cdots+ \lambda'_{s-1})$. By Lemma~\ref{loc_int_directions}, one can perturb in such a way that the new path
$\eta:=\pi_{\lambda'_1}*\cdots*\pi_{\lambda'_s}$ is a rational 
dominant path and the first $s-1$ line segments of $\eta$ satisfy the affine line condition $(*)$. 
The last line segment of $\eta$ meets the lattice point $\lambda$ and a rational point, and thus 
satisfies the affine line condition $(*)$ too. Hence $\eta$ is a locally integral concatenation, 
approximating the dominant rational path $\pi$.
\end{proof}

\subsection{From paths to MV cycles}
\label{ss:PathsToMV}
We need additional terminology before we proceed to the main construction
of this section.

To each coroot $\alpha^\vee\in\Phi^\vee$ corresponds an additive
one-parameter subgroup $x_{\alpha^\vee}:\mathbb G_a\to G^\vee$.
Given additionally an integer $p\in\mathbb Z$, we define a map
$$x_{(\alpha^\vee,\,p)}:\mathbb C\to G^\vee(\mathcal K),\quad
a\mapsto x_{\alpha^\vee}(az^p).$$

An affine coroot is a pair $(\alpha^\vee,p)$ consisting of a coroot
$\alpha^\vee\in\Phi^\vee$ and an integer $p\in\mathbb Z$.
The direction of an affine coroot $(\alpha^\vee,p)$ is $\alpha^\vee$.
An affine coroot is said to be positive if its direction is so. We
denote the set of affine coroots by $\Phi_a^\vee$ and the set of
positive affine coroots by $\PACS$.

To an affine coroot $\beta$, besides the map $x_\beta$ defined above,
we attach an hyperplan $H_\beta$ and a negative closed half-space
$H_\beta^-$ in $\Lambda_{\mathbb R}$ as follows:
$$H_{(\alpha^\vee,\,p)}=\bigl\{x\in\Lambda_{\mathbb R}\bigm|
\langle\alpha^\vee,x\rangle=p\bigr\},
\qquad H_{(\alpha^\vee,\,p)}^-=\bigl\{x\in\Lambda_{\mathbb R}\bigm|
\langle\alpha^\vee,x\rangle\leq p\bigr\}.$$

Let $s_\beta$ be the reflection across the hyperplane $H_\beta$; concretely
$$s_{(\alpha^\vee,\,p)}(x)=x-(\langle\alpha^\vee,x\rangle-p)\,\alpha$$
for any $x\in\Lambda_{\mathbb R}$. In addition, we denote by $\tau_\lambda$
the translation $x\mapsto x+\lambda$ by the element $\lambda\in\Lambda$.
The subgroup of $\Aut(\Lambda_{\mathbb R})$ generated by all the reflections
$s_\beta$ is the affine Weyl group $W_a$; when we add the translations
$\tau_\lambda$, we obtain the extended affine Weyl group $\widetilde W_a$.
Then $\tau_\lambda\in W_a$ if and only if $\lambda\in\mathbb Z\Phi$.

The group $\widetilde W_a$ acts on the set $\Phi_a^\vee$ of affine
roots: one demands that $w(H^-_\beta)=H^-_{w\beta}$ for each element
$w\in\widetilde W_a$ and each affine coroot $\beta\in\Phi_a^\vee$.
Then for each $\beta\in\Phi_a^\vee$ and each $\lambda\in\Lambda$, we
have $x_{\tau_\lambda\beta}(a)=z^\lambda\,x_\beta(a)\,z^{-\lambda}$
for all $a\in\mathbb C$.

We denote by $\mathfrak H$ the arrangement formed by the hyperplanes
$H_\beta$, where $\beta\in\Phi_a^\vee$. It divides the vector space
$\Lambda_{\mathbb R}$ into faces. The closure of a face is the disjoint
union of faces of smaller dimension. Endowed with the set of all faces,
$\Lambda_{\mathbb R}$ becomes a polysimplicial complex, called the affine
Coxeter complex.

For each face $\mathfrak f$ of the affine Coxeter complex, we denote by
$N^\vee(\mathfrak f)$ the subgroup of $N^\vee(\mathcal K)$ generated
by the elements $x_{\alpha^\vee}(az^p)$, where $a\in\mathcal O$ and
$(\alpha^\vee,p)$ is a positive affine coroot such that
$\mathfrak f\subseteq H_{(\alpha^\vee,\,p)}^-$.
We note that $N^\vee(\tau_\lambda\mathfrak f)=z^\lambda\,
N^\vee(\mathfrak f)\,z^{-\lambda}$ for each face $\mathfrak f$ and
each $\lambda\in\Lambda$ and that $N^\vee(\mathfrak f)\subseteq
N^\vee(\mathcal O)$ if $\mathfrak f\subseteq\Lambda_{\mathbb R}^+$.

For $x\in\Lambda_{\mathbb R}$, we denote by $\mathfrak f_x$ the face in
the affine Coxeter complex that contains the point $x$. We use the
symbol $\prod'$ to denote the restricted product of groups, consisting
of families involving only finitely many nontrivial terms. Also,
remember the notation introduced in sect.~\ref{ss:MVCyc} about the
$n$-fold convolution variety $\Gr_n$.

With these conventions, given $(\pi_1,\ldots,\pi_n)\in\Pi^n$,
we define $\mathring{\mathbf Z}(\pi_1\otimes\cdots\otimes\pi_n)$
as the subset of $\Gr_n$ of all elements
$$\left[\left(\prod_{t_1\in[0,1]}v_{1,t_1}\right)z^{\wt(\pi_1)},\;
\ldots,\;
\left(\prod_{t_n\in[0,1]}v_{n,t_n}\right)z^{\wt(\pi_n)}\right]$$
with
$$((v_{1,t_1}),\ldots,(v_{n,t_n}))\in
\mathop{\prod\nolimits'}\limits_{t_1\in[0,1]}
N^\vee\bigl(\mathfrak f_{\pi_1(t_1)}\bigr)\times\cdots\times
\mathop{\prod\nolimits'}\limits_{t_n\in[0,1]}
N^\vee\bigl(\mathfrak f_{\pi_n(t_n)}\bigr).$$
The resort to the restricted infinite products is here merely
cosmetic; we could as well use finite products, since the paths
$\pi_j$ meet only finitely many faces. The point that
matters is to compute products of elements in the various groups
$N^\vee(\mathfrak f_{\pi_j(t)})$ in the order indicated by the paths.

\begin{proposition}
\label{pr:PropRingZ}
Let $(\pi_1,\ldots,\pi_n)\in\Pi^n$.
\begin{enumerate}
\item
\label{it:PrPRZa}
The set $\mathring{\mathbf Z}(\pi_1\otimes\cdots\otimes\pi_n)$
is stable under left multiplication by $N^\vee(\mathcal O)$.
\item
\label{it:PrPRZb}
Let $\mu=\wt(\pi_1)+\cdots+\wt(\pi_n)$; then the set
$\mathring{\mathbf Z}(\pi_1\otimes\cdots\otimes\pi_n)$ is an
irreducible constructible subset of $(m_n)^{-1}(S_\mu)$.
\item
\label{it:PrPRZc}
We have
$$\mathring{\mathbf Z}(\pi_1\otimes\cdots\otimes\pi_n)=
\Psi\Bigl(\mathring{\mathbf Z}(\pi_1)\ltimes\cdots\ltimes
\mathring{\mathbf Z}(\pi_n)\Bigr)$$
and
$$\mathring{\mathbf Z}(\pi_1*\cdots*\pi_n)=
m_n\Bigl(\mathring{\mathbf Z}(\pi_1\otimes\cdots\otimes\pi_n)\Bigr).$$
\item
\label{it:PrPRZd}
Let $i\in I$ and compute $\eta_1\otimes\cdots\otimes\eta_n=
\tilde e_i(\pi_1\otimes\cdots\otimes\pi_n)$ in the crystal
$\Pi^{\otimes n}$, assuming that this operation is doable.
Then\/ $\mathring{\mathbf Z}(\pi_1\otimes\cdots\otimes\pi_n)$
is contained in the closure of\/
$\mathring{\mathbf Z}(\eta_1\otimes\cdots\otimes\eta_n)$ in~$\Gr_n$.
\end{enumerate}
\end{proposition}
\begin{proof}
Assertion~\ref{it:PrPRZa} is a direct consequence of the equality
$N^\vee\bigl(\mathfrak f_{\pi_1(0)}\bigr)=N^\vee(\mathfrak f_0)=
N^\vee(\mathcal O)$.

Assertion~\ref{it:PrPRZb} comes from general principles once we
have replaced the restricted infinite product by a finite one.

The first equation in~\ref{it:PrPRZc} is tautological. In the
second one, we view the concatenation $\pi=\pi_1*\cdots*\pi_n$ as a
map from $[0,n]$ to $\Lambda_{\mathbb R}$, each path $\pi_1$, \dots,
$\pi_n$ being travelled at nominal speed. For $j\in\{1,\ldots,n\}$,
we set $\nu_{j-1}=\wt(\pi_1)+\cdots+\wt(\pi_{j-1})$; then
$\pi(t+(j-1))=\nu_{j-1}+\pi_j(t)$ for all $t\in[0,1]$, and accordingly
$N^\vee\bigl(\mathfrak f_{\pi(t+(j-1))}\bigr)=z^{\nu_{j-1}}\,
N^\vee\bigl(\mathfrak f_{\pi_j(t)}\bigr)\,z^{-\nu_{j-1}}$.
A banal calculation then yields the desired result.

The proof of assertion~\ref{it:PrPRZd} is much more involved.
We defer its presentation to sect.~\ref{ss:ProofPrPRZd}.
\end{proof}

For $(\pi_1,\ldots,\pi_n)\in\Pi^n$, we denote by
$\mathbf Z(\pi_1\otimes\cdots\otimes\pi_n)$ the closure of
$\mathring{\mathbf Z}(\pi_1\otimes\cdots\otimes\pi_n)$ in
$(m_n)^{-1}(S_\mu)$, where $\mu=\wt(\pi_1)+\cdots+\wt(\pi_n)$.

\begin{theorem}
\label{th:ZIsMVCyc}
Let $(\pi_1,\ldots,\pi_n)\in\Pi^n$, set
$\mu=\wt(\pi_1)+\cdots+\wt(\pi_n)$ and for $j\in\{1,\ldots,n\}$,
let $\lambda_j$ be the weight of the unique dominant path in
$\mathcal A\pi_j$. Then $\mathbf Z(\pi_1\otimes\cdots\otimes\pi_n)$
is an MV cycle; specifically $\mathbf Z(\pi_1\otimes\cdots\otimes\pi_n)
\in\,{}_*\mathscr Z(\lambda_1,\ldots,\lambda_n)_\mu$.
\end{theorem}
\begin{proof}
We start with the particular case $n=1$. Let $\pi\in\Pi$, let $\eta$
be the unique dominant path in $\mathcal A\pi$, set $\lambda=\wt(\eta)$
and $\mu=\wt(\pi)$, and set $p=2\rho(\lambda)$ and $k=\rho(\lambda+\mu)$.
By Proposition~\ref{pr:PathModel}~\ref{it:PrPMc}, the crystal
$\mathcal A\pi$ is isomorphic to $B(\lambda)$, so it contains
a unique lowest weight element $\xi$. Then $\pi$ can be reached
by applying a sequence of root operators $\tilde f_i$ to $\eta$
or by applying a sequence of root operators $\tilde e_i$ to $\xi$.
Thus, there exists a finite sequence $(\pi_0,\ldots,\pi_p)$ of elements
in $\mathcal A\pi$ such that $\pi_0=\xi$, $\pi_k=\pi$, $\pi_p=\eta$
and such that each $\pi_{j+1}$ is obtained from $\pi_j$ by applying
a root operator $\tilde e_i$.

Since $\eta$ is dominant, each face $\mathfrak f_{\eta(t)}$ is
contained in $\Lambda_{\mathbb R}^+$, so each group $N^\vee\bigl(
\mathfrak f_{\eta(t)}\bigr)$ is contained in $N^\vee(\mathcal O)$.
Then by construction
$\mathring{\mathbf Z}(\eta)\subseteq N^\vee(\mathcal O)\,L_\lambda$,
and therefore $\overline{\mathbf Z(\eta)}$ (the closure of
$\mathbf Z(\eta)$ in $\Gr$) is contained in $\overline{\Gr^\lambda}$.
Also, Proposition~\ref{pr:PropRingZ}~\ref{it:PrPRZd} implies that
\begin{equation}
\label{eq:ZIsMVCyc}
\overline{\mathbf Z(\pi_0)}\subseteq\overline{\mathbf Z(\pi_1)}
\subseteq\cdots\subseteq\overline{\mathbf Z(\pi_p)}.
\end{equation}
These inclusions are strict because $\overline{\mathbf Z(\pi_j)}$ is
contained in the closure of $\overline{\Gr^\lambda}\cap S_{\wt(\pi_j)}$,
which is disjoint from $S_{\wt(\pi_{j+1})}$ by \cite{MirkovicVilonen},
Proposition~3.1~(a), while $\mathbf Z(\pi_{j+1})$ is contained in
$S_{\wt(\pi_{j+1})}$. Thus \eqref{eq:ZIsMVCyc} is a strictly increasing
chain of closed irreducible subsets of $\overline{\Gr^\lambda}$.
As $\overline{\Gr^\lambda}$ has dimension $p$, we see that each
$\overline{\mathbf Z(\pi_j)}$ has dimension $j$.

In particular, $\overline{\mathbf Z(\pi)}$ has dimension $k$. But
$\mathbf Z(\pi)$ is locally closed, because it is a closed subset
of~$S_\mu$ which is locally closed. So $\mathbf Z(\pi)$ has
dimension $k$. At this point, we know that $\mathbf Z(\pi)$ is a
closed irreducible subset of $\overline{\Gr^\lambda}\cap S_\mu$ of
dimension $k=\rho(\lambda+\mu)$. Therefore $\mathbf Z(\pi)$ belongs
to~$\mathscr Z(\lambda)_\mu$.

The reasoning above establishes the case $n=1$ of the Theorem.
The general case then follows from
Propositions~\ref{pr:MVCyc}~\ref{it:PrMVCb}
and~\ref{pr:PropRingZ}~\ref{it:PrPRZc}.
\end{proof}

\subsection{A more economical definition}
For the proofs in the following sections, it will be convenient
to have a more economical presentation of the sets
$\mathring{\mathbf Z}(\pi_1\otimes\cdots\otimes\pi_n)$.
We need a few additional pieces of notation.

When $\mathfrak f$ and $\mathfrak f'$ are two faces of the affine
Coxeter complex such that $\mathfrak f$ is contained in the closure
$\overline{\mathfrak f'}$ of $\mathfrak f'$, we denote by
$\PACS(\mathfrak f,\mathfrak f')$ the set of all positive
affine coroots $\beta$ such that $\mathfrak f\subseteq H_\beta$
and $\mathfrak f'\not\subseteq H^-_\beta$. We denote by
$\mathscr N^\vee(\mathfrak f,\mathfrak f')$ the subgroup of
$N^\vee(\mathcal K)$ generated by the elements $x_\beta(a)$ with
$\beta\in\PACS(\mathfrak f,\mathfrak f')$ and $a\in\mathbb C$.

The following result is Proposition~19~(ii) in \cite{BaumannGaussent}.
\begin{lemma}
\label{le:PropBaGa}
Let $\mathfrak f$ and $\mathfrak f'$ be two faces of the affine Coxeter
complex such that $\mathfrak f\subseteq\overline{\mathfrak f'}$.
Then $N^\vee(\mathfrak f)$ is the bicrossed product
$\mathscr N^\vee(\mathfrak f,\mathfrak f')\bowtie N^\vee(\mathfrak f')$;
in other words, the product induces a bijection
$$\mathscr N^\vee(\mathfrak f,\mathfrak f')\times N^\vee(\mathfrak f')
\xrightarrow\simeq N^\vee(\mathfrak f).$$
In addition, the map
$$\mathbb C^{\PACS(\mathfrak f,\mathfrak f')}\to\mathscr N^\vee
(\mathfrak f,\mathfrak f'),\quad(a_\beta)\mapsto
\prod_{\beta\in\PACS(\mathfrak f,\mathfrak f')}x_\beta(a_\beta)$$
is bijective, whichever order on $\PACS(\mathfrak f,\mathfrak f')$
is used to compute the product.
\end{lemma}

Given $\pi\in\Pi$ and $t\in[0,1[$, we denote by
$\mathfrak f_{\pi(t+0)}$ the face in the affine Coxeter complex
that contains the points $\pi(t+h)$ for all small enough $h>0$.
Obviously, its closure meets, hence contains, the face
$\mathfrak f_{\pi(t)}$. We set
$\PACS(\pi,t)=\PACS\bigl(\mathfrak f_{\pi(t)},\mathfrak f_{\pi(t+0)}
\bigr)$ and $\mathscr N^\vee(\pi,t)=\mathscr N^\vee
\bigl(\mathfrak f_{\pi(t)},\mathfrak f_{\pi(t+0)}\bigr)$.

Concretely, $\PACS(\pi,t)$ is the set of all $\beta\in\PACS$ such
that $\pi$ quits the half-space $H_\beta^-$ at time $t$ and
$\mathscr N^\vee(\pi,t)$ is the subgroup of $N^\vee(\mathcal K)$
generated by the elements $x_\beta(a)$ with $\beta\in\PACS(\pi,t)$
and $a\in\mathbb C$. Note that $\PACS(\pi,t)$ is empty save for
finitely many $t$.

\begin{proposition}
\label{pr:AltDefRingZ}
Let $(\pi_1,\ldots,\pi_n)\in\Pi^n$. Then
$\mathring{\mathbf Z}(\pi_1\otimes\cdots\otimes\pi_n)$ is the set of
all elements
$$\left[\left(\prod_{t_1\in[0,1[}v_{1,t_1}\right)z^{\wt(\pi_1)},\;
\ldots,\;\left(\prod_{t_n\in[0,1[}v_{n,t_n}\right)z^{\wt(\pi_n)}
\right]$$
with
$$((v_{1,t_1}),\ldots,(v_{n,t_n}))\in
\prod_{t_1\in[0,1[}\mathscr N^\vee(\pi_1,t_1)\times\cdots\times
\prod_{t_n\in[0,1[}\mathscr N^\vee(\pi_n,t_n).$$
\end{proposition}
\begin{proof}
Let $\pi\in\Pi$. Let $(t_1,\ldots,t_m)$ be the ordered list of all
elements $t\in[0,1[$ such that $\PACS(\pi,t)\neq\varnothing$ and set
$t_{m+1}=1$.

Pick $\ell\in\{1,\ldots,m\}$; between the times $t_\ell$ and $t_{\ell+1}$,
the path $\pi$ never quits the half-space $H_\beta^-$ of a positive
affine coroot $\beta$; as a consequence, the map
$t\mapsto N^\vee\bigl(\mathfrak f_{\pi(t)}\bigr)$ is non-decreasing
on the interval $]t_\ell,t_{\ell+1}]$. This map is also non-decreasing
on the interval $[0,t_1]$ if $t_1>0$. However when $t$ goes past
a point $t_\ell$, the group $N^\vee\bigl(\mathfrak f_{\pi(t)}\bigr)$
slims down by dropping the subgroup $\mathscr N^\vee(\pi,t_\ell)$,
in the sense brought by the bicrossed product
$N^\vee\bigl(\mathfrak f_{\pi(t_\ell)}\bigr)=\mathscr N^\vee(\pi,t_\ell)
\bowtie N^\vee\bigl(\mathfrak f_{\pi(t_\ell+0)}\bigr)$.

It follows that for any family $(u_t)$ in
$\prod'_{t\in[t_\ell,t_{\ell+1}]}N^\vee\bigl(\mathfrak f_{\pi(t)}\bigr)$,
there exists $(v,u)\in\mathscr N^\vee(\pi,t_\ell)\times
N^\vee\bigl(\mathfrak f_{\pi(t_{\ell+1})}\bigr)$ such that
$\prod_{t\in[t_\ell,t_{\ell+1}]}u_t=vu$. To see this, one
decomposes $u_{t_\ell}$ as $vu'$ according to the bicrossed product
and one defines $u$ as the product of $u'$ and of the $u_t$ for
$t\in]t_\ell,t_{\ell+1}]$. Assembling these pieces (and an analogous
statement over the interval $[0,t_1]$ if $t_1>0$) from left to right,
and noting that $N^\vee\bigl(\mathfrak f_{\pi(1)}\bigr)$ stabilizes
$L_{\wt(\pi)}$, we deduce that $\mathring{\mathbf Z}(\pi)$ is the
image of the map
$$\prod_{t\in[0,1[}\mathscr N^\vee(\pi,t)\to\Gr,
\quad(v_t)\mapsto\left(\prod_{t\in[0,1[}v_t\right)\,L_{\wt(\pi)}.$$

This proves our statement in the case of just one path.
The general case then follows from Proposition~\ref{pr:PropRingZ}
\ref{it:PrPRZc}.
\end{proof}

\subsection{Isomorphisms of crystals}
\label{ss:IsomCrys}
In the previous section we explained how to build elements in
${}_*\mathscr Z(\bm\lambda)_\mu$, while in sect.~\ref{se:LPerf}
we were dealing with MV cycles in $\mathscr Z(\bm\lambda)$.
This clumsiness is due to a mismatch between the definition of
the path model and the conventions in \cite{MirkovicVilonen} and
\cite{BaumannKamnitzerKnutson}. To mitigate the disagreement, we
define a crystal structure on
$${}_*\mathscr Z(\bm\lambda)=\bigsqcup_{\mu\in\Lambda}
{}_*\mathscr Z(\bm\lambda)_\mu.$$

Remember the setup of sect.~\ref{ss:SatResLevi}: we consider a
subset $J\subseteq I$, define an action of $\mathbb C^\times$ on
$\Gr$ given by a special dominant weight $\theta_J$, and get
the diagram~\eqref{eq:ParabRetrac}. Now let
$(\bm\lambda,\nu)\in(\Lambda^+)^n\times\Lambda$, let $\zeta$ be the
coset of $\nu$ modulo $\mathbb Z\Phi_J$, and let
$Z\in\,{}_*\mathscr Z(\bm\lambda)_\nu$. Then
$m_n(Z)\subseteq S_\nu\subseteq\Gr_{J,\zeta}^+$ and there is a unique
weight $\mu\in\Lambda_J^+$ characterized by the conditions
$$(p_{J,\zeta}\circ m_n)(Z)\subseteq\overline{\Gr_J^\mu}\quad\text{and}
\quad(p_{J,\zeta}\circ m_n)(Z)\cap\Gr_J^\mu\neq\varnothing.$$
We denote this weight by $\mu_J(Z)$.

By analogy with Proposition~\ref{pr:CharacCrysMV}, we can
then claim the existence of a crystal structure on
${}_*\mathscr Z(\bm\lambda)$ such that for all $\nu\in\Lambda$,
$i\in I$ and $Z\in\mathscr Z(\bm\lambda)_\nu$:
\begin{itemize}
\item
We have $\wt(Z)=\nu$, $\varepsilon_i(Z)=\frac12\bigl\langle
\alpha_i^\vee,\mu_{\{i\}}(Z)-\nu\bigr\rangle$ and\/
$\varphi_i(Z)=\frac12\bigl\langle\alpha_i^\vee,
\mu_{\{i\}}(Z)+\nu\bigr\rangle$.
\item
Let $Y\in\,{}_*\mathscr Z(\bm\lambda)_{\nu+\alpha_i}$. Then $Y=\tilde e_iZ$
if and only if $\overline Y\supseteq Z$ and $\mu_{\{i\}}(Y)=\mu_{\{i\}}(Z)$.
\end{itemize}

\begin{theorem}
\label{th:ZMorCrys}
Let $(\lambda_1,\ldots,\lambda_n)\in(\Lambda^+)^n$, and for each
$j\in\{1,\ldots,n\}$ choose a subcrystal $\Pi_j$ of $\Pi$
isomorphic to $B(\lambda_j)$. Then the map
$(\pi_1,\ldots,\pi_n)\mapsto\mathbf Z(\pi_1\otimes\cdots\otimes\pi_n)$
is an isomorphism of crystals
$$\Pi_1\otimes\cdots\otimes\Pi_n\xrightarrow\simeq\,
{}_*\mathscr Z(\lambda_1,\ldots,\lambda_n).$$
\end{theorem}
\begin{proof}
Let $i\in I$ and let $(\pi_1,\ldots,\pi_n)\in\Pi_1\times\cdots\times\Pi_n$.
Set $\nu=\wt(\pi_1)+\cdots+\wt(\pi_n)$, let $\zeta$ be the coset of
$\nu$ modulo $\mathbb Z\alpha_i$, and set $\pi=\pi_1*\cdots*\pi_n$,
$$p=\min\bigl\{\langle\alpha_i^\vee,\pi(t)\rangle\bigm|
t\in[0,1]\bigr\}\quad\text{and}\quad
q=\langle\alpha_i^\vee,\nu\rangle=\langle\alpha_i^\vee,\pi(1)\rangle.$$
For any $a\in\mathbb C[z,z^{-1}]$ and any positive coroot $\alpha^\vee$,
we have
$$\lim_{c\to0}\;\theta_{\{i\}}(c)\;x_{\alpha^\vee}(a)\;
\theta_{\{i\}}(c)^{-1}=
\begin{cases}
x_{\alpha^\vee}(a)&\text{if $\alpha^\vee=\alpha_i^\vee$,}\\
1&\text{otherwise.}
\end{cases}$$

Using Proposition~\ref{pr:AltDefRingZ}, we see that
$p_{\{i\},\zeta}\Bigl(\mathring{\mathbf Z}(\pi)\Bigr)$ is the set of
all elements of the form
$$\lim_{c\to0}\prod_{t\in[0,1[}\left(\prod_{\beta\in\PACS(\pi,t)}
\theta_{\{i\}}(c)\;x_\beta(a_{t,\beta})\;\theta_{\{i\}}(c)^{-1}\right)\,L_\nu$$
where $a_{t,\beta}$ are complex numbers. All factors in the product
disappear in the limit $c\to0$, except those for the affine roots
$\beta$ of direction $\alpha_i^\vee$. Let $(\alpha_i^\vee,p_1)$,
\dots, $(\alpha_i^\vee,p_s)$ be these affine roots. Since
the function $t\mapsto\langle\alpha_i^\vee,\pi(t)\rangle$ assumes
the value $p$ and thereafter reaches the value $q$, the path $\pi$
must, at some point, quit each half-space
$H_{(\alpha_i^\vee,\,p)}^-$, $H_{(\alpha_i^\vee,\,p+1)}^-$, \dots,
$H_{(\alpha_i^\vee,\,q-1)}^-$, so
$$\{p,p+1,\ldots,q-1\}\subseteq\{p_1,\ldots,p_s\}\subseteq\{p,p+1,\ldots\}.$$
We conclude that
$$p_{\{i\},\zeta}\Bigl(\mathring{\mathbf Z}(\pi)\Bigr)=
\bigl\{x_{\alpha_i^\vee}(az^p)\,L_\nu\bigm|a\in\mathcal O/
z^{q-p}\mathcal O\bigr\}.$$
Proposition~\ref{pr:PropRingZ}~\ref{it:PrPRZc} and a variant of
Proposition~\ref{pr:MVRk1} then yield
$$(p_{\{i\},\zeta}\circ m_n)\Bigl(\mathring{\mathbf Z}
(\pi_1\otimes\cdots\otimes\pi_n)\Bigr)=
p_{\{i\},\zeta}\Bigl(\mathring{\mathbf Z}(\pi)\Bigr)=
\overline{\Gr_{\{i\}}^\mu}\cap S_{\{i\},\nu}^{}$$
where $\mu=\nu-p\alpha_i$. Thus,
\begin{equation}
\label{eq:ZMorCrys1}
\mu_{\{i\}}(\mathbf Z(\pi_1\otimes\cdots\otimes\pi_n))=\nu-p\alpha_i
\end{equation}
and
$$\varepsilon_i(\mathbf Z(\pi_1\otimes\cdots\otimes\pi_n))=-p
\quad\text{and}\quad
\varphi_i(\mathbf Z(\pi_1\otimes\cdots\otimes\pi_n))=q-p.$$

These latter equations show that the map
$(\pi_1,\ldots,\pi_n)\mapsto\mathbf Z(\pi_1\otimes\cdots\otimes\pi_n)$
is compatible with the functions $\varepsilon_i$ and $\varphi_i$.

Now compute
$$\eta_1\otimes\cdots\otimes\eta_n=
\tilde e_i(\pi_1\otimes\cdots\otimes\pi_n)$$
in the crystal $\Pi_1\otimes\cdots\otimes\Pi_n$,
assuming this operation to be doable.
By Proposition~\ref{pr:PropRingZ}~\ref{it:PrPRZd},
\begin{equation}
\label{eq:ZMorCrys2}
\overline{\mathbf Z(\eta_1\otimes\cdots\otimes\eta_n)}\supseteq
\mathbf Z(\pi_1\otimes\cdots\otimes\pi_n).
\end{equation}
Let $\eta=\eta_1*\cdots*\eta_n$. Then
$\eta=\tilde e_i\pi$ by Proposition~\ref{pr:PathModel}~\ref{it:PrPMb},
and therefore
$$\wt(\eta)=\nu+\alpha_i\quad\text{and}\quad
\min\bigl\{\langle\alpha_i^\vee,\eta(t)\rangle\bigm|
t\in[0,1]\bigr\}=p+1.$$
Repeating the arguments above, we get
$$\mu_{\{i\}}(\mathbf Z(\eta_1\otimes\cdots\otimes\eta_n))
=(\nu+\alpha_i)-(p+1)\alpha_i=\nu-p\alpha_i.$$
Together with \eqref{eq:ZMorCrys1} and \eqref{eq:ZMorCrys2}, this
gives
$$\mathbf Z(\eta_1\otimes\cdots\otimes\eta_n)=
\tilde e_i\,\mathbf Z(\pi_1\otimes\cdots\otimes\pi_n).$$
We conclude that the map
$(\pi_1,\ldots,\pi_n)\mapsto\mathbf Z(\pi_1\otimes\cdots\otimes\pi_n)$
has the required compatibility with the operations $\tilde e_i$.
\end{proof}

\begin{corollary}
\label{co:ProdCrysMV}
Let $(\lambda_1,\ldots,\lambda_n)\in(\Lambda^+)^n$. Then the map
$(Z_1,\ldots,Z_n)\mapsto\overline{\Psi(Z_1\ltimes\cdots\ltimes Z_n)}$
from Proposition~\ref{pr:MVCyc}~\ref{it:PrMVCb} is an isomorphism of
crystals
$${}_*\mathscr Z(\lambda_1)\otimes\cdots\otimes\,{}_*\mathscr Z(\lambda_n)
\xrightarrow\simeq\,{}_*\mathscr Z(\lambda_1,\ldots,\lambda_n).$$
\end{corollary}

The crystals $\mathscr Z(\bm\lambda)$ enjoy a factorization property
analogous to Corollary~\ref{co:ProdCrysMV}; \textbf{one must however
use the opposite tensor product on crystals.}

\begin{other}{Remark}
The plactic algebra \cite{LascouxSchutzenberger} is an algebraic combinatorial tool invented by Lascoux and Schützenberger
long before the notion of a crystal basis of a representation was introduced. Loosely speaking, for a  complex reductive 
algebraic group, the plactic algebra is the algebra having as basis the union 
$\bigcup_{\lambda\in \Lambda^+} B(\lambda)$ of the crystal bases $B(\lambda)$ for all irreducible representations, the product 
being given by the tensor product of crystals. For $G=\SL_n(\mathbb C)$, Lascoux and Schützenberger give a description of such an 
algebra in terms of the word algebra modulo the Knuth relations, and it was shown later
that this algebra is isomorphic to the one given by the crystal basis. A combinatorial Lascoux-Schützenberger type description
for the other types was given in \cite{Littelmann96}; this description uses the path model.

It is natural to ask whether it is possible to do the same with MV cycles: endow the set of all MV cycles for all dominant weights $\lambda\in\Lambda^+$
with the structure of a crystal and define (in a geometric way) a multiplication on the cycles which mimics the plactic algebra. For $G=\SL_n(\mathbb C)$,
a positive answer was given in \cite{GaussentLittelmannNguyen}. This approach was adapted to the symplectic case in \cite{Torres}.

The results in this section can be naturally viewed as a generalization of \cite{GaussentLittelmannNguyen} 
to arbitrary connected reductive groups. Using \cite{Littelmann96} and 
Proposition~\ref{pr:PathModel}, one can use the set $\Pi$ to construct the plactic algebra so that
it has as basis equivalence classes (generalized Knuth relations) of elements in $\Pi$. The sets $\mathring{\mathbf Z}(\pi_1\otimes\cdots\otimes\pi_n)$
(Proposition~\ref{pr:PropRingZ})
replace in the general setting the Białynicki-Birula cells in \cite{GaussentLittelmannNguyen}.
By combining Proposition~\ref{pr:PropRingZ}~\ref{it:PrPRZc} and Theorem~\ref{th:ZMorCrys}, we see that the closure of $m_n\Bigl(\mathring{\mathbf Z}(\pi_1\otimes\cdots\otimes\pi_n)\Bigr)$ is an MV cycle which depends only on the class of the path $\pi_1*\cdots*\pi_n$ modulo the generalized Knuth relations.
In particular, the main result of \cite{GaussentLittelmannNguyen} follows as a special case. 

A different approach to this problem was taken by Xiao and Zhu
\cite{XiaoZhu}. They define a set of `elementary Littelmann paths',
modeled over minuscule or quasi-minuscule representations,
use the methods from \cite{NgoPolo} to assign an MV cycle to each
concatenation of elementary Littelmann paths, and show that the
resulting map factorizes through the generalized Knuth relations.
\end{other}

\subsection{Proof of Proposition~\ref{pr:PropRingZ}~\ref{it:PrPRZd}}
\label{ss:ProofPrPRZd}
This section can be skipped without substantial loss for the appreciation
of our main storyline. We follow the same method as in
\cite{BaumannGaussent}, proof of Proposition~5.11.

The group $G^\vee(\mathbb C)$ is generated by elements
$x_{\alpha^\vee}(a)$ and $c^\lambda$,
where $(a,\alpha^\vee)\in\mathbb C\times\Phi^\vee$ and
$(c,\lambda)\in\mathbb C^\times\times\Lambda$, which obey the
following relations:
\begin{itemize}
\item
For any $(a,\alpha^\vee)\in\mathbb C\times\Phi^\vee$ and
any $(c,\lambda)\in\mathbb C^\times\times\Lambda$,
$$c^\lambda\;x_{\alpha^\vee}(a)\;c^{-\lambda}=
x_{\alpha^\vee}\bigl(c^{\langle\alpha^\vee,\lambda\rangle}a\bigr).$$
\item
Given two linearly independent elements $\alpha^\vee$ and $\beta^\vee$
in $\Phi^\vee$, there exist constants $C_{i,j}$ such that
\begin{equation}
\label{eq:ChevComm}
x_{\alpha^\vee}(a)\;x_{\beta^\vee}(b)\;x_{\alpha^\vee}(a)^{-1}\;
x_{\beta^\vee}(b)^{-1}=\prod_{(i,j)}\,x_{i\alpha^\vee+j\beta^\vee}
\bigl(C_{i,j}\,a^ib^j\bigr)
\end{equation}
for any $(a,b)\in\mathbb C^2$. The product in the right-hand side
is taken over all pairs of positive integers $(i,j)$ for which
$i\alpha^\vee+j\beta^\vee\in\Phi^\vee$, in order of increasing $i+j$.
\end{itemize}
Further, the one-parameter subgroups $x_{\alpha^\vee}$ can
be normalized so that for any root $\alpha\in\Phi$:
\begin{itemize}
\item
For any $(a,b)\in\mathbb C^2$ such that $1-ab\neq0$,
\begin{equation}
\label{eq:SL2Comm}
x_{\alpha^\vee}(a)\;x_{-\alpha^\vee}(b)=x_{-\alpha^\vee}(b/(1-ab))\;
(1-ab)^\alpha\;x_{\alpha^\vee}(a/(1-ab)).
\end{equation}
\item
There exists an element $\overline{s_\alpha}\in G^\vee(\mathbb C)$
such that for any $a\in\mathbb C^\times$,
\begin{equation}
\label{eq:TitsRela}
x_{\alpha^\vee}(a)\;x_{-\alpha^\vee}(a^{-1})\;x_{\alpha^\vee}(a)
=x_{-\alpha^\vee}(a^{-1})\;x_{\alpha^\vee}(a)\;x_{-\alpha^\vee}(a^{-1})
=a^\alpha\;\overline{s_\alpha}=\overline{s_\alpha}\;a^{-\alpha}.
\end{equation}
\end{itemize}
This element $\overline{s_\alpha}$ lifts in the normalizer of
$T^\vee(\mathbb C)$ the reflection $s_\alpha\in W$ along the root
$\alpha$. All the above relations also hold for scalars $b$,
$c$ in $\mathcal K$, provided of course that we regard them
in~$G^\vee(\mathcal K)$.

The Chevalley commutation relation~\eqref{eq:ChevComm} implies the
following easy lemma.
\begin{lemma}
\label{le:StabFace}
Let $\mathfrak f$ be a face of the affine Coxeter complex, let
$(\alpha^\vee,p)$ and $(\beta^\vee,q)$ be two positive affine
coroots, and let $(a,b)\in\mathcal O^2$. Assume that $\alpha^\vee$
is simple, that $\alpha^\vee\neq\beta^\vee$, and that
$$\mathfrak f\subseteq H_{(-\alpha^\vee,\,-p)}^-\cap
H_{(\beta^\vee,\,q)}^-.$$
Then
$$x_{-\alpha^\vee}(az^{-p})\;x_{\beta^\vee}(bz^q)\;
x_{-\alpha^\vee}(-az^{-p})\in N^\vee(\mathfrak f).$$
\end{lemma}
\begin{proof}
We consider the situation set forth in the statement of the lemma.
Using \eqref{eq:ChevComm}, we write
\begin{equation}
\label{eq:StabFace}
x_{-\alpha^\vee}(az^{-p})\;x_{\beta^\vee}(bz^q)\;
x_{-\alpha^\vee}(-az^{-p})\;x_{\beta^\vee}(-bz^q)=\prod_{(i,j)}\,
x_{-i\alpha^\vee+j\beta^\vee}\bigl(C_{i,j}\,a^ib^jz^{-ip+jq}\bigr)
\end{equation}
where the product in the right-hand side is taken over all pairs of
positive integers $(i,j)$ for which $-i\alpha^\vee+j\beta^\vee$ is a
coroot.

Consider such a pair $(i,j)$. In view of our assumptions, the coroot
$-i\alpha^\vee+j\beta^\vee$ is necessarily positive. Moreover for any
$x\in\mathfrak f$ we have
$$\langle-i\alpha^\vee+j\beta^\vee,x\rangle=
i\langle-\alpha^\vee,x\rangle+j\langle\beta^\vee,x\rangle\leq
i(-p)+jq,$$
so $\mathfrak f\subseteq H_{(-i\alpha^\vee+j\beta^\vee,\,-ip+jq)}^-$.
It follows that the right-hand side of \eqref{eq:StabFace} lies in
$\mathscr N^\vee(\mathfrak f)$, which readily implies the statement.
\end{proof}

Given $g\in N^\vee(\mathcal K)$, there is a
unique tuple $(a_i)\in\mathcal K^I$ such that
$$g\equiv\prod_{i\in I}x_{\alpha_i^\vee}(a_i)\ \;\bmod
(N^\vee(\mathcal K),N^\vee(\mathcal K));$$
looking at a specific $i\in I$, we denote by $\mathbf a_{i,p}(g)$
the coefficient of $z^p$ in the Laurent series~$a_i$. This procedure
defines a morphism of groups
$\mathbf a_{i,p}:N^\vee(\mathcal K)\to\mathbb C$ for each pair
$(i,p)\in I\times\mathbb Z$.

\begin{lemma}
\label{le:PRZCtrl}
Let $\pi\in\Pi$ and let $(t_1,\ldots,t_m)$ be the ordered list of
all elements $t\in[0,1[$ such that $\PACS(\pi,t)\neq\varnothing$.
Set $t_{m+1}=1$. Let $i\in I$ and set
$$p=\min\bigl\{\langle\alpha_i^\vee,\pi(t)\rangle
\bigm|t\in[0,1]\bigr\}.$$
Let $r\in\{1,\ldots,m+1\}$ and let
$(v_\ell)\in\prod_{\ell=r}^m\mathscr N^\vee(\pi,t_\ell)$.
\begin{enumerate}
\item
\label{it:LePRZCa}
Let $r^+$ be the smallest element in
$$\bigl\{\ell\in\{r,\ldots,m\}\bigm|
(\alpha_i^\vee,p)\in\PACS(\pi,t_\ell)\bigr\},$$
contingent on this set to be nonempty. Then for any
$u\in N^\vee\bigl(\mathfrak f_{\pi(t_r)}\bigr)$ there exists
$(v'_\ell)\in\prod_{\ell=r}^m\mathscr N^\vee(\pi,t_\ell)$ such that
$$v'_r\cdots v'_m\,L_{\wt(\pi)}=u\,v_r\cdots v_m\,L_{\wt(\pi)}$$
and
$$\mathbf a_{i,p}(v'_\ell)=
\begin{cases}
\mathbf a_{i,p}(u)+\mathbf a_{i,p}(v_\ell)&\text{if $\ell=r^+$,}\\
\mathbf a_{i,p}(v_\ell)&\text{for all other $\ell\in\{r,\ldots,m\}$.}
\end{cases}$$
\item
\label{it:LePRZCb}
For any $c\in1+z\mathcal O$ and any $\lambda\in\Lambda$, there exists
$(v'_\ell)\in\prod_{\ell=r}^m\mathscr N^\vee(\pi,t_\ell)$ such that
$$v'_r\cdots v'_m\,L_{\wt(\pi)}=c^\lambda\,v_r\cdots v_m\,L_{\wt(\pi)}$$
and $\mathbf a_{i,p}(v'_\ell)=\mathbf a_{i,p}(v_\ell)$ for all
$\ell\in\{r,\ldots,m\}$.
\item
\label{it:LePRZCc}
For any $b\in\mathbb C$ not in
$$\bigl\{0\bigr\}\cup\bigl\{\mathbf a_{i,p}(v_r)+\cdots+
\mathbf a_{i,p}(v_\ell)\bigm|\ell\in\{r,\ldots,m\}\bigr\},$$
there exists $(v'_\ell)\in\prod_{\ell=r}^m\mathscr N^\vee(\pi,t_\ell)$
such that
$$v'_r\cdots v'_m\,L_{\wt(\pi)}=x_{(-\alpha_i^\vee,\,-p)}(1/b)\;
v_r\cdots v_m\,L_{\wt(\pi)}.$$
\end{enumerate}
\end{lemma}
\begin{proof}
The lemma is trivial for $r=m+1$. Proceeding by decreasing induction,
we choose $r\in\{1,\ldots,m\}$, assume that statements~\ref{it:LePRZCa},
\ref{it:LePRZCb} and~\ref{it:LePRZCc} hold for $r+1$, and show that
they also hold for $r$. We recall (see the proof of
Proposition~\ref{pr:AltDefRingZ}) that
$$N^\vee\bigl(\mathfrak f_{\pi(t_r)}\bigr)=\mathscr N^\vee(\pi,t_r)
\bowtie N^\vee\bigl(\mathfrak f_{\pi(t_r+0)}\bigr)\quad\text{and}\quad
N^\vee\bigl(\mathfrak f_{\pi(t_r+0)}\bigr)\subseteq
N^\vee\bigl(\mathfrak f_{\pi(t_{r+1})}\bigr).$$
Let $(v_\ell)\in\prod_{\ell=r}^m\mathscr N^\vee(\pi,t_\ell)$.

We start with \ref{it:LePRZCa}. Let
$u\in N^\vee\bigl(\mathfrak f_{\pi(t_r)}\bigr)$. We can write
$uv_r\in N^\vee\bigl(\mathfrak f_{\pi(t_r)}\bigr)$ as a product
$v'_ru'$ with $(v'_r,u')\in\mathscr N^\vee(\pi,t_r)\times
N^\vee\bigl(\mathfrak f_{\pi(t_r+0)}\bigr)$. Then
$$\mathbf a_{i,p}(u)+\mathbf a_{i,p}(v_r)=
\mathbf a_{i,p}(v'_r)+\mathbf a_{i,p}(u').$$
Noting that $u'\in N^\vee\bigl(\mathfrak f_{\pi(t_{r+1})}\bigr)$,
we make use of the inductive assumption: there exists
$(v'_\ell)\in\prod_{\ell=r+1}^m\mathscr N^\vee(\pi,t_\ell)$ such that
$$v'_{r+1}\cdots v'_m\,L_{\wt(\pi)}=u'\,v_{r+1}\cdots v_m\,L_{\wt(\pi)}$$
and
$$\mathbf a_{i,p}(v'_\ell)=
\begin{cases}
\mathbf a_{i,p}(u')+\mathbf a_{i,p}(v_\ell)&\text{if $\ell=(r+1)^+$,}\\
\mathbf a_{i,p}(v_\ell)&\text{for all other $\ell\in\{r+1,\ldots,m\}$.}
\end{cases}$$

We distinguish two cases.
If $(\alpha_i^\vee,p)\in\PACS(\pi,t_r)$, then
$\mathfrak f_{\pi(t_r+0)}\not\subseteq H_{(\alpha_i^\vee,\,p)}^-$,
whence $\mathbf a_{i,p}(u')=0$; also $r^+=r$ in this case.
If $(\alpha_i^\vee,p)\notin\PACS(\pi,t_r)$, then
$\mathbf a_{i,p}(v_r)=\mathbf a_{i,p}(v'_r)=0$; here $r^+=(r+1)^+$.
In both cases, routine checks conclude the proof of~\ref{it:LePRZCa}.

We now turn to statement~\ref{it:LePRZCb}. Let $c\in1+z\mathcal O$
and let $\lambda\in\Lambda$. One easily checks that any subgroup of the
form $N^\vee(\mathfrak f)$, in particular
$N^\vee\bigl(\mathfrak f_{\pi(t_r)}\bigr)$, is stable under
conjugation by $c^\lambda$. Additionally, for any
$v\in N^\vee\bigl(\mathfrak f_{\pi(t_r)}\bigr)$, when we write
$$v\equiv\prod_{i\in I}x_{\alpha_i^\vee}(a_i)\ \;\bmod
(N^\vee(\mathcal K),N^\vee(\mathcal K)),$$
the Laurent series $a_i$ has valuation at least $p$; this series is
multiplied by $c^{\langle\alpha_i^\vee,\lambda\rangle}$ when one
conjugates $v$ by $c^\lambda$; looking at the coefficient of $z^p$ then
gives $\mathbf a_{i,p}(v)=\mathbf a_{i,p}(c^\lambda vc^{-\lambda})$.

Write $c^\lambda v_rc^{-\lambda}\in
N^\vee\bigl(\mathfrak f_{\pi(t_r)}\bigr)$ as a product $v'_ru$ with
$(v'_r,u)\in\mathscr N^\vee(\pi,t_r)\times
N^\vee\bigl(\mathfrak f_{\pi(t_r+0)}\bigr)$. Then
$$\mathbf a_{i,p}(v_r)=\mathbf a_{i,p}(c^\lambda v_rc^{-\lambda})=
\mathbf a_{i,p}(v'_r)+\mathbf a_{i,p}(u).$$
By induction, there exists
$(v'_\ell)\in\prod_{\ell=r+1}^m\mathscr N^\vee(\pi,t_\ell)$ such that
$$v'_{r+1}\cdots v'_m\,L_{\wt(\pi)}=uc^\lambda\,v_{r+1}\cdots
v_m\,L_{\wt(\pi)}$$
and
$$\mathbf a_{i,p}(v'_\ell)=
\begin{cases}
\mathbf a_{i,p}(u)+\mathbf a_{i,p}(v_\ell)&\text{if $\ell=(r+1)^+$,}\\
\mathbf a_{i,p}(v_\ell)&\text{for all other $\ell\in\{r+1,\ldots,m\}$.}
\end{cases}$$

Again we distinguish two cases.
If $(\alpha_i^\vee,p)\in\PACS(\pi,t_r)$,
then $\mathfrak f_{\pi(t_r+0)}\not\subseteq H_{(\alpha_i^\vee,\,p)}^-$
and therefore $\mathbf a_{i,p}(u)=0$.
If $(\alpha_i^\vee,p)\notin\PACS(\pi,t_r)$,
then $\mathbf a_{i,p}(v_r)=\mathbf a_{i,p}(v'_r)=0$
and anew $\mathbf a_{i,p}(u)=0$.
Thus, $\mathbf a_{i,p}(u)=0$ holds unconditionally, which concludes
the proof of~\ref{it:LePRZCb}.

Lastly, let us deal with statement~\ref{it:LePRZCc}.
We distinguish three cases.

Suppose first that $(\alpha_i^\vee,p)\in\PACS(\pi,t_r)$.
We write $v_r=x_{(\alpha_i^\vee,\,p)}(a)\,\widetilde v_r$
where $a=\mathbf a_{i,p}(v_r)$ and $\widetilde v_r$ is a
product of elements $x_{\beta}(a_\beta)$ with
$\beta\in\PACS(\pi,t_r)\setminus\{(\alpha_i^\vee,p)\}$
and $a_\beta\in\mathbb C$. From~\eqref{eq:SL2Comm} we get
$$x_{(-\alpha_i^\vee,\,-p)}(1/b)\;x_{(\alpha_i^\vee,\,p)}(a)
=(1-a/b)^{-\alpha_i}\;x_{(\alpha_i^\vee,\,p)}(a(1-a/b))\;
x_{(-\alpha_i^\vee,\,-p)}(1/(b-a)).$$
By Lemma~\ref{le:StabFace},
$$x_{(-\alpha_i^\vee,\,-p)}(1/(b-a))\;\widetilde v_r\;
x_{(-\alpha_i^\vee,\,-p)}(-1/(b-a))$$
belongs to $N^\vee\bigl(\mathfrak f_{\pi(t_r)}\bigr)$;
we write it as a product $\widetilde v'_ru$ with
$(\widetilde v'_r,u)\in\mathscr N^\vee(\pi,t_r)\times
N^\vee\bigl(\mathfrak f_{\pi(t_r+0)}\bigr)$. By induction, there exists
$(v'_\ell)\in\prod_{\ell=r+1}^m\mathscr N^\vee(\pi,t_\ell)$ such that
$$v'_{r+1}\cdots v'_m\,L_{\wt(\pi)}=u\;x_{(-\alpha_i^\vee,\,-p)}
(1/(b-a))\;v_{r+1}\cdots v_m\,L_{\wt(\pi)}.$$
Then
$$x_{(-\alpha_i^\vee,\,-p)}(1/b)\;v_r\cdots v_m\,L_{\wt(\pi)}
=(1-a/b)^{-\alpha_i}\,\bigl[x_{(\alpha_i^\vee,\,p)}(a(1-a/b))\;
\widetilde v'_r\bigr]\,v'_{r+1}\cdots v'_m\,L_{\wt(\pi)}.$$
Denoting the element between square brackets above by $v'_r$,
we get the desired expression, up to the inconsequential left
multiplication by $(1-a/b)^{-\alpha_i}$.

Suppose now that there exists $q>p$ such that
$(\alpha_i^\vee,q)\in\PACS(\pi,t_r)$; then $\mathbf a_{i,p}(v_r)=0$.
We write $v_r=x_{(\alpha_i^\vee,\,q)}(a)\,\widetilde v_r$
where $a\in\mathbb C$ and $\widetilde v_r$ is a
product of elements $x_{\beta}(a_\beta)$ with
$\beta\in\PACS(\pi,t_r)\setminus\{(\alpha_i^\vee,q)\}$
and $a_\beta\in\mathbb C$. Let $c$ be a square root in $1+t\mathcal O$
of $1-(a/b)t^{q-p}$. From~\eqref{eq:SL2Comm} we get
$$x_{(-\alpha_i^\vee,\,-p)}(1/b)\;x_{(\alpha_i^\vee,\,q)}(a)
=c^{-\alpha_i}\;x_{(\alpha_i^\vee,\,q)}(a)\;
x_{(-\alpha_i^\vee,\,-p)}(1/b)\;c^{-\alpha_i}.$$
By Lemma~\ref{le:StabFace},
$$x_{(-\alpha_i^\vee,\,-p)}(1/b)\,\bigl(c^{-\alpha_i}\,
\widetilde v_r\,c^{\alpha_i}\bigr)\,x_{(-\alpha_i^\vee,\,-p)}(-1/b)$$
belongs to $N^\vee\bigl(\mathfrak f_{\pi(t_r)}\bigr)$;
we write it as a product $\widetilde v'_ru$ with
$(\widetilde v'_r,u)\in\mathscr N^\vee(\pi,t_r)\times
N^\vee\bigl(\mathfrak f_{\pi(t_r+0)}\bigr)$. By induction, there exists
$(v'_\ell)\in\prod_{\ell=r+1}^m\mathscr N^\vee(\pi,t_\ell)$ such that
$$v'_{r+1}\cdots v'_m\,L_{\wt(\pi)}=u\;x_{(-\alpha_i^\vee,\,-p)}
(1/b)\;c^{-\alpha_i}\,v_{r+1}\cdots v_m\,L_{\wt(\pi)}.$$
Then
$$x_{(-\alpha_i^\vee,\,-p)}(1/b)\;v_r\cdots v_m\,L_{\wt(\pi)}
=c^{-\alpha_i}\,\bigl[x_{(\alpha_i^\vee,\,q)}(a)\;
\widetilde v'_r\bigr]\,v'_{r+1}\cdots v'_m\,L_{\wt(\pi)}.$$
Denoting the element between square brackets above by $v'_r$,
we get the desired expression, up to the inopportune left
multiplication by $c^{-\alpha_i}$; the latter can however be
wiped off by a further use of the inductive assumption.

Last, suppose that no affine coroot of direction $\alpha_i^\vee$
occurs in $\PACS(\pi,t_r)$; then $\mathbf a_{i,p}(v_r)=0$.
By Lemma~\ref{le:StabFace},
$$x_{(-\alpha_i^\vee,\,-p)}(1/b)\;v_r\;
x_{(-\alpha_i^\vee,\,-p)}(-1/b)$$
belongs to $N^\vee\bigl(\mathfrak f_{\pi(t_r)}\bigr)$;
we write it as a product $v'_ru$ with
$(v'_r,u)\in\mathscr N^\vee(\pi,t_r)\times
N^\vee\bigl(\mathfrak f_{\pi(t_r+0)}\bigr)$. By induction, there exists
$(v'_\ell)\in\prod_{\ell=r+1}^m\mathscr N^\vee(\pi,t_\ell)$ such that
$$v'_{r+1}\cdots v'_m\,L_{\wt(\pi)}=u\;x_{(-\alpha_i^\vee,\,-p)}
(1/b)\;v_{r+1}\cdots v_m\,L_{\wt(\pi)}.$$
Then
$$x_{(-\alpha_i^\vee,\,-p)}(1/b)\;v_r\cdots v_m\,L_{\wt(\pi)}
=v'_r\,v'_{r+1}\cdots v'_m\,L_{\wt(\pi)},$$
as desired, which concludes the proof of~\ref{it:LePRZCc}.
\end{proof}

Let us now consider $i\in I$ and two integral paths $\pi$ and $\eta$
related by the equation $\eta=\tilde e_i\pi$. We denote by~$p$ the
minimum of the function $t\mapsto\langle\alpha_i^\vee,\pi(t)\rangle$
over the interval $[0,1]$ and by $a$ and $b$ the two points in time
where $\pi$ is bent to produce $\eta$. Noting that the conditions
spelled out in sect.~\ref{ss:RecPathMod} do not uniquely determine
$b$, we choose it to be the largest possible: either $b=1$ or
$\langle\alpha_i^\vee,\pi(b+h)\rangle>p$ for all small enough $h>0$.

Let $(t_1,\ldots,t_m)$ be the ordered list of all elements in $[0,1[$
such that $\PACS(\pi,t)\neq\varnothing$. We set $t_{m+1}=1$.
The set $\PACS(\pi,a)$ may be empty; if this happens, we insert $a$ in
the list $(t_1,\ldots,t_m)$, for it will simplify the notation
hereafter. On the contrary, the above condition imposed on $b$ ensures
that either $b=1$ or $(\alpha_i^\vee,p)\in\PACS(\pi,b)$, so
$b$ automatically appears in the list $(t_1,\ldots,t_{m+1})$. We denote
by $r$ and $s$ the indices in $\{1,\ldots,m+1\}$ such that $a=t_r$ and
$b=t_s$; by design $t_r=a<t_{r+1}\leq t_s=b$.

\begin{lemma}
\label{le:PRZFold}
In the setting described in the two paragraphs above, let
$(v_\ell)\in\prod_{\ell=1}^m\mathscr N^\vee(\pi,t_\ell)$.
Assume that $\mathbf a_{i,p}(v_s)+\cdots+\mathbf a_{i,p}(v_\ell)\neq0$
for each $\ell\in\{s,\ldots,m\}$. Then for any $h\in\mathbb C^\times$,
one can construct $(w_\ell)\in\prod_{\ell=1}^m\mathscr N^\vee(\eta,t_\ell)$
such that
$$v_1\cdots v_{r-1}\;x_{(-\alpha_i^\vee,\,-p-1)}(h)\;v_r\cdots
v_m\,L_{\wt(\pi)}=w_1\ldots w_m\,L_{\wt(\eta)}.$$
\end{lemma}
\begin{proof}
Let $(v_\ell)$ be as in the statement and let $h\in\mathbb C^\times$.
We set
$$A=v_1\cdots v_{r-1}\quad\text{and}\quad B=v_r\cdots v_m.$$
We note that $\mathfrak f_{\pi(t_r)}\subseteq H_{(\alpha_i^\vee,\,p+1)}$,
so $x_{(\alpha_i^\vee,\,p+1)}(-1/h)\in
N^\vee\bigl(\mathfrak f_{\pi(t_r)}\bigr)$.

Using Lemma~\ref{le:PRZCtrl}~\ref{it:LePRZCa}, we find
$(v'_{r+1},\ldots,v'_m)\in\prod_{\ell=r}^m\mathscr N^\vee(\pi,t_\ell)$
such that
$$x_{(\alpha_i^\vee,\,p+1)}(-1/h)\;B\,L_{\wt(\pi)}=v'_r\cdots
v'_m\,L_{\wt(\pi)}$$
and $\mathbf a_{i,p}(v'_\ell)=\mathbf a_{i,p}(v_\ell)$ for all
$\ell\in\{r,\ldots,m\}$. We set $c=\mathbf a_{i,p}(v_s)$ and write
$v'_s=x_{(\alpha_i^\vee,\,p)}(c)\,\widetilde v'_s$;
then $\widetilde v'_s\in\mathscr N^\vee(\pi,t_s)$
and $\mathbf a_{i,p}(\widetilde v'_s)=0$. We also set
$$C=v'_r\cdots v'_{s-1}\quad\text{and}\quad
D=\widetilde v'_s\,v'_{s+1}\cdots v'_m.$$
Using Lemma~\ref{le:PRZCtrl}~\ref{it:LePRZCc}, we find
$(\widetilde v''_s,v''_{s+1},\ldots,v''_m)\in
\prod_{\ell=s}^m\mathscr N^\vee(\pi,t_\ell)$ such that
$$x_{(-\alpha_i^\vee,\,-p)}(-1/c)\;D\,L_{\wt(\pi)}=
\widetilde v''_s\,v''_{s+1}\cdots v''_m\,L_{\wt(\pi)}.$$
Last, we set
\begin{align*}
E&=x_{(\alpha_i^\vee,\,p)}(c)\;x_{(-\alpha_i^\vee,\,-p)}(1/c)\;
x_{(\alpha_i^\vee,\,p)}(c),\\[2pt]
F&=x_{(\alpha_i^\vee,\,p)}(-c)\;\widetilde v''_s\,v''_{s+1}\cdots
v''_m,\\[2pt]
K&=x_{(-\alpha_i^\vee,\,-p-1)}(h)\;x_{(\alpha_i^\vee,\,p+1)}(1/h).
\end{align*}
Then
\begin{equation}
\label{eq:PRZFold1}
A\;x_{(-\alpha_i^\vee,\,-p-1)}(h)\;B\,L_{\wt(\pi)}=
AKCEF\,L_{\wt(\pi)}.
\end{equation}

Observing that
$$\PACS(\eta,t_\ell)=\begin{cases}
\PACS(\pi,t_\ell)&\text{if $1\leq\ell<r$,}\\[2pt]
\{(\alpha_i^\vee,p+1)\}\sqcup s_{(\alpha_i^\vee,\,p+1)}
\bigl(\PACS(\pi,t_r)\bigr)&\text{if $\ell=r$,}\\[2pt]
s_{(\alpha_i^\vee,\,p+1)}\bigl(\PACS(\pi,t_\ell)\bigr)&
\text{if $r<\ell<s$,}\\[2pt]
\tau_{\alpha_i}\bigl(\PACS(\pi,t_\ell)\bigr)&
\text{if $s\leq\ell\leq m$,}
\end{cases}$$
we check that the sequence
\begin{multline}
\label{eq:PRZFold2}
\Bigl(v_1,\;\ldots,\;v_{r-1},\;x_{(\alpha_i^\vee,\,p+1)}(-h)\;
\bigl(z^{(p+1)\alpha_i}\overline{s_i}\,\bigr)\;v'_r\;
\bigl(z^{(p+1)\alpha_i}\overline{s_i}\,\bigr)^{-1},\\[4pt]
\bigl(z^{(p+1)\alpha_i}\overline{s_i}\,\bigr)\;v'_{r+1}\;
\bigl(z^{(p+1)\alpha_i}\overline{s_i}\,\bigr)^{-1},\;
\ldots,\,\bigl(z^{(p+1)\alpha_i}\overline{s_i}\,\bigr)
\;v'_{s-1}\;\bigl(z^{(p+1)\alpha_i}\overline{s_i}\,\bigr)^{-1},\\[4pt]
z^{\alpha_i}\;x_{(\alpha_i^\vee,\,p)}(-c)\;\widetilde v''_s\;
z^{-\alpha_i},\;z^{\alpha_i}\;v''_{s+1}\;z^{-\alpha_i},\;\ldots,\;
z^{\alpha_i}\;v''_m\;z^{-\alpha_i}\Bigr)
\end{multline}
belongs to $\prod_{\ell=1}^m\mathscr N^\vee(\eta,t_\ell)$. In
addition, the product of the elements in this sequence is
$$A\;x_{(\alpha_i^\vee,\,p+1)}(-h)\;\bigl(z^{(p+1)\alpha_i}\overline{s_i}\,
\bigr)\;C\;\bigl(z^{(p+1)\alpha_i}\overline{s_i}\,\bigr)^{-1}\;
z^{\alpha_i}\,F\,z^{-\alpha_i}.$$

We now apply two transformations to the sequence \eqref{eq:PRZFold2}:
we conjugate the last $m-s+1$ terms by $(-c)^{-\alpha_i}$, and
we conjugate the last $m-r+1$ by $h^{-\alpha_i}$. The resulting
sequence, denoted by $(w_\ell)$, still belongs to
$\prod_{\ell=1}^m\mathscr N^\vee(\eta,t_\ell)$, because all our
constructions are $T^\vee(\mathbb C)$-equivariant.

Observing that
$$K=h^{-\alpha_i}\,x_{(\alpha_i^\vee,\,p+1)}(-h)\;\bigl(z^{(p+1)\alpha_i}
\overline{s_i}\,\bigr)\quad\text{and}\quad
E=\bigl(z^{(p+1)\alpha_i}\overline{s_i}\,\bigr)^{-1}\;
(-c)^{-\alpha_i}\;z^{\alpha_i}$$
(see equation~\eqref{eq:TitsRela}), we obtain
$$w_1\cdots w_m=AKCEF\,z^{-\alpha_i}\,(-ch)^{\alpha_i},$$
and a comparison with \eqref{eq:PRZFold1} yields
$$A\;x_{(-\alpha_i^\vee,\,-p-1)}(h)\;B\,L_{\wt(\pi)}
=AKCEF\,z^{-\alpha_i}\,L_{\wt(\eta)}
=w_1\cdots w_m\,L_{\wt(\eta)},$$
as desired.
\end{proof}

We can now prove Proposition~\ref{pr:PropRingZ}~\ref{it:PrPRZd}.
We consider the situation
$$\eta_1\otimes\cdots\otimes\eta_n=
\tilde e_i(\pi_1\otimes\cdots\otimes\pi_n)$$
in the crystal $\Pi^{\otimes n}$, and our aim is to show that
$\mathring{\mathbf Z}(\pi_1\otimes\cdots\otimes\pi_n)$ is
contained in the closure of
$\mathring{\mathbf Z}(\eta_1\otimes\cdots\otimes\eta_n)$ in $\Gr_n$.

As in the proof of Proposition~\ref{pr:PropRingZ}~\ref{it:PrPRZc},
we regard the concatenation $\pi=\pi_1*\cdots*\pi_n$ as a map from
$[0,n]$ to $\Lambda_{\mathbb R}$, each path $\pi_1$, \dots, $\pi_n$ being
travelled at nominal speed, and ditto for $\eta=\eta_1*\cdots*\eta_n$.
Thus, for each $j\in\{1,\ldots,n\}$ the restriction of $\pi$ to the
interval $[j-1,j]$ is $\pi_j$, up to the obvious shifts in time and space.

By Proposition~\ref{pr:PathModel}~\ref{it:PrPMb}, we have
$\eta=\tilde e_i\pi$. We denote by $a$ and $b$ the two points in time
where $\pi$ is bent to produce $\eta$. Let $(t_1,\ldots,t_m)$ be the
ordered list of all elements in $[0,n[$ such that
$\PACS(\pi,t)\neq\varnothing$. We insert $a$ in this
list if it does not already appear there. We set $t_0=0$ and $t_{m+1}=n$.
We denote by $r$ and $s$ the indices in $\{1,\ldots,m+1\}$ such that
$a=t_r$ and $b=t_s$.

There is a unique integer $k\in\{1,\ldots,n\}$ such that $a$ and
$b$ both belong to $[k-1,k]$; plainly, $\eta_k=\tilde e_i\pi_k$
and $\eta_j=\pi_j$ for all $j\in\{1,\ldots,n\}\setminus\{k\}$. We
record that $\eta_1*\cdots*\eta_j=\tilde e_i(\pi_1*\cdots*\pi_j)$
if $j\in\{k,\ldots,n\}$.

For $j\in\{1,\ldots,n\}$, we set $\nu_j=\wt(\pi_1)+\cdots+\wt(\pi_j)$
and denote by $m_j$ the largest element
$\ell\in\{0,\ldots,m\}$ such that $t_\ell\in[0,j[$.
Then $\mathring{\mathbf Z}(\pi_1\otimes\cdots\otimes\pi_n)$ is
the set of all elements
\begin{equation}
\label{eq:ProofPrPRZd1}
\left[\left(\prod_{\ell=1\vphantom{m_1}}^{m_1}v_\ell\right)z^{\nu_1},\;
z^{-\nu_1}\left(\prod_{\ell=m_1+1}^{m_2}v_\ell\right)z^{\nu_2},\;
\ldots,\;
z^{-\nu_{n-1}}\left(\prod_{\ell=m_{n-1}+1}^{m_n}v_\ell\right)z^{\nu_n}
\right]
\end{equation}
with $(v_\ell)\in\prod_{\ell=1}^m\mathscr N^\vee(\pi,t_\ell)$.

Now assume that $(v_\ell)$ is chosen so that
$\mathbf a_{i,p}(v_s)+\cdots+\mathbf a_{i,p}(v_\ell)\neq0$
for each $\ell\in\{s,\ldots,m\}$ and pick $h\in\mathbb C^\times$.
Lemma~\ref{le:PRZFold} provides us with a sequence
$(w_\ell)\in\prod_{\ell=1}^m\mathscr N^\vee(\eta,t_\ell)$ such that
$$v_1\cdots v_{r-1}\;x_{(-\alpha_i^\vee,\,-p-1)}(h)\;v_r\cdots
v_m\,L_{\wt(\pi)}=w_1\cdots w_m\,L_{\wt(\eta)}.$$
However $(w_\ell)$ satisfies more equations: for $j\in\{1,\ldots,n\}$,
we have
\begin{equation}
\label{eq:ProofPrPRZd2}
\left\{\begin{alignedat}2
v_1\cdots v_{m_j}\,L_{\nu_j}&=w_1\cdots w_{m_j}\,L_{\nu_j}
&&\text{if $j<k$,}\\[2pt]
v_1\cdots v_{r-1}\;x_{(-\alpha_i^\vee,\,-p-1)}(h)\;v_r\cdots
v_{m_j}\,L_{\nu_j}&=w_1\cdots w_{m_j}\,L_{\nu_j+\alpha_i}\quad\
&&\text{if $j\geq k$,}
\end{alignedat}\right.
\end{equation}
in the first case because $w_\ell=v_\ell$ for all
$\ell\in\{1,\ldots,m_{k-1}\}$, in the second case because
Lemma~\ref{le:PRZFold} would have returned the subsequence
$(w_\ell)_{1\leq\ell\leq m_j}$ if we had fed it with the paths
$\pi_1*\cdots*\pi_j$ and $\eta_1*\cdots*\eta_j$
and the datum $(v_\ell)_{1\leq\ell\leq m_j}$ and $h$.

The system \eqref{eq:ProofPrPRZd2} translates to a single equation
in $\Gr_n$, which manifests that the element obtained by inserting
$x_{(-\alpha_i^\vee,\,-p-1)}(h)$ just before $v_r$ in
\eqref{eq:ProofPrPRZd1} belongs to
$\mathring{\mathbf Z}(\eta_1\otimes\cdots\otimes\eta_n)$.
Letting $h$ tend to $0$, we conclude that \eqref{eq:ProofPrPRZd1}
lies in the closure of this set.
To be sure, this conclusion has been reached under the assumption
that $\mathbf a_{i,p}(v_s)+\cdots+\mathbf a_{i,p}(v_\ell)\neq0$
for each $\ell\in\{s,\ldots,m\}$, but this restriction can be
removed by a small perturbation of $\mathbf a_{i,p}(v_s)$.

Thus, Proposition~\ref{pr:PropRingZ}~\ref{it:PrPRZd} is, at last,
fully proven.

\section{Comparison with the tensor product basis}
\label{se:CompTPB}
We keep the notation from sect.~\ref{se:MVCycBas}. Let
$\bm\lambda=(\lambda_1,\ldots,\lambda_n)$ in $(\Lambda^+)^n$. The
tensor product $V(\bm\lambda)$ can be endowed on the one hand
with its MV basis (sect.~\ref{ss:MVBas}), on the other hand with the
tensor product of the MV bases of the factors $V(\lambda_1)$, \dots,
$V(\lambda_n)$. In this section, we compare these two bases through
the explicit identification
$$F(\mathscr I_{\lambda_1}*\cdots*\mathscr I_{\lambda_n})\cong
F(\mathscr I_{\lambda_1})\otimes\cdots\otimes
F(\mathscr I_{\lambda_n})$$
afforded by Beilinson and Drinfeld's fusion product. We show that
the transition matrix is upper unitriangular and that its entries are
intersection multiplicities. The order relation needed to convey
the triangularity involves the inclusion of cycles.

\subsection{Deformations}
\label{ss:Deform}
The Beilinson-Drinfeld Grassmannian $\BDGr{}$ is a relative version of
the affine Grassmannian where the base is the space of effective divisors
on a smooth curve. The choice of the affine line amply satisfies our
needs and offers three advantages: there is a natural global coordinate
on $\mathbb A^1$, every $G$-torsor on $\mathbb A^1$ is trivializable,
and the monodromy of any local system is trivial. Rather than looking for
more generality, we will pragmatically stick with this choice. Consistent
with sect.~\ref{se:MVCycBas}, the coordinate on $\mathbb A^1$ is denoted
by $z$.

Formally, the Beilinson-Drinfeld Grassmannian $\BDGr n$ is defined as
the functor on the category of commutative $\mathbb C$-algebras that
assigns to an algebra $R$ the set of isomorphism classes of triples
$(x_1,\ldots,x_n;\mathcal F,\beta)$, where
$(x_1,\ldots,x_n)\in\mathbb A^n(R)$, $\mathcal F$ is a $G^\vee$-torsor
over $\mathbb A^1_R$ and $\beta$ is a trivialization of $\mathcal F$
away from the points $x_1$, \dots, $x_n$ (\cite{BeilinsonDrinfeld},
sect.~5.3.10; \cite{Richarz}, Definition~3.3; \cite{Zhu},
Definition~3.1.1). We denote by $\pi:\BDGr n\to\mathbb A^n$ the morphism
to the base, which forgets $\mathcal F$ and $\beta$. It is known that
$\BDGr n$ is representable by an ind-scheme and that $\pi$ is ind-proper.

We are only interested in the set of $\mathbb C$-points, endowed with
its ind-variety structure. Using a trivialization of $\mathcal F$, we can
thus adopt the following simplified definition: $\BDGr n$ is the set of
pairs $(x_1,\ldots,x_n;[\beta])$, where $(x_1,\ldots,x_n)\in\mathbb C^n$
and $[\beta]$ belongs to the homogenenous space
$$G^\vee\bigl(\mathbb C\bigl[z,(z-x_1)^{-1},\ldots,(z-x_n)^{-1}\bigr]
\bigr)\,/\,G^\vee\bigl(\mathbb C[z]\bigr).$$
This set is endowed with the structure of an ind-variety.

\begin{other}{Example}
\label{ex:BDGrIndVar}
(\cite{BeilinsonDrinfeld}, Remark in sect.~5.3.10.)
We consider the case $G^\vee=\GL_N$. Here the datum of $[\beta]$ is
equivalent to the datum of the $\mathbb C[z]$-lattice $\beta(L_0)$ in
$\mathbb C(z)^N$, where $L_0=\mathbb C[z]^N$ is the standard lattice.
Let us write $\mathbf x$ for the point $(x_1,\ldots,x_n)$ and set
$f_{\mathbf x}=(z-x_1)\cdots(z-x_n)$; then a lattice $L$ is of this form
$\beta(L_0)$ if and only if there exists a positive integer $k$ such
that $f_{\mathbf x}^kL_0\subseteq L\subseteq f_{\mathbf x}^{-k}L_0$.
For each positive integer $k$, define $\bigl(\BDGr n\bigr){}^{}_k$ to be
the subset of $\BDGr n$ consisting of all pairs $(\mathbf x;L)$ with
$f_{\mathbf x}^kL_0\subseteq L\subseteq f_{\mathbf x}^{-k}L_0$.
We identify $\mathbb C[z]/(f_{\mathbf x}^{2k})$ with the vector space
$V$ of polynomials of degree strictly less than $2kn$, and
subsequently identify $L_0/f_{\mathbf x}^{2k}L_0$ with $V^N$. The space
$\bigl(\BDGr n\bigr){}^{}_k$ can then be realized as a Zariski-closed
subset~of
$$\mathbb C^n\times\bigcup_{d=0}^{2knN}\mathbb G_d\bigl(V^N\bigr)$$
where $\mathbb G_d\bigl(V^N\bigr)$ denotes the Grassmannian of $d$-planes
in $V^N$. In this way, $\BDGr n$ is the inductive limit of a system
of algebraic varieties and closed embeddings, in other words, an
ind-variety.
\end{other}

We also want to deform the $n$-fold convolution variety $\Gr_n$.
Accordingly, we define $\BDConv_n$ as the set of pairs
$(x_1,\ldots,x_n;[\beta_1,\ldots,\beta_n])$, where
$(x_1,\ldots,x_n)\in\mathbb C^n$ and $[\beta_1,\ldots,\beta_n]$ belongs to
$$G^\vee\bigl(\mathbb C\bigl[z,(z-x_1)^{-1}\bigr]\bigr)
\,\times^{G^\vee(\mathbb C[z])}\;\cdots\,\times^{G^\vee(\mathbb C[z])}\;
G^\vee\bigl(\mathbb C\bigl[z,(z-x_n)^{-1}\bigr]\bigr)
\,/\,G^\vee\bigl(\mathbb C[z]\bigr)$$
(see \cite{Richarz}, Definition~3.8, or \cite{Zhu}, (3.1.21)). This set
$\BDConv_n$ is endowed with the structure of an ind-variety; it comes
with a map $m_n:\BDConv_n\to\BDGr n$ defined by
$$m_n(x_1,\ldots,x_n;[\beta_1,\ldots,\beta_n])=
(x_1,\ldots,x_n;[\beta_1\cdots\beta_n]).$$

\begin{other}{Example}
\label{ex:BDConvIndVar}
We again consider the case $G^\vee=\GL_N$. Then an element in $\BDConv_n$
is the datum of a point $(x_1,\ldots,x_n)\in\mathbb C^n$
and a sequence $(L_1,\ldots,L_n)$ of $\mathbb C[z]$-lattices in
$\mathbb C(z)^N$ for which there exists a positive integer $k$ such that
$$(z-x_j)^kL_{j-1}\subseteq L_j\subseteq(z-x_j)^{-k}L_{j-1}$$
for all $j\in\{1,\ldots,n\}$; here again $L_0=\mathbb C[z]^N$ is
the standard lattice and $L_j=(\beta_1\cdots\beta_j)(L_0)$.
\end{other}

In the above example, we can partition $\BDConv_n$ into cells by specifying
the relative positions of the pairs of lattices $(L_{j-1},L_j)$ in terms
of invariant factors. This construction can be generalized to an arbitrary
group $G$ as follows: given $\bm\lambda=(\lambda_1,\ldots,\lambda_n)$ in
$(\Lambda^+)^n$, we define $\BDConv_n^{\bm\lambda}$ as the subset of
$\BDConv_n$ consisting of all pairs
$(x_1,\ldots,x_n;[\beta_1,\ldots,\beta_n])$ with
$$\beta_j\in G^\vee(\mathbb C[z])\;(z-x_j)^{\lambda_j}\;
G^\vee(\mathbb C[z])$$
for $j\in\{1,\ldots,n\}$. The Cartan decomposition
$$G^\vee\bigl(\mathbb C\bigl[z,(z-x_j)^{-1}\bigr]\bigr)=
\bigsqcup_{\lambda_j\in\Lambda^+}G^\vee(\mathbb C[z])\;
(z-x_j)^{\lambda_j}\;G^\vee(\mathbb C[z])$$
yields
$$\BDConv_n=\bigsqcup_{\bm\lambda\in(\Lambda^+)^n}\BDConv_n^{\bm\lambda}$$
and it can be checked that
\begin{equation}
\label{eq:ClosureStratBD}
\overline{\BDConv_n^{\bm\lambda}}=
\bigsqcup_{\substack{\bm\mu\in(\Lambda^+)^n\\[2pt]
\mu_1\leq\lambda_1,\,\ldots,\,\mu_n\leq\lambda_n}}\BDConv_n^{\bm\mu}.
\end{equation}

In addition, the maps $(x_1,\ldots,x_j;[\beta_1,\ldots,\beta_j])\mapsto
(x_1,\ldots,x_{j-1};[\beta_1,\ldots,\beta_{j-1}])$
exhibit $\BDConv_n^{\bm\lambda}$ as the total space of an iterated
fibration with base $\BDConv_1^{\lambda_1}$ and successive fibers
$\BDConv_1^{\lambda_2}$, \dots, $\BDConv_1^{\lambda_n}$; it follows
that $\BDConv_n^{\bm\lambda}$ is a smooth connected variety of
dimension $2\rho(|\bm\lambda|)+n$.

Let us now investigate the fibers of the map
$\pi\circ m_n:\BDConv_n\to\mathbb C^n$.
Given $x\in\mathbb C$, we set $\mathcal O_x=\mathbb C[\![z-x]\!]$ and
$\mathcal K_x=\mathbb C(\!(z-x)\!)$; thus $\mathcal K_x$ is the completion
of $\mathbb C(z)$ at the place defined by $x$. We identify $\mathcal O$
and $\mathcal K$ with $\mathcal O_x$ and $\mathcal K_x$ by means of the
map $z\mapsto z-x$.

We fix $\mathbf x=(x_1,\ldots,x_n)$ in $\mathbb C^n$. Let
$\supp(\mathbf x)$ be the set of values $y\in\mathbb C$ that appear
in the tuple~$\mathbf x$. For $y\in\supp(\mathbf x)$, denote by $m_y$
the number of indices $j\in\{1,\ldots,n\}$ such that $x_j=y$ and
choose an increasing sequence $(p_0=0,p_1,p_2,\ldots,p_{m_y}=n)$ in a
way that each interval $[p_{k-1}+1,p_k]$ contains exactly one index $j$
such that $x_j=y$. For $\bm\beta=[\beta_1,\ldots,\beta_n]$ in the fiber
$$(\BDConv_n)_{\mathbf x}=
G^\vee\bigl(\mathbb C\bigl[z,(z-x_1)^{-1}\bigr]\bigr)
\,\times^{G^\vee(\mathbb C[z])}\;\cdots\,\times^{G^\vee(\mathbb C[z])}\;
G^\vee\bigl(\mathbb C\bigl[z,(z-x_n)^{-1}\bigr]\bigr)
\,/\,G^\vee\bigl(\mathbb C[z]\bigr),$$
we define $\Theta(\bm\beta)_y$ as the point $[(\beta_1\cdots\beta_{p_1}),
(\beta_{p_1+1}\cdots\beta_{p_2}),\ldots,(\beta_{p_{m_y-1}+1}\cdots\beta_n)]$
in
$$\underbrace{G^\vee(\mathcal K_y)\,\times^{G^\vee(\mathcal O_y)}\;\cdots\,
\times^{G^\vee(\mathcal O_y)}\;G^\vee(\mathcal K_y)}_{m_y\ \text{factors}\
G^\vee(\mathcal K_y)}\,/\,G^\vee(\mathcal O_y)\cong\Gr_{m_y}$$
(note that $\Theta(\bm\beta)_y$ does not depend on this choice,
because $\beta_j\in G^\vee(\mathcal O_y)$ if $x_j\neq y$).

\begin{proposition}
\label{pr:FibBDConv}
The map $\bm\beta\mapsto(\Theta(\bm\beta)_y)$
is a bijection
$$(\BDConv_n)_{\mathbf x}\xrightarrow\simeq
\prod_{y\in\supp(\mathbf x)}\Gr_{m_y}.$$
\end{proposition}
\begin{proof}
Combining the Iwasawa decomposition~\eqref{eq:IwasDec} with the easily
proven equality
\begin{equation}
\label{eq:UniPolyLoop}
N^\vee(\mathcal K_x)=
N^\vee\bigl(\mathbb C\bigl[z,(z-x)^{-1}\bigr]\bigl)\,N^\vee(\mathcal O_x),
\end{equation}
we obtain the well-known equality
$$G^\vee(\mathcal K_x)=G^\vee\bigl(\mathbb C\bigl[z,(z-x)^{-1}\bigr]
\bigr)\,G^\vee(\mathcal O_x),$$
for each $x\in\mathbb C$.

The case $n=1$ of the proposition is banal. Assume that $n\geq2$,
and for $y\in\supp(\mathbf x)$, pick $\bm\gamma_y\in\Gr_{m_y}$.
Set $\mathbf x'=(x_1,\ldots,x_{n-1})$ and $m=m_{x_n}$, write
$\bm\gamma_{x_n}=[\gamma_1,\ldots,\gamma_m]$. Reasoning by
induction on $n$, we know that there is a unique
$\bm\beta'=[\beta_1,\ldots,\beta_{n-1}]$ in
$(\BDConv_{n-1})_{\mathbf x'}$ such that
$$\Theta({\bm\beta'})_y=
\begin{cases}
\bm\gamma_y&\text{if $x_n\neq y$},\\
[\gamma_1,\ldots,\gamma_{m-1}]&\text{if $x_n=y$}.
\end{cases}$$
The elements $\gamma_1$, \dots, $\gamma_m$ belong to $G^\vee(\mathcal K)$,
which we identify to $G^\vee(\mathcal K_{x_n})$. We choose
$\beta_n\in G^\vee\bigl(\mathbb C\bigl[z,(z-x_n)^{-1}\bigr]\bigr)$
such that
$$(\beta_1\ldots\beta_{n-1})^{-1}(\gamma_1\ldots\gamma_m)
\in\beta_n\,G^\vee(\mathcal O_{x_n}).$$
Then $[\beta_1,\ldots,\beta_{n-1},\beta_n]$ is the unique element
$\bm\beta$ in $(\BDConv_n)_{\mathbf x}$ such that
$\Theta({\bm\beta})_y=\bm\gamma_y$ for all $y$.
\end{proof}

Keep the notation above for $\mathbf x$ and the integers $m_y$ and
let $\bm\lambda=(\lambda_1,\dots,\lambda_n)$ in $(\Lambda^+)^n$. For
each $y\in\supp(\mathbf x)$, define $\bm\lambda_y\in(\Lambda^+)^{m_y}$
as the ordered tuple formed by the weights $\lambda_j$, for
$j\in\{1,\ldots,n\}$ such that $x_j=y$. Then, under the bijection
given in Proposition~\ref{pr:FibBDConv}, the fiber
$\bigl(\BDConv_n^{\bm\lambda}\bigr){}^{}_{\mathbf x}$ identifies with
$$\prod_{y\in\supp(\mathbf x)}\Gr_{m_y}^{\bm\lambda_y}.$$

\subsection{Global cycles}
\label{ss:GloCyc}
Recall our notation $N^\vee$ for the unipotent radical of $B^\vee$.
For $\mu\in\Lambda$ and $x\in\mathbb C$, we define
$$\widetilde S_{\mu|x}=
(z-x)^\mu\,N^\vee\bigl(\mathbb C\bigl[z,(z-x)^{-1}\bigr]\bigr)=
N^\vee\bigl(\mathbb C\bigl[z,(z-x)^{-1}\bigr]\bigr)\,(z-x)^\mu.$$
Equation \eqref{eq:UniPolyLoop} expresses that the natural map
$$N^\vee\bigl(\mathbb C\bigl[z,(z-x)^{-1}\bigr]\bigr)\,/\,
N^\vee(\mathbb C[z])\to N^\vee(\mathcal K_x)/N^\vee(\mathcal O_x)$$
is bijective; composing with the natural map
$N^\vee(\mathcal K)/N^\vee(\mathcal O)\to\Gr$, we obtain, after left
multiplication by $(z-x)^\mu$, a bijection
$$\widetilde S_{\mu|x}\,/\,N^\vee(\mathbb C[z])\xrightarrow\simeq S_\mu.$$
For $(\mu_1,\ldots,\mu_n)\in\Lambda^n$, let
$S_{\mu_1}\caltimes\cdots\caltimes S_{\mu_n}$ be the set
of all pairs $(x_1,\ldots,x_n;[\beta_1,\ldots,\beta_n])$ with
$(x_1,\ldots,x_n)$ in $\mathbb C^n$ and $[\beta_1,\ldots,\beta_n]$ in
$$\widetilde S_{\mu_1|x_1}\,\times^{N^\vee(\mathbb C[z])}\;\cdots\,
\times^{N^\vee(\mathbb C[z])}\;\widetilde S_{\mu_n|x_n}\ /\,
N^\vee(\mathbb C[z]).$$
Rewriting the Iwasawa decomposition as
$$G^\vee\bigl(\mathbb C\bigl[z,(z-x)^{-1}\bigr]\bigr)=
\bigsqcup_{\mu\in\Lambda}N^\vee\bigl(\mathbb C
\bigl[z,(z-x)^{-1}\bigr]\bigl)\;(z-x)^\mu\;G^\vee(\mathbb C[z]),$$
we then see that the natural map
$$\Uppsi:\bigsqcup_{(\mu_1,\ldots,\mu_n)\in\Lambda^n}
S_{\mu_1}\caltimes\cdots\caltimes S_{\mu_n}\to\BDConv_n$$
is bijective. Here $\Uppsi$ is regarded as the calligraphic variant of
the letter $\Psi$ used in sect.~\ref{ss:MVCyc}; these two glyphs may
be hard to distinguish, but hopefully this choice will not lead to any
confusion.

More generally, given $(\mu_1,\ldots,\mu_n)\in\Lambda^n$ and
$N^\vee(\mathcal O)$-stable subsets $Z_1\subseteq S_{\mu_1}$, \dots,
$Z_n\subseteq S_{\mu_n}$, we define $Z_1\caltimes\cdots\caltimes Z_n$
to be the subset of all pairs $(x_1,\ldots,x_n;[\beta_1,\ldots,\beta_n])$
with $(x_1,\ldots,x_n)\in\mathbb C^n$ and
$$[\beta_1,\ldots,\beta_n]\in
\widetilde Z_{1|x_1}\,\times^{N^\vee(\mathbb C[z])}\;\cdots\,
\times^{N^\vee(\mathbb C[z])}\;\widetilde Z_{n|x_n}\ /\,
N^\vee(\mathbb C[z])$$
where each $\widetilde Z_{j|x_j}$ is the preimage of $Z_j$ under
the map $\widetilde S_{\mu_j|x_j}\to S_{\mu_j}$.

For $\mu\in\Lambda$, we define
$$\dot S_\mu=\bigcup_{\substack{(\mu_1,\ldots,\mu_n)\in\Lambda^n\\[2pt]
\mu_1+\cdots+\mu_n=\mu}}\Uppsi\bigl(S_{\mu_1}\caltimes\cdots\caltimes
S_{\mu_n}\bigr).$$

\begin{proposition}
\label{pr:GlobCyc}
Let $\bm\lambda=(\lambda_1,\ldots,\lambda_n)$ in $(\Lambda^+)^n$
and let $\mu\in\Lambda$.
\begin{enumerate}
\item
\label{it:PrGCa}
All the irreducible components of $\overline{\BDConv_n^{\bm\lambda}}\cap
\dot S_\mu$ have dimension $\rho(|\bm\lambda|+\mu)+n$.
\item
\label{it:PrGCb}
The map $(Z_1,\ldots,Z_n)\mapsto\overline{\Uppsi
(Z_1\caltimes\cdots\caltimes Z_n)}$ induces a bijection
$$\bigsqcup_{\substack{(\mu_1,\ldots,\mu_n)\in\Lambda^n\\[2pt]
\mu_1+\cdots+\mu_n=\mu}}{}_*\mathscr Z(\lambda_1)_{\mu_1}
\times\cdots\times\,{}_*\mathscr Z(\lambda_n)_{\mu_n}
\xrightarrow\simeq
\Irr\Bigl(\overline{\BDConv_n^{\bm\lambda}}\cap\dot S_\mu\Bigr).$$
(The bar above $\Uppsi(Z_1\caltimes\cdots\caltimes Z_n)$ means closure
in $\dot S_\mu$.)
\end{enumerate}
\end{proposition}
\begin{proof}
Let $(\mu_1,\ldots,\mu_n)\in\Lambda^n$ be such that $\mu_1+\cdots+\mu_n=\mu$
and let $(Z_1,\ldots,Z_n)\in{}_*\mathscr Z(\lambda_1)_{\mu_1}
\times\cdots\times\,{}_*\mathscr Z(\lambda_n)_{\mu_n}$.
Then the set $\Uppsi(Z_1\caltimes\cdots\caltimes Z_n)$ is irreducible.
By Proposition~\ref{pr:FibBDConv} and its proof, the fiber of this set
over a point $\mathbf x\in\mathbb C^n$ is isomorphic to the product,
over all $y\in\supp(\mathbf x)$, of cycles
$$\Psi(Z_{j_1}\ltimes\cdots\ltimes Z_{j_m})\subseteq\Gr_m$$
where $j_1$, \dots, $j_m$ are the indices $j\in\{1,\ldots,n\}$ such
that $x_j=y$. We remark that if we set
$\bm\lambda_y=(\lambda_{j_1},\ldots,\lambda_{j_m})$
and $\mu_y=\mu_{j_1}+\cdots+\mu_{j_m}$,
then this cycle belongs to ${}_*\mathscr Z(\bm\lambda_y)_{\mu_y}$.
By Proposition~\ref{pr:MVCyc}~\ref{it:PrMVCa}, the dimension of the
fiber of $\Uppsi(Z_1\caltimes\cdots\caltimes Z_n)$ over $\mathbf x$
is therefore
$$\sum_{y\in\supp(\mathbf x)}\rho(|\bm\lambda_y|+\mu_y)=
\rho(|\bm\lambda|+\mu)$$
and we conclude that $\Uppsi(Z_1\caltimes\cdots\caltimes Z_n)$ has
dimension $\rho(|\bm\lambda|+\mu)+n$.

To finish the proof, we observe that these sets
$\Uppsi(Z_1\caltimes\cdots\caltimes Z_n)$
cover $\overline{\BDConv_n^{\bm\lambda}}\cap\dot S_\mu$
and are not redundant.
\end{proof}

Our MV bases are defined with the help of the unstable subsets $T_\mu$
instead of the stable subsets $S_\mu$. We can easily adapt the
constructions of this subsection to this case by replacing the Borel
subgroup $B^\vee$ with its opposite with respect to $T^\vee$, and
replacing similarly its unipotent radical $N^\vee$. We shall do this
while keeping the notation $\caltimes$ and $\Uppsi$. Note that when
we replace $\dot S_\mu$ by
$$\dot T_\mu=\bigcup_{\substack{(\mu_1,\ldots,\mu_n)\in\Lambda^n\\[2pt]
\mu_1+\cdots+\mu_n=\mu}}\Uppsi\bigl(T_{\mu_1}\caltimes\cdots\caltimes
T_{\mu_n}\bigr)$$
in Proposition~\ref{pr:GlobCyc}, $\rho(|\bm\lambda|+\mu)+n$ must
be replaced by $\rho(|\bm\lambda|-\mu)+n$ and the sets
${}_*\mathscr Z(\lambda_j)_{\mu_j}$ must be replaced by their unstarred
counterparts.

\subsection{The fusion product}
\label{ss:FusProd}
For any $x\in\mathbb C$, the fibers of $\BDGr n$ and $\BDConv_n$ over
$(x,\ldots,x)$ are isomorphic to $\Gr$ and $\Gr_n$, respectively. Thus,
$$\BDGr n\bigl|_\Delta\xrightarrow\simeq\Delta\times\Gr
\quad\text{and}\quad
\BDConv_n\bigl|_\Delta\xrightarrow\simeq\Delta\times\Gr_n,$$
where $\Delta$ is the small diagonal, defined as the image of the
map $x\mapsto(x,\ldots,x)$ from $\mathbb C$ to $\mathbb C^n$.
In the other extreme, the morphism $m_n:\BDConv_n\to\BDGr n$ is an
isomorphism after restriction to the open locus $U\subseteq\mathbb C^n$
of points with pairwise different coordinates (\cite{Zhu}, Lemma~3.1.23),
and by Proposition~\ref{pr:FibBDConv}, $\BDConv_n\bigl|_U$ is isomorphic
to $U\times(\Gr)^n$. We define maps $\tau$, $i$, $j$ and $\zeta$ according
to the diagram below.
$$\xymatrix@!C=40pt@R=20pt{\Gr_n&\Delta\times\Gr_n\ar[d]\ar[l]_-\tau
\ar[r]^-i&\BDConv_n\ar[d]^{m_n}&U\times(\Gr)^n\ar[d]^\simeq\ar[l]_-j
\ar[r]^-\zeta&(\Gr)^n\\
&\Delta\times\Gr\ar[d]&\BDGr n\ar[d]^\pi&\BDGr n\bigl|_U\ar[d]&\\
&\Delta\ar[r]&\mathbb C^n&U\ar[l]&}$$

Let $\bm\lambda\in(\Lambda^+)^n$ and $\mu\in\Lambda$, set
$$\mathscr B(\bm\lambda)=\IC\Bigl(\overline{\BDConv_n^{\bm\lambda}},
\,\underline{\mathbb C}\Bigr),\qquad
d=\dim\BDConv_n^{\bm\lambda}=2\rho(|\bm\lambda|)+n,
\qquad k=2\rho(\mu)-n$$
and denote the inclusion $\dot T_\mu\to\BDConv_n$ by $\dot t_\mu$.
The next statement is due to Mirković and Vilonen.

\begin{proposition}
\label{pr:FusProd}
\mbox{}
\vspace{-6pt}
\begin{enumerate}
\item
\label{it:PrFPa}
There are natural isomorphisms
$$i^!\mathscr B(\bm\lambda)[n]\cong\tau^!
\IC\Bigl(\overline{\Gr_n^{\bm\lambda}},\,\underline{\mathbb C}\Bigr)$$
and
$$j^!\mathscr B(\bm\lambda)[n]\cong\zeta^!\Bigl(
\IC\Bigl(\overline{\Gr^{\lambda_1}},\,\underline{\mathbb C}\Bigr)
\boxtimes\cdots\boxtimes
\IC\Bigl(\overline{\Gr^{\lambda_n}},\,\underline{\mathbb C}\Bigr)\Bigr).$$
\item
\label{it:PrFPb}
Each cohomology sheaf of\/ $(\pi\circ m_n)_*\,\mathscr B(\bm\lambda)$
is a local system on $\mathbb C^n$.
\item
\label{it:PrFPc}
The complex of sheaves $(\pi\circ m_n\circ\dot t_\mu)_*\,(\dot t_\mu)^!\,
\mathscr B(\bm\lambda)$ is concentrated in degree $k$ and its $k$-th
cohomology sheaf is a local system on $\mathbb C^n$.
\end{enumerate}
\end{proposition}
\begin{proof}
To prove statement~\ref{it:PrFPa}, one follows the reasoning in
\cite{BaumannRiche}, sect.~1.7.5, noting that $\mathscr B(\bm\lambda)$
and $\IC\Bigl(\overline{\Gr_n^{\bm\lambda}},\,
\underline{\mathbb C}\Bigr)$ are the sheaves denoted by
$\bigl(\tau^\circ\mathscr I_{\lambda_1}\bigr)\;
\widetilde\boxtimes\,\cdots\,\widetilde\boxtimes\;
\bigl(\tau^\circ\mathscr I_{\lambda_n}\bigr)$ and
$\mathscr I_{\lambda_1}\;\widetilde\boxtimes\,\cdots\,
\widetilde\boxtimes\;\mathscr I_{\lambda_n}$ in \textit{loc.~cit.}
Statement~\ref{it:PrFPb} is~\cite{MirkovicVilonen}, (6.4).
Statement~\ref{it:PrFPc} is contained in the proof of
\cite{MirkovicVilonen}, Proposition~6.4, up to a base change in
the Cartesian square
$$\xymatrix@C=30pt@R=20pt{\dot T_\mu
\ar[r]^{\dot t_\mu}\ar[d]&\BDConv_n
\ar[d]^{m_n}\\T_\mu(\mathbb A^n)\ar[r]^{k_\mu}&\BDGr n.}$$
\end{proof}

Combining Propositions~\ref{pr:ICnFold} and~\ref{pr:FusProd}~\ref{it:PrFPa},
we see that the total cohomology of the stalk of the complex
$(\pi\circ m_n)_*\,\mathscr B(\bm\lambda)$ identifies with
$F(\mathscr I_{\bm\lambda})$ at any point in $\Delta$, and with
$F(\mathscr I_{\lambda_1})\otimes\cdots\otimes F(\mathscr I_{\lambda_n})$
at any point in $U$. Statement~\ref{it:PrFPb} in
Proposition~\ref{pr:FusProd} thus provides the identification
$$F(\mathscr I_{\bm\lambda})\cong F(\mathscr I_{\lambda_1})
\otimes\cdots\otimes F(\mathscr I_{\lambda_n})$$
required to compare the two bases of $V(\bm\lambda)$.
Statement~\ref{it:PrFPc} further identifies the weight spaces
$$F_\mu(\mathscr I_{\bm\lambda})\cong
\bigoplus_{\substack{(\mu_1,\ldots,\mu_n)\in\Lambda^n\\[2pt]
\mu_1+\cdots+\mu_n=\mu}}F_{\mu_1}(\mathscr I_{\lambda_1})
\otimes\cdots\otimes F_{\mu_n}(\mathscr I_{\lambda_n}).$$

\subsection{Intersection multiplicities}
\label{ss:InterMult}
We keep the setup introduced in the previous section, in particular
$\bm\lambda\in(\Lambda^+)^n$, $\mu\in\Lambda$,
$$\dot t_\mu:\dot T_\mu\to\BDConv_n,\quad
\mathscr B(\bm\lambda)=\IC\Bigl(\overline{\BDConv_n^{\bm\lambda}},
\,\underline{\mathbb C}\Bigr),\quad
d=\dim\BDConv_n^{\bm\lambda}=2\rho(|\bm\lambda|)+n,
\quad k=2\rho(\mu)-n.$$
In addition, we denote by
$$\mathscr L_\mu(\bm\lambda)=\mathscr H^k\,(\pi\circ m_n\circ\dot t_\mu)_*\,
(\dot t_\mu)^!\,\mathscr B(\bm\lambda)$$
the local system appearing in Proposition~\ref{pr:FusProd}~\ref{it:PrFPc}.

For each point $\mathbf x\in\mathbb C^n$, we define maps as indicated
below
$$\xymatrix@!C=20pt@R=12pt{\bigl(\BDConv_n^{\bm\lambda}\cap
\dot T_\mu\bigr){}^{}_{\mathbf x}\ar[rr]\ar[dd]_-{g'}\ar[dr]_(.44){h'}&&
\bigl(\dot T_\mu\bigr){}^{}_{\mathbf x}\ar'[d]^(.6){\dot t'_\mu}[dd]
\ar[dr]^(.54){i'}&&&\\
&\BDConv_n^{\bm\lambda}\cap\dot T_\mu\ar[rr]\ar[dd]^(.3)g&&
\dot T_\mu\ar[dd]^(.3){\dot t_\mu}&&\\
\bigl(\BDConv_n^{\bm\lambda}\bigr){}^{}_{\mathbf x}
\ar'[r]^(.64){j'}[rr]\ar[dr]_(.46)h&&(\BDConv_n)_{\mathbf x}
\ar[dr]^(.54)i\ar'[r][rr]&&\{\mathbf x\}\ar[dr]^(.54){i_0}&\\
&\BDConv_n^{\bm\lambda}\ar[rr]_(.52)j&&\BDConv_n
\ar[rr]_(.52){\pi\circ m_n}&&\mathbb C^n}$$
where for instance $\bigl(\BDConv_n^{\bm\lambda}\bigr){}^{}_{\mathbf x}$
is the fiber of $\BDConv_n^{\bm\lambda}$ over $\mathbf x$. (The notation
$i$ and $j$ does not designate the same maps as in the previous subsection.)
We then construct the following diagram, referred to as ($\heartsuit$)
in the sequel.
$$\xymatrix@R=12pt@C-=10pt{
H^k\Bigl(\dot T_\mu,(\dot t_\mu)^!\,\mathscr B(\bm\lambda)\Bigr)
\ar[r]^-\simeq\ar[dd]_\simeq&
H^{k+d}\Bigl(\BDConv_n^{\bm\lambda}\cap\dot T_\mu,
g^!\,\underline{\mathbb C}_{\kern1pt\BDConv_n^{\bm\lambda}}\Bigr)
\ar[r]_-\simeq^-{\cap\bigl[\BDConv_n^{\bm\lambda}\bigr]}\ar[dd]&
H^\BM_{d-k}\Bigl(\BDConv_n^{\bm\lambda}\cap\dot T_\mu\Bigr)
\ar[dd]^{(g^*u_{\mathbf x})\cap}\\
&&\kern174pt\\
\kern-1ptH^k\Bigl((\dot T_\mu)_{\mathbf x},(\dot t'_\mu)^!i^*
\mathscr B(\bm\lambda)\Bigr)\ar[r]^-\simeq&
H^{k+d}\Bigl(\bigl(\BDConv_n^{\bm\lambda}\cap\dot T_\mu\bigr)
{}^{}_{\mathbf x},g^{\prime\,!}\,
\underline{\mathbb C}_{\kern1pt(\BDConv_n^{\bm\lambda})^{}_{\mathbf x}}
\Bigr)\ar[r]^-{\kern1pt\cap\bigl[(\BDConv_n^{\bm\lambda}
)^{}_{\mathbf x}\bigr]}&
H^\BM_{d-k-2n}\Bigl(\bigl(\BDConv_n^{\bm\lambda}\cap
\dot T_\mu\bigr){}^{}_{\mathbf x}\Bigr)}$$

The left vertical arrow in ($\heartsuit$) is the restriction of
the cohomology with support in $\dot T_\mu$ from $\BDConv_n$ to
$(\BDConv_n)_{\mathbf x}$; in other words, it is the image by the
functor $H^k\bigl(\dot T_\mu,(\dot t_\mu)^!\,-\bigr)$ of the adjunction
morphism $\mathscr B(\bm\lambda)\to\,i_*i^*\,\mathscr B(\bm\lambda)$.
Lemma~\ref{le:ResLocSys} below implies that it is an isomorphism.
Likewise, the middle vertical arrow is the restriction from
$\BDConv_n^{\bm\lambda}$ to
$\bigl(\BDConv_n^{\bm\lambda}\bigr){}^{}_{\mathbf x}$, afforded by
the adjunction morphism
$j^*\,\mathscr B(\bm\lambda)\to\,h_*h^*j^*\,\mathscr B(\bm\lambda)$.

On the top line, the left arrow is the restriction from $\BDConv_n$
to $\BDConv_n^{\bm\lambda}$, fulfilled by the adjunction morphism
$\mathscr B(\bm\lambda)\to\,j_*j^*\,\mathscr B(\bm\lambda)=\,j_*\,
\underline{\mathbb C}_{\kern1pt\BDConv_n}[d]$. On the bottom line,
it is the restriction from $(\BDConv_n)_{\mathbf x}$ to
$\bigl(\BDConv_n^{\bm\lambda}\bigr){}^{}_{\mathbf x}$, achieved by
$i^*\,\mathscr B(\bm\lambda)\to\,(j')_*(j')^*i^*\,\mathscr B(\bm\lambda)$.
Mirković and Vilonen's argument (reproduced in sect.~\ref{ss:MVBas})
shows that these two arrows are isomorphisms.

The two paths around the left square in ($\heartsuit$) are two
different expressions for the restriction from $\BDConv_n$ to
$\bigl(\BDConv_n^{\bm\lambda}\bigr){}^{}_{\mathbf x}$; therefore
this square commutes.

In both lines of ($\heartsuit$) the right arrow is Alexander duality;
we note that $H^\BM_{d-k}\bigl(\BDConv_n^{\bm\lambda}\cap\dot T_\mu\bigr)$
and $H^\BM_{d-k-2n}\bigl(\bigl(\BDConv_n^{\bm\lambda}\cap
\dot T_\mu\bigr){}^{}_{\mathbf x}\bigr)$ are the top-dimensional
Borel-Moore homology groups.

The map $h$ is a regular embedding of codimension $n$.
Its orientation class (generalized Thom class) is an element
$$u_{\mathbf x}\in H^{2n}\Bigl(\bigl(
\BDConv_n^{\bm\lambda}\bigr){}^{}_{\mathbf x},\,
h^!\,\underline{\mathbb C}_{\kern1pt\BDConv_n^{\bm\lambda}}\Bigr).$$
The right vertical arrow in ($\heartsuit$) is the cap product with
$$g^*u_{\mathbf x}\in H^{2n}\Bigl(\bigl(
\BDConv_n^{\bm\lambda}\cap\dot T_\mu\bigr){}^{}_{\mathbf x},\,
(h')^!\,\underline{\mathbb C}_{\kern1pt\BDConv_n^{\bm\lambda}\cap\dot
T_\mu}\Bigr),$$
the restriction of $u_{\mathbf x}$ to
$\BDConv_n^{\bm\lambda}\cap\dot T_\mu$.

\begin{lemma}
\label{le:HeartComm}
In the diagram (\/$\heartsuit$), the square on the right commutes.
\end{lemma}
\begin{proof}
Applying formula IX.4.9 in \cite{Iversen}, we get
$u_{\mathbf x}\cap\bigl[\BDConv_n^{\bm\lambda}\bigr]=
\bigl[(\BDConv_n^{\bm\lambda})^{}_{\mathbf x}\bigr]$.

Formula~(8) in \cite{Fulton}, sect.~19.1 (or
formula~IX.3.4 in \cite{Iversen}) asserts that given a topological
manifold $X$ and inclusions of closed subsets $a:A\to X$ and $b:B\to X$,
for any
$$\alpha\in H^\bullet(A,\,a^!\,\underline{\mathbb C}_X),\quad
\beta\in H^\bullet(B,\,b^!\,\underline{\mathbb C}_X)\quad\text{and}
\quad C\in H_\bullet^\BM(X)$$
one has
\begin{equation}
\label{eq:IversenIX.3.4}
(b^*\alpha)\cap(\beta\cap C)=(\alpha\cup\beta)\cap C.
\end{equation}
Using the six operations formalism, one checks without much trouble that
this result is also valid if $A$ and $B$ are only locally closed.

Now pick
$$\xi\in H^{k+d}\Bigl(\BDConv_n^{\bm\lambda}\cap\dot T_\mu,
g^!\,\underline{\mathbb C}_{\kern1pt\BDConv_n^{\bm\lambda}}\Bigr).$$
Applying \eqref{eq:IversenIX.3.4} twice and using that $u_{\mathbf x}$
has even degree, we compute
$$(h^*\xi)\cap\Bigl(u_{\mathbf x}\cap
\bigl[\BDConv_n^{\bm\lambda}\bigr]\Bigr)=
(\xi\cup u_{\mathbf x})\cap\bigl[\BDConv_n^{\bm\lambda}\bigr]=
(u_{\mathbf x}\cup\xi)\cap\bigl[\BDConv_n^{\bm\lambda}\bigr]=
(g^*u_{\mathbf x})\cap\Bigl(\xi\cap
\bigl[\BDConv_n^{\bm\lambda}\bigr]\Bigr).$$
This equality means precisely that $\xi$ has the same image under the
two paths in ($\heartsuit$) that circumscribe the square on the right.
\end{proof}

\begin{lemma}
\label{le:ResLocSys}
There are natural isomorphisms
$$H^k\Bigl(\dot T_\mu,\,(\dot t_\mu)^!\,\mathscr B(\bm\lambda)\Bigr)
\cong H^0\bigl(\mathbb C^n,\mathscr L_\mu(\bm\lambda)\bigr)
\quad\text{and}\quad H^k\Bigl((\dot T_\mu)_{\mathbf x},
\,(\dot t'_\mu)^!\,i^*\,\mathscr B(\bm\lambda)\Bigr)
\cong\bigl(\mathscr L_\mu(\bm\lambda)\bigr){}^{}_{\mathbf x}$$
and the left vertical arrow in (\/$\heartsuit$) is the stalk map
$H^0\bigl(\mathbb C^n,\mathscr L_\mu(\bm\lambda)\bigr)\to
\bigl(\mathscr L_\mu(\bm\lambda)\bigr){}^{}_{\mathbf x}$.
\end{lemma}
\begin{proof}
The first isomorphism is
$$H^0\bigl(\mathbb C^n,\mathscr L_\mu(\bm\lambda)\bigr)=
H^k\Bigl(\mathbb C^n,\,(\pi\circ m_n)_*\,(\dot t_\mu)_*
\,(\dot t_\mu)^!\,\mathscr B(\bm\lambda)\Bigr)=
H^k\Bigl(\dot T_\mu,\,(\dot t_\mu)^!\,\mathscr B(\bm\lambda)\Bigr).$$
The second one requires the notion of a universally locally
acyclic complex (see \cite{BravermanGaitsgory02}, sect.~5.1).
Specifically, $\mathscr B(\bm\lambda)$ is $(\pi\circ m_n)$-ULA
(\cite{Reich}, proof of Proposition~IV.3.4, or \cite{Richarz},
Lemma~3.20), so there is an isomorphism
$$i^*\,\mathscr B(\bm\lambda)\to i^!\,\mathscr B(\bm\lambda)[2n].$$
Then
\begin{align*}
H^k\Bigl((\dot T_\mu)_{\mathbf x},
\,(\dot t'_\mu)^!\,i^*\,\mathscr B(\bm\lambda)\Bigr)
&=H^k\Bigl((\dot T_\mu)_{\mathbf x},
\,(\dot t'_\mu)^!\,i^!\,\mathscr B(\bm\lambda)[2n]\Bigr)\\[2pt]
&=H^k\Bigl(\{\mathbf x\},\,(\pi\circ m_n)_*\,(\dot t'_\mu)_*\,
(\dot t'_\mu)^!\,i^!\,\mathscr B(\bm\lambda)[2n]\Bigr)\\[2pt]
&=H^k\Bigl(\{\mathbf x\},\,(\pi\circ m_n)_*\,(\dot t'_\mu)_*\,
i'{}^!\,(\dot t_\mu)^!\,\mathscr B(\bm\lambda)[2n]\Bigr)\\[2pt]
&=H^k\Bigl(\{\mathbf x\},\,(i_0)^!\,(\pi\circ m_n)_*\,
(\dot t_\mu)_*\,(\dot t_\mu)^!\,\mathscr B(\bm\lambda)[2n]\Bigr),
\end{align*}
the last step being proper base change. Now $(\pi\circ m_n)_*\,
(\dot t_\mu)_*\,(\dot t_\mu)^!\,\mathscr B(\bm\lambda)$ is the local
system $\mathscr L_\mu(\bm\lambda)$ shifted by~$-k$, and therefore
$$H^k\Bigl((\dot T_\mu)_{\mathbf x},
\,(\dot t'_\mu)^!\,i^*\,\mathscr B(\bm\lambda)\Bigr)=
H^0\Bigl(\{\mathbf x\},(i_0)^!\,\mathscr L_\mu(\bm\lambda)[2n]\Bigr)=
H^0\Bigl(\{\mathbf x\},(i_0)^*\,\mathscr L_\mu(\bm\lambda)\Bigr)=
\bigl(\mathscr L_\mu(\bm\lambda)\bigr){}^{}_{\mathbf x}$$
as announced.
\end{proof}

By Proposition~\ref{pr:GlobCyc}, the irreducible components of
$\BDConv_n^{\bm\lambda}\cap\dot T_\mu$ are all top-dimensional
and can be indexed by
\begin{equation}
\label{eq:IndexSet}
\bigsqcup_{\substack{(\mu_1,\ldots,\mu_n)\in\Lambda^n\\[2pt]
\mu_1+\cdots+\mu_n=\mu}}\mathscr Z(\lambda_1)_{\mu_1}
\times\cdots\times\,\mathscr Z(\lambda_n)_{\mu_n};
\end{equation}
namely, to a tuple $\mathbf Z=(Z_1,\ldots,Z_n)$ is assigned the component
$$\mathcal X(\mathbf Z)=\overline{\Uppsi(Z_1\caltimes\cdots\caltimes Z_n)}
\cap\BDConv_n^{\bm\lambda},$$
the bar denoting closure in $\dot T_\mu$.
From now on, to lighten the writing, we will substitute
$\mathscr Z(\bm\lambda)_\mu$ for the cumbersome compound
\eqref{eq:IndexSet}, using implicitly the
bijection~\eqref{eq:BijMVCyc}.

The proof of Proposition~\ref{pr:GlobCyc} shows that for any
$\mathbf x\in\mathbb C^n$, the irreducible components of the fiber
$\bigl(\BDConv_n^{\bm\lambda}\cap\dot T_\mu\bigr){}^{}_{\mathbf x}$
have all the same dimension and can be indexed by
$\mathscr Z(\bm\lambda)_\mu$. Let us look more closely at two
particular cases.

If $\mathbf x\in\mathbb C^n$ lies in the open locus $U$ of
points with pairwise different coordinates, then, under the
bijection $(\BDConv_n)_{\mathbf x}\cong(\Gr)^n$ from
Proposition~\ref{pr:FibBDConv}, the irreducible components of
$\bigl(\BDConv_n^{\bm\lambda}\cap\dot T_\mu\bigr){}^{}_{\mathbf x}$
are identified with the sets
\begin{equation}
\label{eq:FiberX/U}
\mathcal X(\mathbf Z)_{\mathbf x}\cong
\bigl(Z_1\cap\Gr^{\lambda_1}\bigr)\times\cdots\times
\bigl(Z_n\cap\Gr^{\lambda_n}\bigr)
\end{equation}
with $\mathbf Z=(Z_1,\ldots,Z_n)$ in $\mathscr Z(\bm\lambda)_\mu$.

On the other hand, recalling that an element
$\mathbf Z\in\mathscr Z(\bm\lambda)_\mu$ is a subset of
$\overline{\Gr_n^{\bm\lambda}}$, we may consider the
preimage $\mathcal Y(\mathbf Z)$ of
$\Delta\times\bigl(\mathbf Z\cap\Gr_n^{\bm\lambda}\bigr)$ under the
isomorphism $\BDConv_n\bigl|_\Delta\xrightarrow\simeq\Delta\times\Gr_n$.
Then for any $\mathbf x\in\Delta$, the irreducible components of the
fiber $\bigl(\BDConv_n^{\bm\lambda}\cap\dot T_\mu\bigr){}^{}_{\mathbf x}$
are the sets $\mathcal Y(\mathbf Z)_{\mathbf x}$
for $\mathbf Z\in\mathscr Z(\bm\lambda)_\mu$.

Let us introduce a last piece of notation before stating the next theorem.
In sect.~\ref{ss:MVBas}, we explained the construction of the MV basis of
the $\mu$-weight space of $V(\bm\lambda)$; this basis is in bijection with
$\mathscr Z(\bm\lambda)_\mu$ and we denote by $\langle\mathbf Z\rangle$
the element indexed by $\mathbf Z$. On the other hand, given
$\mathbf Z=(Z_1,\ldots,Z_n)$ in $\mathscr Z(\lambda_1)\times\cdots\times
\mathscr Z(\lambda_n)$, we can look at~$\dbllangle\mathbf Z\dblrangle
=\langle Z_1\rangle\otimes\cdots\otimes\langle Z_n\rangle$, another
element in $V(\bm\lambda)$.

\begin{theorem}
\label{th:TransMat}
Let $(\mathbf Z',\mathbf Z'')\in(\mathscr Z(\bm\lambda)_\mu)^2$.
The coefficient $a_{\mathbf Z',\mathbf Z''}$ in the expansion
$$\bigl\langle\!\bigl\langle\mathbf Z''\bigr\rangle\!\bigr\rangle=
\sum_{\mathbf Z\in\mathscr Z(\bm\lambda)_\mu}
a_{\mathbf Z,\mathbf Z''}\;\bigl\langle\mathbf Z\bigr\rangle$$
is the multiplicity of $\mathcal Y(\mathbf Z')$ in the intersection
product $\mathcal X(\mathbf Z'')\,\cdot
\bigl(\BDConv_n^{\bm\lambda}\bigr)\bigl|_\Delta$
computed in the ambient space $\BDConv_n^{\bm\lambda}$.
\end{theorem}
\begin{proof}
Taking into account Lemma~\ref{le:ResLocSys}, the diagram
($\heartsuit$) can be rewritten as follows.
$$\xymatrix{H^0\bigl(\mathbb C^n,\mathscr L_\mu(\bm\lambda)\bigr)
\ar[r]^-\simeq\ar[d]_\simeq&
H^\BM_{\mathrm{top}}\Bigl(\BDConv_n^{\bm\lambda}\cap\dot T_\mu\Bigr)
\ar[d]^{(g^*u_{\mathbf x})\cap}\\
\bigl(\mathscr L_\mu(\bm\lambda)\bigr){}^{}_{\mathbf x}\ar[r]^-\simeq&
H^\BM_{\mathrm{top}}\Bigl(\bigl(\BDConv_n^{\bm\lambda}\cap
\dot T_\mu\bigr){}^{}_{\mathbf x}\Bigr)}$$
The fundamental classes of the irreducible components of
$\BDConv_n^{\bm\lambda}\cap\dot T_\mu$ and
$\bigl(\BDConv_n^{\bm\lambda}\cap\dot T_\mu\bigr){}^{}_{\mathbf x}$
provide bases of the two Borel-Moore homology groups, both indexed by
$\mathscr Z(\bm\lambda)_\mu$. In these bases, the right vertical arrow
can be regarded as a matrix, say $Q_{\mathbf x}$. This matrix can be
computed by intersection theory: applying Theorem~19.2 in \cite{Fulton},
we see that if $\mathbf x\in U$ (respectively, $\mathbf x\in\Delta)$, then
the entry in $Q_{\mathbf x}$ at position $(\mathbf Z',\mathbf Z'')$ is the
multiplicity of $\mathcal X(\mathbf Z')_{\mathbf x}$ (respectively,
$\mathcal Y(\mathbf Z')_{\mathbf x}$) in the intersection product
$$\mathcal X(\mathbf Z'')\,\cdot\bigl(\BDConv_n^{\bm\lambda}\bigr)
{}^{}_{\mathbf x}$$
computed in the ambient space $\BDConv_n^{\bm\lambda}$. Because
\eqref{eq:FiberX/U} holds steadily over the open set $U$, we see
that $Q_{\mathbf x_U}$ is just the identity matrix.

According to the discussion at the end of sect.~\ref{ss:FusProd},
the geometric Satake correspondence identifies $V(\bm\lambda)_\mu$
with each fiber of the local system $\mathscr L_\mu(\bm\lambda)$.
The basis element $\langle\mathbf Z\rangle$ is the fundamental class
of $\mathcal X(\mathbf Z)_{\mathbf x}$ when $\mathbf x\in\Delta$, and
the basis element $\dbllangle\mathbf Z\dblrangle$ is the fundamental
class of $\mathcal X(\mathbf Z)_{\mathbf x}$ when $\mathbf x\in U$.
Therefore, the coefficient $a_{\mathbf Z',\mathbf Z''}$ in the statement
of the theorem is the entry at position $(\mathbf Z',\mathbf Z'')$
in the product $Q_{\mathbf x_\Delta}\times(Q_{\mathbf x_U})^{-1}$,
for any choice of $(\mathbf x_\Delta,\mathbf x_U)\in\Delta\times U$.
\end{proof}

In particular, the entries $a_{\mathbf Z',\mathbf Z''}$ of the
transition matrix between our two bases are nonnegative integers.

\begin{proposition}
\label{pr:DiagOne}
In the setup of Theorem~\ref{th:TransMat}, the diagonal entry
$a_{\mathbf Z'',\mathbf Z''}$ is equal to one.
\end{proposition}
\begin{proof}
Write $\mathbf Z''=(Z_1,\ldots,Z_n)$ in
$\mathscr Z(\lambda_1)\times\cdots\times\mathscr Z(\lambda_n)$.
By the slice theorem applied to the quotient map
$G^\vee\bigl(\mathbb C\bigl[z,z^{-1}\bigr]\bigr)\to\Gr$
(or, in this concrete situation, using Remark~15 and Corollary~5
in~\cite{GaussentLittelmann}), we can find, for each
$j\in\{1,\ldots,n\}$, an affine variety $U_j$ and a map
$\phi_j:U_j\to G^\vee\bigl(\mathbb C\bigl[z,z^{-1}\bigr]\bigr)$ such
that $u\mapsto[\phi_j(u)]$ sends $U_j$ isomorphically to an open subset
of $\Gr^{\lambda_j}$ which meets $Z_j$.

For $x\in\mathbb C$ and $u\in U_j$, let $\phi_j(u)_{|x}$ denote the
result of substituting $z-x$ for $z$ in $\phi_j(u)$. We can then
define an open embedding $\upphi$ as on the diagram
$$\xymatrix@C=40pt{\mathbb C^n\times(U_1\times\cdots\times U_n)
\ar[r]^-\upphi\ar[d]&\BDConv_n^{\bm\lambda}\ar[d]^{\pi\circ m_n}\\
\mathbb C^n\ar@{=}[r]&\mathbb C^n}$$
by setting
$$\upphi(x_1,\ldots,x_n;u_1,\ldots,u_n)=\Bigl(x_1,\ldots,x_n;\Bigl[
\phi_1(u_1)_{|x_1},\;\ldots,\;\phi_n(u_n)_{|x_n}\Bigr]
\Bigr).$$
Since intersection multiplicities are of local nature,
$a_{\mathbf Z'',\mathbf Z''}$ can be computed after restriction to
the image of $\upphi$, where the situation is that of a trivial bundle.
\end{proof}

\subsection{An example}
\label{ss:Example}
It is possible to put coordinates on $\BDConv_n^{\bm\lambda}$ and
to effectively compute the intersection multiplicities mentioned
in Theorem~\ref{th:TransMat}. In this section, we look at the case
of the group $G=\SL_3$. We adopt the usual description
$\Lambda=(\mathbb Z\varepsilon_1\oplus\mathbb Z
\varepsilon_2\oplus\mathbb Z\varepsilon_3)/\mathbb Z
(\varepsilon_1+\varepsilon_2+\varepsilon_3)$ of the weight lattice,
so that $V(\varepsilon_1)$ is the defining representation of $G$ and
$V(-\varepsilon_3)$ is its dual.

We consider the sequence of dominant weights
$\bm\lambda=(\varepsilon_1,-\varepsilon_3)$. The basic MV cycles
are $Z_i=\overline{\Gr^{\varepsilon_1}}\cap T_{\varepsilon_i}$ and
$Z_{-i}=\overline{\Gr^{-\varepsilon_3}}\cap T_{-\varepsilon_i}$
for $i\in\{1,2,3\}$, and with this notation
$$\mathscr Z(\bm\lambda)=\bigl\{(Z_i,Z_{-j})\bigm|(i,j)\in\{1,2,3\}^2\bigr\}.$$
To abbreviate, we set $\mathbf Z_{i,-j}=(Z_i,Z_{-j})$. For weight reasons,
$\dbllangle\mathbf Z_{i,-j}\dblrangle=\langle\mathbf Z_{i,-j}\rangle$
if $i\neq j$. The rest of the transition matrix between the two bases
is given as follows.
\begin{align*}
\dbllangle\mathbf Z_{1,-1}\dblrangle&=\langle\mathbf Z_{1,-1}\rangle\\
\dbllangle\mathbf Z_{2,-2}\dblrangle&=\langle\mathbf Z_{2,-2}\rangle
+\langle\mathbf Z_{1,-1}\rangle\\
\dbllangle\mathbf Z_{3,-3}\dblrangle&=\langle\mathbf Z_{3,-3}\rangle
+\langle\mathbf Z_{2,-2}\rangle
\end{align*}
From these relations, we get $\langle\mathbf Z_{3,-3}\rangle=
\dbllangle\mathbf Z_{3,-3}\dblrangle-\dbllangle\mathbf Z_{2,-2}\dblrangle+
\dbllangle\mathbf Z_{1,-1}\dblrangle$. This allows to check that
$\langle\mathbf Z_{3,-3}\rangle$ is $G$-invariant, which in truth is a
consequence of the compatibility of the MV basis of $V(\bm\lambda)$
with the isotypical filtration (Theorem~\ref{th:MVBasLPerf}).

As an example, let us sketch out a computation which justifies that
$\langle\mathbf Z_{1,-1}\rangle$ appears with coefficient one
in~$\dbllangle\mathbf Z_{2,-2}\dblrangle$.
We consider two charts on $\BDConv_2^{\bm\lambda}$, both with
$\mathbb C^6$ as domain:
\begin{align*}
\upphi_1:(x_1,x_2,a,b,c,d)&\mapsto\left(x_1,x_2;
\left[\begin{pmatrix}z-x_1&a&b\\0&1&0\\0&0&1\end{pmatrix},
\begin{pmatrix}1&0&0\\c&z-x_2&0\\d&0&z-x_2\end{pmatrix}\right]\right),\\[8pt]
\upphi_2:(x_1,x_2,a',b',c',d')&\mapsto\left(x_1,x_2;
\left[\begin{pmatrix}1&0&0\\a'&z-x_1&b'\\0&0&1\end{pmatrix},
\begin{pmatrix}z-x_2&c'&0\\0&1&0\\0&d'&z-x_2\end{pmatrix}\right]\right).
\end{align*}
(The matrices here belong to the group $\PGL_3(\mathbb C[z-x_1,z-x_2]$.)
One easily computes the transition map between these two charts:
$$a'=1/a,\quad b'=-b/a,\quad c'=-a(ac+bd+x_2-x_1),\quad d'=-ad.$$

In the chart $\upphi_1$, the cycle $\mathcal Y(\mathbf Z_{1,-1})$
is defined by the equations $a=b=x_2-x_1=0$.
In the chart $\upphi_2$, the cycle $\mathcal X(\mathbf Z_{2,-2})$
is defined by the equations $b'=c'=0$. Thus, the ideals in
$R=\mathbb C[x_1,x_2,a,b,c,d]$ of the subvarieties
$$V=\upphi_1^{-1}\bigl(\mathcal Y(\mathbf Z_{1,-1})\bigr)
\quad\text{and}\quad
X=\upphi_1^{-1}\bigl(\mathcal X(\mathbf Z_{2,-2})\bigr)$$
are respectively
$$\mathfrak p=(a,b,x_2-x_1)\quad\text{and}\quad
\mathfrak q=(b,ac+x_2-x_1).$$
Since $\mathfrak q\subseteq\mathfrak p$, we have $V\subseteq X$;
in fact, $V$ is a subvariety of $X$ of codimension one. The local
ring~$A=\mathscr O_{V,X}$ of $X$ along $V$ is the localization of
$R/\mathfrak q$ at the ideal $\mathfrak p/\mathfrak q$. Observing that
$c$ is not in $\mathfrak p$, we see that its image in $A$ is invertible,
and then that $x_2-x_1$ generates the maximal ideal of $A$. As a
consequence, the order of vanishing of $x_2-x_1$ along $V$ (see
\cite{Fulton}, sect.~1.2) is equal to one. By definition, this is
the multiplicity of $\mathcal Y(\mathbf Z_{1,-1})$ in the
intersection product~$\mathcal X(\mathbf Z_{2,-2})\cdot
\BDConv_2^{\bm\lambda}\bigl|_\Delta$.

\subsection{Factorizations}
\label{ss:Factor}
A nice feature of the Beilinson-Drinfeld Grassmannian is its so-called
factorizable structure (see for instance~\cite{Reich}, Proposition~II.1.13).
On the other side of the geometric Satake equivalence, this corresponds
to associativity properties of partial tensor products.

Let $\mathbf n=(n_1,\ldots,n_r)$ be a composition of $n$ in $r$ parts.
We define the partial diagonal
$$\Delta_{\mathbf n}=\{(\underbrace{x_1,\ldots,x_1}_{n_1\text{ times}},
\ldots,\underbrace{x_r,\ldots,x_r}_{n_r\text{ times}})\mid
(x_1,\ldots,x_r)\in\mathbb C^r\}.$$
We write $\bm\lambda$ as a concatenation $\bigl(\bm\lambda_{(1)},
\ldots,\bm\lambda_{(r)}\bigr)$, where each $\bm\lambda_{(j)}$ belongs
to $(\Lambda^+)^{n_j}$, and similarly we write each
$\mathbf Z\in\mathscr Z(\bm\lambda)_\mu$ as
$\strut\bigl(\mathbf Z_{(1)},\ldots,\mathbf Z_{(r)}\bigr)$ with
$\mathbf Z_{(j)}\in\mathscr Z(\bm\lambda_{(j)})$. Then
$$V(\bm\lambda)=V\bigl(\bm\lambda_{(1)}\bigr)\otimes\cdots\otimes
V\bigl(\bm\lambda_{(r)}\bigr)\quad\text{and}\quad
\bigl\langle\mathbf Z_{(j)}\bigr\rangle\in V\bigl(\bm\lambda_{(j)}\bigr).$$
Further, define
$$\mathcal X(\mathbf Z,\mathbf n)=
\overline{\Uppsi(Z_1\caltimes\cdots\caltimes Z_n)
\bigl|_{\Delta_{\mathbf n}}}\cap\BDConv_n^{\bm\lambda}$$
where the bar means closure in
$\bigl(\dot T_\mu\bigr)\bigl|_{\Delta_{\mathbf n}}$. These
$\mathcal X(\mathbf Z,\mathbf n)$ generalize the set
$\mathcal X(\mathbf Z)$ defined in sect.~\ref{ss:InterMult},
as the latter corresponds to the composition $(1,\ldots,1)$.

Theorem~\ref{th:TransMat} can then be extended to this context in a
straightforward fashion, as demonstrated by the following statement.

\begin{proposition}
\label{pr:Factor}
Let $(\mathbf Z',\mathbf Z'')\in(\mathscr Z(\bm\lambda)_\mu)^2$.
The coefficient $b_{\mathbf Z',\mathbf Z''}$ in the expansion
$$\bigl\langle\mathbf Z''_{(1)}\bigr\rangle\otimes\cdots\otimes
\bigl\langle\mathbf Z''_{(r)}\bigr\rangle=\sum_{\mathbf Z\in\mathscr Z
(\bm\lambda)_\mu}b_{\mathbf Z,\mathbf Z''}\;\langle\mathbf Z\rangle$$
is the multiplicity of $\mathcal Y(\mathbf Z')$
in the intersection product
$\mathcal X(\mathbf Z'',\mathbf n)\,\cdot
\bigl(\BDConv_n^{\bm\lambda}\bigr)\bigl|_\Delta$
computed in the ambient space
$\BDConv_n^{\bm\lambda}\bigl|_{\Delta_{\mathbf n}}$.
\end{proposition}

The proof does not require any new ingredient and is left to the reader.

\subsection{Triangularity}
\label{ss:Triang}
In this section, we show that the transition matrix described in
Theorem~\ref{th:TransMat} is unitriangular with respect to an adequate
order on $\mathscr Z(\bm\lambda)_\mu$.

\begin{proposition}
\label{pr:HoroClosure}
Let $(\mu_1,\ldots,\mu_n)$ and\/ $(\nu_1,\ldots,\nu_n)$ in $\Lambda^n$
and let $\mathcal S$ be a stratum for the ind-structure of\/ $\BDConv_n$.
If\/ $\Uppsi\bigl(T_{\nu_1}\caltimes\cdots\caltimes T_{\nu_n}\bigr)$
meets the closure of\/ $\mathcal S\cap\Uppsi\bigl(T_{\mu_1}\caltimes
\cdots\caltimes T_{\mu_n}\bigr)$, then
$$\nu_1\geq\mu_1,\quad\nu_1+\nu_2\geq\mu_1+\mu_2,\quad\ldots,\quad
\nu_1+\cdots+\nu_n\geq\mu_1+\cdots+\mu_n.$$
\end{proposition}
\begin{proof}
Given a tuple $\bm\zeta=(\zeta_1,\ldots,\zeta_n)$ in
$(\Lambda/\mathbb Z\Phi)^n$, we set
$$\BDConv_{n,\bm\zeta}=\bigsqcup_{\substack{\bm\lambda\in(\Lambda^+)^n\\[2pt]
\lambda_1\in\zeta_1,\,\ldots,\,\lambda_n\in\zeta_n}}\BDConv_n^{\bm\lambda}.$$
From equation~\eqref{eq:ClosureStratBD}, we deduce that each
$\BDConv_{n,\bm\zeta}$ is closed and connected in the ind-topology;
as these subsets form a finite partition of the space $\BDConv_n$,
they are its connected components. We easily verify that a subset of
the form $\Uppsi\bigl(T_{\mu_1}\caltimes\cdots\caltimes T_{\mu_n}\bigr)$
is contained in $\BDConv_{n,\bm\zeta}$ if each $\zeta_j$ is the coset
of $\mu_j$ modulo $\mathbb Z\Phi$. Therefore, a necessary condition for
$\Uppsi\bigl(T_{\nu_1}\caltimes\cdots\caltimes T_{\nu_n}\bigr)$
to meet the closure of\/ $\mathcal S\cap\Uppsi\bigl(T_{\mu_1}\caltimes
\cdots\caltimes T_{\mu_n}\bigr)$ is that $\mu_j-\nu_j\in\mathbb Z\Phi$
for each $j\in\{1,\ldots,n\}$.

Let $\lambda^\vee\in\Hom_{\mathbb Z}(\Lambda,\mathbb Z)$ be a dominant
integral weight for the group $G^\vee$ and let $V$ be the finite
dimensional irreducible representation of $G^\vee$ of highest weight
$\lambda^\vee$. Then $G^\vee(\mathbb C(z))$ acts on
$V\otimes\mathbb C(z)$. The standard lattice $L_0=V\otimes\mathbb C[z]$
is left stable by $G^\vee(\mathbb C[z])$.

We choose a nonzero linear form $p:V\to\mathbb C$ that vanishes
on all weight subspaces of $V$ but the highest weight subspace.
Extending the scalars, we regard $p$ as a linear form
$V\otimes\mathbb C(z)\to\mathbb C(z)$.

For $\mathbf x=(x_1,\ldots,x_n)$ in $\mathbb C^n$, we set
$f_{\mathbf x}=(z-x_1)\cdots(z-x_n)$. Let $\mathcal S$ be
a stratum for the ind-structure of $\BDConv_n$.
There exists a positive integer $k$ such that
$f_{\mathbf x}^kL_0\subseteq\beta_1\ldots\beta_n(L_0)\subseteq
f_{\mathbf x}^{-k}L_0$ for each
$(x_1,\ldots,x_n;[\beta_1,\ldots,\beta_n])\in\mathcal S$.

Now we take $(\mu_1,\ldots,\mu_n)\in\Lambda^n$ and
$(x_1,\ldots,x_n;[\beta_1,\ldots,\beta_n])$ in $\mathcal S\cap
\Uppsi\bigl(T_{\mu_1}\caltimes\cdots\caltimes T_{\mu_n}\bigr)$.
Then $p(\beta_1\ldots\beta_n(L_0))$ is the fractional ideal
$$(z-x_1)^{\langle\lambda^\vee,\mu_1\rangle}\cdots
(z-x_n)^{\langle\lambda^\vee,\mu_n\rangle}\,\mathbb C[z],$$
and therefore
$$\dim\bigl(p(\beta_1\ldots\beta_n(L_0))\kern1pt/\kern1pt
f_{\mathbf x}^k\,\mathbb C[z]\bigr)=
kn-\bigl\langle\lambda^\vee,\mu_1+\cdots+\mu_n\bigr\rangle.$$

If the point $(x_1,\ldots,x_n;[\beta_1,\ldots,\beta_n])$ degenerates to
$$(y_1,\ldots,y_n;[\gamma_1,\ldots,\gamma_n])\in
\Uppsi\bigl(T_{\nu_1}\caltimes\cdots\caltimes T_{\nu_n}\bigr),$$
then
$$\dim\bigl(p(\gamma_1\ldots\gamma_n(L_0))\kern1pt/\kern1pt
f_{\mathbf y}^k\,\mathbb C[z]\bigr)\leq
\dim\bigl(p(\beta_1\ldots\beta_n(L_0))\kern1pt/\kern1pt
f_{\mathbf x}^k\,\mathbb C[z]\bigr)$$
which translates to
$$\bigl\langle\lambda^\vee,\nu_1+\cdots+\nu_n\bigr\rangle\geq
\bigl\langle\lambda^\vee,\mu_1+\cdots+\mu_n\bigr\rangle.$$
This inequation holds for any dominant coweight $\lambda^\vee$,
hence $\nu_1+\cdots+\nu_n\geq\mu_1+\cdots+\mu_n$.

This proves the last inequality among those announced in the statement.
The other ones can be shown in a similar way, by taking the image
under the obvious truncation map $\BDConv_n\to\BDConv_j$ for each
$j\in\{1,\ldots,n\}$.
\end{proof}

\begin{corollary}
\label{co:TriangWeight}
Adopt the setup of Theorem~\ref{th:TransMat}. Let
$(\mu_1,\ldots,\mu_n)$ and $(\nu_1,\ldots,\nu_n)$ in $\Lambda^n$
be such that $\mathbf Z'\in\mathscr Z(\lambda_1)_{\nu_1}
\times\cdots\times\,\mathscr Z(\lambda_n)_{\nu_n}$ and
$\mathbf Z''\in\mathscr Z(\lambda_1)_{\mu_1}
\times\cdots\times\,\mathscr Z(\lambda_n)_{\mu_n}$.
A necessary condition for $a_{\mathbf Z',\mathbf Z''}\neq0$ is that
$$\nu_1\geq\mu_1,\quad\nu_1+\nu_2\geq\mu_1+\mu_2,\quad\ldots,\quad
\nu_1+\cdots+\nu_{n-1}\geq\mu_1+\cdots+\mu_{n-1}.$$
\end{corollary}

We can obtain more stringent conditions regarding the transition matrix
by looking at the associativity properties from sect.~\ref{ss:Factor}.
The sharpest result is obtained with a composition $(n_1,n_2)$ of $n$ in
two parts. Accordingly, we write $\bm\lambda$ as a concatenation
$(\bm\lambda_{(1)},\bm\lambda_{(2)})$ and similarly write each
$\mathbf Z\in\mathscr Z(\bm\lambda)$ as
$(\mathbf Z_{(1)},\mathbf Z_{(2)})$. Here $\mathbf Z_{(1)}$ is an element
in $\mathscr Z(\lambda_1)\times\cdots\times\mathscr Z(\lambda_{n_1})$,
but owing to the bijection~\eqref{eq:BijMVCyc} it can also be regarded
as a cycle in $\overline{\Gr_{n_1}^{\bm\lambda_{(1)}}}$.

\begin{theorem}
\label{th:TriangIncl}
Let $(\mathbf Z',\mathbf Z'')\in(\mathscr Z(\bm\lambda)_\mu)^2$.
Consider the expansion
$$\bigl\langle\mathbf Z''_{(1)}\bigr\rangle\otimes
\bigl\langle\mathbf Z''_{(2)}\bigr\rangle=\sum_{\mathbf Z\in\mathscr Z
(\bm\lambda)_\mu}b_{\mathbf Z,\mathbf Z''}\;\langle\mathbf Z\rangle.$$
If $b_{\mathbf Z',\mathbf Z''}\neq0$, then either
$\mathbf Z'=\mathbf Z''$ or
$\mathbf Z'_{(1)}\subsetneq\overline{\mathbf Z''_{(1)}}$
as cycles in $\overline{\Gr_{n_1}^{\bm\lambda_{(1)}}}$.
In addition, $b_{\mathbf Z'',\mathbf Z''}=1$.
\end{theorem}
\begin{proof}
Let $\mathbf Z'=(Z'_1,\ldots,Z'_n)$ and $\mathbf Z''=(Z''_1,\ldots,Z''_n)$
in $\mathscr Z(\bm\lambda)_\mu$.

For $j\in\{1,\ldots,n\}$, let $\mu_j$ be the weight such that
$Z''_j\in\mathscr Z(\lambda_j)_{\mu_j}$. Using the gallery models
from~\cite{GaussentLittelmann} (or Theorem~\ref{th:ZIsMVCyc} and
Proposition~\ref{pr:AltDefRingZ} above), we find a nonnegative
integer $d_j$ and construct a map $\phi_j:\mathbb C^{d_j}\to
N^{-,\vee}\bigl(\mathbb C\bigl[z,z^{-1}\bigr]\bigr)$ such that
$\bigl\{\bigl[\phi_j(\mathbf a)\,z^{\mu_j}\bigr]\bigm|
\mathbf a\in\mathbb C^{d_j}\bigr\}$ is a dense subset of~$Z''_j$. Then
$$\upphi:(\mathbf x;\mathbf a_1,\ldots,\mathbf a_n)\mapsto\Bigl(\mathbf x;
\Bigl[\phi_1(\mathbf a_1)_{|x_1}\,(z-x_1)^{\mu_1},\;\ldots,\;
\phi_n(\mathbf a_n)_{|x_n}\,(z-x_n)^{\mu_n}\Bigr]\Bigr)$$
maps $\mathbb C^n\times\mathbb C^{d_1}\times\cdots\times\mathbb C^{d_n}$
onto a dense subset of $\Uppsi(Z''_1\caltimes\cdots\caltimes Z''_n)$,
where $\phi_j(\mathbf a_j)_{|x_j}$ means the result of
substituting $z-x_j$ for $z$ in $\phi_j(\mathbf a_j)$.

Assume that $b_{\mathbf Z',\mathbf Z''}\neq0$. By Proposition
\ref{pr:Factor}, $\mathcal Y(\mathbf Z')$ is contained in
$\mathcal X(\mathbf Z'',(n_1,n_2))$, hence in the closure of
$\Uppsi(Z''_1\caltimes\cdots\caltimes Z''_n)\bigl|_{\Delta_{(n_1,n_2)}}$.

Take a point in $\mathbf Z'_{(1)}\cap\Gr_{n_1}^{\bm\lambda_{(1)}}$,
written as $[g_1,\ldots,g_{n_1}]$ where each $g_j$ is in
$G^\vee\bigl(\mathbb C\bigl[z,z^{-1}\bigr]\bigr)$.
We can complete this datum to get an element
$$\Gamma=\bigl(0,\ldots,0;\bigl[g_1,\ldots,g_n\bigr]\bigr)$$
of $\mathcal Y(\mathbf Z')$. Working in the analytic topology for
expositional simplicity, we see that $\Gamma$ is the limit of a
sequence $(\upphi(\mathbf x_p;\mathbf a_{1,p},\ldots,\mathbf a_{n,p})
)_{p\in\mathbb N}$ with $\mathbf x_p\in\Delta_{(n_1,n_2)}$ and
$(\mathbf a_{1,p},\ldots,\mathbf a_{n,p})\in\mathbb C^{d_1}\times\cdots
\times\mathbb C^{d_n}$. We write
\begin{equation}
\label{eq:PfTriangIncl1}
\mathbf x_p=(\underbrace{x_{1,p},\ldots,x_{1,p}}_{n_1\text{ times}},
\underbrace{x_{2,p},\ldots,x_{2,p}}_{n_2\text{ times}})\quad\text{ with
of course }\quad\lim_{p\to\infty}x_{1,p}=\lim_{p\to\infty}x_{2,p}=0.
\end{equation}
Then
\begin{align*}
[g_1,\ldots,g_{n_1}]
&=\lim_{p\to\infty}\bigl[\phi_1(\mathbf a_{1,p})\,z^{\mu_1},\,\ldots,\,
\phi_{n_1}(\mathbf a_{1,n_1})\,z^{\mu_{n_1}}\bigr]_{|x_{1,p}}\\[4pt]
&=\lim_{p\to\infty}\bigl[\phi_1(\mathbf a_{1,p})\,z^{\mu_1},\,\ldots,\,
\phi_{n_1}(\mathbf a_{1,n_1})\,z^{\mu_{n_1}}\bigr]
\end{align*}
is the limit of a sequence of points in $\mathbf Z''_{(1)}$.
Therefore $\mathbf Z'_{(1)}\cap\Gr_{n_1}^{\bm\lambda_{(1)}}\subseteq
\overline{\mathbf Z''_{(1)}}$, whence the inclusion
$\mathbf Z'_{(1)}\subseteq\overline{\mathbf Z''_{(1)}}$.

In addition to $b_{\mathbf Z',\mathbf Z''}\neq0$, assume that the latter
inclusion is an equality. Then $\mathbf Z'_{(1)}=\mathbf Z''_{(1)}$
because these two MV cycles are irreducible components of the same
$\overline{\Gr_{n_1}^{\bm\lambda_{(1)}}}\cap(m_{n_1})^{-1}(T_{\mu_{(1)}})$,
with indeed $\mu_{(1)}=\mu_1+\cdots+\mu_{n_1}$. We regard
$\mathbf Z'_{(2)}$ and $\mathbf Z''_{(2)}$ as cycles in
$\overline{\Gr_{n_2}^{\bm\lambda_{(2)}}}$. Take a point in
$\mathbf Z'_{(2)}\cap\Gr_{n_2}^{\bm\lambda_{(2)}}$, written as
$[g_{n_1+1},\ldots,g_n]$ where each $g_j$ is in
$G^\vee\bigl(\mathbb C\bigl[z,z^{-1}\bigr]\bigr)$. We can then
look at the element
$$\Gamma=\bigl(0,\ldots,0;\bigl[z^{\mu_1},\ldots,z^{\mu_{n_1}},
g_{n_1+1},\ldots,g_n\bigr]\bigr)$$
of $\mathcal Y(\mathbf Z')$. Again $\Gamma$ is the limit
of a sequence $(\upphi(\mathbf x_p;\mathbf a_{1,p},\ldots,
\mathbf a_{n,p}))_{p\in\mathbb N}$ with $\mathbf x_p\in\Delta_{(n_1,n_2)}$
and $(\mathbf a_{1,p},\ldots,\mathbf a_{n,p})\in\mathbb C^{d_1}\times
\cdots\times\mathbb C^{d_n}$. We set
$$B_p=z^{-\mu_{(1)}}\;\phi_1(\mathbf a_{1,p})\,z^{\mu_1}\;
\cdots\;\phi_{n_1}(\mathbf a_{n_1,p})\,z^{\mu_{n_1}}.$$
Writing again \eqref{eq:PfTriangIncl1}, we have
\begin{equation}
\label{eq:PfTriangIncl2}
L_{\mu_{(1)}}=\lim_{p\to\infty}\bigl[z^{\mu_{(1)}}B_p\bigr]
_{|x_{1,p}}=\lim_{p\to\infty}\bigl[z^{\mu_{(1)}}B_p\bigr]
\end{equation}
and
\begin{equation}
\label{eq:PfTriangIncl3}
\begin{split}
z^{\mu_{(1)}}[g_{n_1+1},\ldots,g_n]=&\\
\lim_{p\to\infty}&\,(z^{\mu_{(1)}}B_p)_{|x_{1,p}}\,
\bigl[\phi_{n_1+1}(\mathbf a_{n_1+1,p})\,z^{\mu_{n_1+1}},\;
\ldots,\;\phi_n(\mathbf a_{n,p})\,z^{\mu_n}\bigr]
_{|x_{2,p}}.
\end{split}
\end{equation}
Let $K$ be the kernel of the evaluation map
$N^{-,\vee}\bigl(\mathbb C\bigl[z^{-1}\bigr]\bigr)
\to N^{-,\vee}(\mathbb C)$ at $z=\infty$.
The multi\-plication induces a bijection
$$K\times N^{-,\vee}(\mathbb C[z])\xrightarrow\simeq
N^{-,\vee}\bigl(\mathbb C\bigl[z,z^{-1}\bigr]\bigr).$$
We decompose $B_p$ as a product $B_{-,p}B_{+,p}$ according to
this bijection. Using~\eqref{eq:PfTriangIncl2} and identifying the
ind-variety $T_0$ with $K$, we obtain that $B_{-,p}$ tends to one
when $p$ goes to infinity. Inserting this information
in~\eqref{eq:PfTriangIncl3}, we obtain
$$[g_{n_1+1},\ldots,g_n]=\lim_{p\to\infty}B_{+,p}\,
\bigl[\phi_{n_1+1}(\mathbf a_{n_1+1,p})\,z^{\mu_{n_1+1}},\;
\ldots,\;\phi_n(\mathbf a_{n,p})\,z^{\mu_n}\bigr],$$
so $[g_{n_1+1},\ldots,g_n]$ is the limit of a sequence
of points in $\mathbf Z''_{(2)}$. We conclude that
$\mathbf Z'_{(2)}\subseteq\overline{\mathbf Z''_{(2)}}$, and
since these two cycles have the same dimension, that actually
$\mathbf Z'_{(2)}=\mathbf Z''_{(2)}$.

To sum up: if $b_{\mathbf Z',\mathbf Z''}\neq0$, then
$\mathbf Z'_{(1)}\subseteq\overline{\mathbf Z''_{(1)}}$, and in
case of equality $\mathbf Z'_{(1)}=\mathbf Z''_{(1)}$, we
additionally have $\mathbf Z'_{(2)}=\mathbf Z''_{(2)}$.
This proves the first statement in the theorem. The second one is
proved in the same manner as Proposition~\ref{pr:DiagOne}.
\end{proof}

\begin{other}{Remark}
\label{rk:TriangStrict}
Using Theorem~\ref{th:TriangIncl}, one easily sharpens
Corollary~\ref{co:TriangWeight}: with the notation of the latter,
if $a_{\mathbf Z',\mathbf Z''}\neq0$, then either
$\mathbf Z'=\mathbf Z''$ or one of the displayed inequalities
is strict. The proof is left to the reader.
\end{other}

\textit{Application to standard monomial theory.}\\[2pt]
Let $\lambda\in\Lambda^+$ and let $\ell\subseteq V(\lambda)^*$ be the line
spanned by the highest weight vectors. The group $G$ acts on the projective
space $\mathbb P(V(\lambda)^*)$; let $Q$ be the stabilizer of $\ell$, a
parabolic subgroup of $G$. The map $g\mapsto g\ell$ induces an embedding
of the partial flag variety $X=G/Q$ in $\mathbb P(V(\lambda)^*)$. We denote
by $\mathscr L$ the pull-back of the line bundle $\mathscr O(1)$ by this
embedding; then the homogeneous coordinate ring of $X$ is
$$R_\lambda=\bigoplus_{m\geq0}H^0\bigl(X,\mathscr L^{\otimes m}\bigr);$$
here $H^0\bigl(X,\mathscr L^{\otimes m}\bigr)$ is isomorphic to
$V(m\lambda)$ and the multiplication in $R_\lambda$ is given by the
projection onto the Cartan component
$$V(m\lambda)\otimes V(n\lambda)\to V((m+n)\lambda).$$
The algebra $R_\lambda$ is endowed with an MV basis, obtained by gathering
the MV bases of the summands $V(m\lambda)$.

Each MV cycle $Z\in\mathscr Z(\lambda)$ defines a basis element
$\langle Z\rangle\in V(\lambda)$. Given an $m$-tuple
$\mathbf Z=(Z_1,\ldots,Z_m)$ of elements of $\mathscr Z(\lambda)$,
the product $\langle Z_1\rangle\cdots \langle Z_m\rangle$ in the
algebra $R_\lambda$ is the image of $\dbllangle\mathbf Z\dblrangle=
\langle Z_1\rangle\otimes\cdots\otimes\langle Z_m\rangle$ under
the projection $V(\lambda)^{\otimes m}\to V(m\lambda)$. This
product is called standard if $\mathbf Z$ lies in the Cartan
component of the crystal $\mathscr Z(\lambda)^{\otimes m}$.

Remark~\ref{rk:StringP2} implies that the MV basis element
$\langle\mathbf Z\rangle\in V(\lambda)^{\otimes m}$ goes, under the
projection $V(\lambda)^{\otimes m}\to V(m\lambda)$, either to an
element in the MV basis of $V(m\lambda)$ or to $0$, depending on
whether $\mathbf Z$ lies or not in the Cartan component of
$\mathscr Z(\lambda)^{\otimes m}$.

Using Corollary~\ref{co:TriangWeight} and Remark~\ref{rk:TriangStrict},
we can then endow, for each degree $m$, the Cartan component of
$\mathscr Z(\lambda)^{\otimes m}$ with an order, so that the transition
matrix expressing the standard monomials in the MV basis of $R_\lambda$
is unitriangular. In particular, the standard monomials form a basis for
the algebra $R_\lambda$ too, and straightening laws can be obtained from
Theorem~\ref{th:TransMat}.

The dual of the MV basis is compatible with the Demazure modules contained
in $V(m\lambda)^*$; this property is recorded as Remark~2.6~(ii) in
\cite{BaumannKamnitzerKnutson} but the crux of the argument is due to
Kashiwara~\cite{Kashiwara93b}. This implies that for any Schubert variety
$Y\subseteq X$, the kernel of the restriction map
$$\bigoplus_{m\geq0}H^0\bigl(X,\mathscr L^{\otimes m}\bigr)\to
\bigoplus_{m\geq0}H^0\bigl(Y,\mathscr L^{\otimes m}\bigr)$$
is spanned by a subset of the MV basis of $R_\lambda$.
Therefore the homogeneous coordinate ring of $Y$ is also endowed with
an MV basis.

These observations suggest that the MV basis could be a relevant tool
for the study of the standard monomial theory.

\subsection{A conjectural symmetry}
\label{ss:ConjSymm}
Recall the notation set up in sects.~\ref{ss:Crystals}--\ref{ss:LPerfBas}.
Given $\lambda\in\Lambda^+$, we set $\lambda^*=-w_0\lambda$, where as usual
$w_0$ denotes the longest element in the Weyl group $W$. As is well known,
there exists a unique bijection $\sigma:B(\lambda)\to B(\lambda^*)$ which
for each $i\in I$ exchanges the actions of $\tilde e_i$ and $\tilde f_i$.
In our context, we will regard $\sigma$ as a bijection
$\mathscr Z(\lambda)\to\mathscr Z(\lambda^*)$ and may define it by means
of Lemma~2.1~(e) in~\cite{Littelmann95} and Theorem~\ref{th:ZMorCrys}.

Now let $n\geq1$ and let $\bm\lambda=(\lambda_1,\ldots,\lambda_n)$ in
$(\Lambda^+)^n$. We set $\bm\lambda^*=(\lambda_n^*,\ldots,\lambda_1^*)$
and define a bijection
$$\sigma:\mathscr Z(\lambda_1)\times\cdots\times\mathscr Z(\lambda_n)
\to\mathscr Z(\lambda_n^*)\times\cdots\times\mathscr Z(\lambda_1^*)$$
by $\sigma(Z_1,\ldots,Z_n)=(\sigma(Z_n),\ldots,\sigma(Z_1))$. (Using
the same symbol $\sigma$ to denote different bijections is certainly
abusive, but adding extra indices to disambiguate would overload the
notation without clear benefit.) The Cartesian products above are in fact
tensor product of crystals, and here again $\sigma$ exchanges the actions
of $\tilde e_i$ and $\tilde f_i$ for each $i\in I$
(\cite{HenriquesKamnitzer}, Theorem~2).

Let $\mu\in\Lambda$ and choose $(\mathbf Z',\mathbf Z'')\in
(\mathscr Z(\bm\lambda)_\mu)^2$; we then obtain $\sigma(\mathbf Z')$ and
$\sigma(\mathbf Z'')$ in $\mathscr Z(\bm\lambda^*)_{-\mu}$. Recall the
notation introduced in Theorem~\ref{th:TransMat} to denote the entries
of the transition matrix between the two bases of $V(\bm\lambda)$ and
adopt a similar notation as regards $V(\bm\lambda^*)$.
\begin{other}{Conjecture}
\label{cj:ConjSymm}
The equality
$a_{\mathbf Z',\mathbf Z''}=a_{\sigma(\mathbf Z'),\sigma(\mathbf Z'')}$
holds.
\end{other}

According to \cite{Demarais}, this conjecture is true in type $A_1$.
Its general validity would have two interesting consequences.

Firstly, one could then strengthen Theorem~\ref{th:TriangIncl};
indeed $b_{\mathbf Z',\mathbf Z''}\neq0$ would imply not only
$\mathbf Z'_{(1)}\subseteq\overline{\mathbf Z''_{(1)}}$, but also
$\sigma(\mathbf Z'_{(2)})\subseteq\overline{\sigma(\mathbf Z''_{(2)})}$,
restoring the symmetry between the two tensor factors.

Secondly, the MV basis of an irreducible representation $V(\lambda)$
would then satisfy the analogue of \cite{Lusztig93}, Proposition~21.1.2.
In fact, one easily verifies that the MV basis enjoys this property
if $\lambda$ is minuscule or quasi-minuscule; our conjecture would allow
to deduce the general case by taking suitable tensor products, mimicking
the strategy of proof from~\cite{NgoPolo}.

\section{The basis on the invariant subspace}
\label{se:InvSubsp}
Let $n\geq1$ and let $\bm\lambda\in(\Lambda^+)^n$. The MV basis of
$V(\bm\lambda)$ is compatible with the isotypical filtration, hence
provides a basis of the invariant subspace $V(\bm\lambda)^G$, called
the Satake basis in \cite{FontaineKamnitzerKuperberg}. In this section
we study two properties of this basis.

\subsection{Cyclic permutations}
\label{ss:CycPerm}
Let us write $\bm\lambda=(\lambda_1,\ldots,\lambda_n)$ and consider the
rotated sequence $\bm\lambda^{[1]}=(\lambda_2,\ldots,\lambda_n,\lambda_1)$.
Thus,
$$V(\bm\lambda)=V(\lambda_1)\otimes\cdots\otimes V(\lambda_n)
\quad\text{and}\quad
V\bigl(\bm\lambda^{[1]}\bigr)=V(\lambda_2)\otimes\cdots\otimes
V(\lambda_n)\otimes V(\lambda_1).$$
The signed cyclic permutation
$$x_1\otimes\cdots\otimes x_n\mapsto(-1)^{2\rho(\lambda_1)}\;
x_2\otimes\cdots\otimes x_n\otimes x_1,$$
defines an isomorphism of $G$-modules
$R:V(\bm\lambda)\to V\bigl(\bm\lambda^{[1]}\bigr)$. In particular,
$R$ induces a linear bijection between the invariant subspaces.

\begin{theorem}
\label{th:CycPerm}
The signed cyclic permutation $R$ maps the Satake basis of\/
$V(\bm\lambda)^G$ to the Satake basis of\/
$V\bigl(\bm\lambda^{[1]}\bigr)^G$.
\end{theorem}

Theorem~\ref{th:CycPerm} replicates a similar result for the dual
canonical basis due to Lusztig (\cite{Lusztig93}, 28.2.9), and our
proof below mirrors Lusztig's argument. It has been proved by Fontaine,
Kamnitzer and Kuperberg in the case where all the weights $\lambda_j$
are minuscule (\cite{FontaineKamnitzerKuperberg}, Theorem~4.5).
The bijection induced by $R$ between the two Satake bases has a nice
interpretation, both in terms of crystals (see \cite{FontaineKamnitzer})
and in terms of cluster combinatorics (see \cite{GoncharovShen},
sect.~2.1.6).

The rest of this section is devoted to the proof of Theorem~\ref{th:CycPerm}.

As in sect.~\ref{se:LPerf}, we denote by $\{\alpha_i\mid i\in I\}$ the
set of simple roots and choose simple root vectors $e_i$ and $f_i$ in
the Lie algebra of $G$ of weights $\pm\alpha_i$ such that
$[e_i,f_i]=-\alpha_i^\vee$. The Weyl group $W$ is generated by the
simple reflections $s_i$ and contains a longest element $w_0$.

Given $\lambda\in\Lambda^+$ and $w\in W$, we can pick a reduced word
$(i_1,\ldots,i_\ell)$ of $w$ and form the product of divided powers
$$\theta(w,\lambda)=f_{i_1}^{(n_1)}\cdots f_{i_\ell}^{(n_\ell)},
\quad\text{where}\quad
n_j=\langle\alpha_{i_j}^\vee,s_{i_{j+1}}\cdots s_{i_\ell}\lambda\rangle.$$
This element does not depend on the choice of $(i_1,\ldots,i_\ell)$
(\cite{Lusztig93}, Proposition~28.1.2), which legitimizes the notation.
We note that $\theta(w_0,\lambda)$ acts on $V(\lambda)$ by mapping
highest weight vectors to lowest weight vectors.

We set $\lambda=\lambda_1$, the first element in the sequence
$\bm\lambda$. With the notation of sect.~\ref{ss:InterMult}, the
highest and lowest weight elements in the MV basis of $V(\lambda)$ are
$$v_\lambda=\bigl\langle\{L_\lambda\}\bigr\rangle\quad\text{and}\quad
v_{w_0\lambda}=\Bigl\langle\,\overline{\Gr^\lambda}\,\Bigr\rangle.$$
Under suitable normalizations in the geometric Satake equivalence,
these two elements are related by
$v_{w_0\lambda}=\theta(w_0,\lambda)\cdot v_\lambda$
(see \cite{BaumannKamnitzerKnutson}, Theorem~5.2 and Remark~2.10).
We define $\xi$ to be the linear form on $V(\lambda)$ such that
$\langle\xi,v_\lambda\rangle=1$ and that vanishes on all
weight subspaces of weight different from $\lambda$. Similarly,
we define $\eta$ to be the linear form on $V(\lambda)$ such that
$\langle\eta,v_{w_0\lambda}\rangle=1$ and that vanishes on all weight
subspaces of weight different from $w_0\lambda$.

Let $M$ be a representation of $G$. The assignment
$v\otimes m\mapsto(-1)^{2\rho(\lambda)}\;m\otimes v$ defines
an isomor\-phism $P:V(\lambda)\otimes M\to M\otimes V(\lambda)$.

We set $\lambda^*=-w_0\lambda$. Let $M^\circ$ be the isotypical
component of $M$ corresponding to the highest weight $\lambda^*$,
namely, the sum of all subrepresentations isomorphic to
$V(\lambda^*)$. Given a weight $\mu\in\Lambda$, we denote by
$M_\mu$ the corresponding weight subspace of $M$ and set
$M^\circ_\mu=M^\circ\cap M_\mu$. Then $M^\circ_{\lambda^*}$ is
the set of all vectors in $M_{\lambda^*}$ that are annihilated by all
the root vectors $e_i$ and $M^\circ_{w_0\lambda^*}$ is the set of
all vectors in $M_{w_0\lambda^*}$ that are annihilated by all the
root vectors~$f_i$.

\begin{lemma}
\label{le:LusTrick}
The following diagram commutes and consists of linear bijections.
\begin{equation}
\label{eq:LusTrick}
\begin{split}
\xymatrix{(V(\lambda)\otimes M)^G\ar[r]^P
\ar[d]_{\eta\otimes\id_M}&(M\otimes V(\lambda))^G
\ar[d]^{\id_M\otimes\xi}\\
M^\circ_{\lambda^*}\ar[r]_{\theta(w_0,\lambda^*)}&
M^\circ_{w_0\lambda^*}}
\end{split}
\end{equation}
\end{lemma}
\begin{proof}
By additivity, we can reduce to the case where $M$ is a simple
representation. If $M$ is not isomorphic to the dual of $V(\lambda)$,
then all four spaces are zero and the statement is banal. We therefore
assume that $M\cong V(\lambda^*)$; in this case, all four spaces are
one dimensional.

Let $m_{\lambda^*}$ be a highest weight vector in $M$ and set
$m_{w_0\lambda^*}=\theta(w_0,\lambda^*)\cdot m_{\lambda^*}$.
There exists a unique $G$-invariant bilinear form
$\Phi:V(\lambda)\times M\to\mathbb C$ such that
$\Phi(v_{w_0\lambda},m_{\lambda^*})=1$. This form $\Phi$ is
non-degenerate and a standard computation gives
$\Phi(v_\lambda,m_{w_0\lambda^*})=(-1)^{2\rho(\lambda)}$.

The assignment $v\otimes m\mapsto\Phi(v,\bm?)\,m$ defines a
$G$-equivariant isomorphism $V(\lambda)\otimes M\to\End(M)$.
The preimage $x$ of $\id_M$ by this bijection spans the
vector space $(V(\lambda)\otimes M)^G$. By construction,
$(\eta\otimes\id_M)(x)=m_{\lambda^*}$ and
$(\xi\otimes\id_M)(x)=(-1)^{2\rho(\lambda)}\,m_{w_0\lambda^*}$.
Thus, both paths around the diagram map $x$ to $m_{w_0\lambda^*}$.
\end{proof}

We take $M=V(\lambda_2)\otimes\cdots\otimes V(\lambda_n)$. We define
$M^\bullet$ to be the step in the isotypical filtration of $M$ where the
component $M^\circ$ is appended to smaller ones. There is a natural
quotient map $p:M^\bullet\to M^\circ$.

We set $\mathscr M=\mathscr Z(\lambda_2)\times\cdots\times
\mathscr Z(\lambda_n)$. Using the notation introduced in
sect.~\ref{ss:InterMult}, the MV basis of $M$ consists of
elements $\langle\mathbf Z\rangle$ for $\mathbf Z\in\mathscr M$.
Let $\mathscr M^\bullet$ be the set of all $Z\in\mathscr M$ such that
$\langle\mathbf Z\rangle\in M^\bullet$; since MV bases are $L$-perfect,
$\{\langle\mathbf Z\rangle\mid\mathbf Z\in\mathscr M^\bullet\}$ is a
basis of $M^\bullet$. Let $\mathscr M^\circ$ be the set of all
$Z\in\mathscr M^\bullet$ such that $\langle\mathbf Z\rangle\notin\ker p$;
then $\{p(\langle\mathbf Z\rangle)\mid\mathbf Z\in \mathscr M^\circ\}$
is a basis of $M^\circ$. In consequence, each weight subspace of
$M^\circ$ is endowed with a basis.

As a crystal, $\mathscr M$ decomposes as the disjoint union (direct
sum) of its connected components, and $\mathscr M^\circ$ is the union
of the connected components of $\mathscr M$ that are isomorphic to
$\mathscr Z(\lambda^*)$. For each connected component
$\mathscr C\subseteq\mathscr M^\circ$, the subspace of $M^\circ$ spanned by
$B_{\mathscr C}=\{p(\langle\mathbf Z\rangle)\mid\mathbf Z\in \mathscr C\}$
is a subrepresentation isomorphic to $V(\lambda^*)$, and by
Remark~\ref{rk:StringP2}, $B_{\mathscr C}$ identifies with the MV basis
of $V(\lambda^*)$. The action of $\theta(w_0,\lambda^*)$ therefore maps
the highest weight element in $B_{\mathscr C}$ to the lowest element in
$B_{\mathscr C}$. We conclude that the bottom horizontal arrow
in~\eqref{eq:LusTrick} maps the basis of $M^\circ_{\lambda^*}$ to the
basis of~$M^\circ_{w_0\lambda^*}$.

Each element in the MV basis of $V\bigl(\bm\lambda^{[1]}\bigr)=
M\otimes V(\lambda)$ is of the form $\langle\mathbf Z\rangle$,
with $\mathbf Z$ in $\mathscr Z\bigl(\bm\lambda^{[1]}\bigr)=
\mathscr M\times\mathscr Z(\lambda)$. Let $V(\lambda)_{\neq\lambda}$
be the sum of all the weight subspaces of $V(\lambda)$ other than the
higher weight subspace. Theorem~\ref{th:TriangIncl} implies that for
each $\mathbf Z_{(1)}\in\mathscr M$,
$$\bigl\langle\mathbf Z_{(1)}\bigr\rangle\otimes
\bigl\langle\{L_\lambda\}\bigr\rangle\equiv
\bigl\langle\bigl(\mathbf Z_{(1)},\{L_\lambda\}\bigr)\bigr\rangle
\quad\bigl(\bmod\ M\otimes V(\lambda)_{\neq\lambda}\bigr).$$
Thus, for $\mathbf Z_{(1)}\in\mathscr M$ and
$\mathbf Z=\bigl(\mathbf Z_{(1)},\{L_\lambda\}\bigr)$, we have
$(\id_M\otimes\xi)(\langle\mathbf Z\rangle)=
\bigl\langle\mathbf Z_{(1)}\bigr\rangle$.

As evidenced by the crystal structure on
$\mathscr M\otimes\mathscr Z(\lambda)$, the Satake basis of
the space $(M\otimes V(\lambda))^G$ consists of the vectors
$\langle\mathbf Z\rangle$ for the pairs
$\mathbf Z=\bigl(\mathbf Z_{(1)},\{L_\lambda\}\bigr)$ such that
$\mathbf Z_{(1)}\in\mathscr M^\circ_{w_0\lambda^*}$. Noting that
$\bigl\langle\mathbf Z_{(1)}\bigr\rangle\in M^\circ_{w_0\lambda^*}$
for those $\mathbf Z_{(1)}$, we conclude that the right vertical arrow
in~\eqref{eq:LusTrick} maps basis elements to basis elements.

Similarly, the left vertical arrow
in~\eqref{eq:LusTrick} maps the Satake basis of
$(V(\lambda)\otimes M)^G$ to the basis of~$M^\circ_{\lambda^*}$.
Lemma~\ref{le:LusTrick} then concludes the proof of
Theorem~\ref{th:CycPerm}.

\subsection{Tensor product with an invariant element}
\label{ss:ProdInvElem}
Let $(n',n'')$ be a composition of $n$ in two parts.
Correspondingly, we write $\bm\lambda\in(\Lambda^+)^n$ as a
concatenation $(\bm\lambda',\bm\lambda'')$ and view each element
in $\mathscr Z(\bm\lambda)$ as a pair $(\mathbf Z',\mathbf Z'')\in
\mathscr Z(\bm\lambda')\times\mathscr Z(\bm\lambda'')$.

The following proposition implies that the Satake basis of the
invariant subspace of $V(\bm\lambda)$ satisfies the second item
in Khovanov and Kuperberg's list of properties for the dual
canonical basis (see the introduction of \cite{KhovanovKuperberg}).

\begin{proposition}
\label{pr:ProdInvElem}
Let $(\mathbf Z',\mathbf Z'')\in\mathscr Z(\bm\lambda)$. If
$\bigl\langle\mathbf Z'\bigr\rangle\in V(\bm\lambda')^G$, then
$\bigl\langle\mathbf Z'\bigr\rangle\otimes
\bigl\langle\mathbf Z''\bigr\rangle=
\bigl\langle(\mathbf Z',\mathbf Z'')\bigr\rangle.$
\end{proposition}
\begin{proof}
Let $\mathbf Z'\in\mathscr Z(\bm\lambda')$. Recall the map
$m_{n'}:\Gr_{n'}\to\Gr$ defined in sect.~\ref{ss:MVCyc} and the notation
$\mu_I$ from sect.~\ref{ss:JDecMVCyc} and set $\mu=\mu_I(\mathbf Z')$.
Then $m_{n'}(\mathbf Z')\subseteq\overline{\Gr^\mu}$ and
$\bigl\langle\mathbf Z'\bigr\rangle$ appear in the isotypical
filtration of $V(\bm\lambda')$ at the step where the component of
type $V(\mu)$ is appended.

If $\bigl\langle\mathbf Z'\bigr\rangle\in V(\bm\lambda')^G$, then
$\mu=0$, accordingly $\overline{\Gr^\mu}=\{L_0\}$, and as a result
$$\overline{\mathbf Z'}\subseteq(m_{n'})^{-1}(\{L_0\})
\subseteq(m_{n'})^{-1}(T_0).$$

This implies that no MV cycle in $\mathscr Z(\bm\lambda')$ can be
strictly contained in $\overline{\mathbf Z'}$. (Such a cycle would be
contained in $(m_{n'})^{-1}(T_0)$, so would be an irreducible component
of $\overline{\Gr_{n'}^{\bm\lambda'}}\cap(m_{n'})^{-1}(T_0)$, and~would
end up having dimension $\rho(|\bm\lambda'|)$, the same as $\mathbf Z'$.)
The desired result now directly follows from Theorem~\ref{th:TriangIncl}.
\end{proof}

\section{Applications to the MV basis of $\mathbb C[N]$}
\label{se:AppC[N]}
We adopt the notation set up in the preamble of sect.~\ref{se:LPerf}.
Let $N$ be the unipotent radical of the Borel subgroup $B$ and let
$\mathbb C[N]$ be the algebra of regular functions on $N$. At the
expense of an isogeny, which does not alter $N$, we can assume that
$G$ is simply-connected.

For each dominant weight $\lambda\in\Lambda^+$, we can choose a highest
weight vector $v_\lambda$ in the representation $V(\lambda)$ and define
the linear form $v_\lambda^*:V(\lambda)\to\mathbb C$ such that
$\langle v_\lambda^*,v_\lambda\rangle=1$ et $\langle v_\lambda^*,v\rangle=0$
for all weight vectors $v$ of weight other than $\lambda$. This yields
an embedding $\Psi_\lambda:v\mapsto\langle v_\lambda^*,\bm?v\rangle$ of
$V(\lambda)$ into $\mathbb C[N]$, where $\langle v_\lambda^*,\bm?v\rangle$
stands for the function $n\mapsto\langle v_\lambda^*,nv\rangle$.
The MV bases of the representations $V(\lambda)$ can be transported
to $\mathbb C[N]$ through these maps $\Psi_\lambda$, and they glue
together to form a basis of $\mathbb C[N]$, which we call the MV basis
of $\mathbb C[N]$ (see~\cite{BaumannKamnitzerKnutson}).

The algebra $\mathbb C[N]$ comes with several remarkable bases: the
MV basis, subject of our current investigation, but also the dual
canonical basis of Lusztig/upper global basis of Kashiwara, and (in
simply laced type) the dual semicanonical basis. The theory of cluster
algebras was developed in order to compute effectively these bases
(or at least, the dual canonical basis). Concretely, the cluster
structure of $\mathbb C[N]$ allows to define specific elements,
called cluster monomials, which are linearly independent and easily
amenable to calculations. It is known that both the dual canonical
and the dual semicanonical bases contain all the cluster monomials
\cite{GeissLeclercSchroer06,KangKashiwaraKimOh}, but also that these
bases differ (except when cluster monomials span $\mathbb C[N]$).

The methods developed in sect.~\ref{se:CompTPB} allow to effectively
compute products of elements of the MV basis of $\mathbb C[N]$. This
allows us to prove that this basis contains quite a few cluster
monomials (Proposition~\ref{pr:ClusMon}) and that it generally differs
from both the dual canonical and the dual semicanonical bases
(Proposition~\ref{pr:BasesD4}).

\subsection{Cluster monomials}
\label{ss:ClusMon}
As explained in~\cite{GeissLeclercSchroer07}, each reduced word
$(i_1,\ldots,i_\ell)$ of the longest element $w_0$ in the Weyl group
$W$ yields a standard seed of the cluster structure of $\mathbb C[N]$.
The main result of this section, Proposition~\ref{pr:ClusMon}, claims
that sometimes the cluster monomials built from the variables of
a standard seed belong to the MV basis of $\mathbb C[N]$.

Set $\mathfrak t^\vee_{\mathbb R}=\Hom_{\mathbb Z}(\Lambda,\mathbb R)$
and let $C=\{x\in\mathfrak t^\vee_{\mathbb R}\mid
\forall i\in I,\ \langle x,\alpha_i\rangle>0\}$ be the Weyl chamber
in $\mathfrak t^\vee_{\mathbb R}$. We consider the following condition
about a reduced word $(i_1,\ldots,i_\ell)$:
\begin{description}
\item[(A)]
There exist $x_1\in s_{i_1}(C)$, $x_2\in(s_{i_1}s_{i_2})(C)$,
\dots, $x_\ell\in(s_{i_1}\cdots s_{i_\ell})(C)$ such that
$x_k-x_{k+1}\in C$ for each $k\in\{1,\ldots,\ell-1\}$.
\end{description}

For instance, choose $(x,y)\in C^2$ in such a way that the straight line
joining $x$ to $-y$ avoids all the two-codimensional faces of the Weyl
fan in $\mathfrak t^\vee_{\mathbb R}$. List in order the chambers
successively crossed by this line: $C$, $s_{i_1}C$, $(s_{i_1}s_{i_2})(C)$,
\dots The word $(i_1,i_2,\ldots)$ produced in this manner is then reduced
and obviously satisfies condition (A).

Let $Q\subseteq\Lambda$ be the root lattice. We denote by $Q_+$ the
positive cone in $Q$ with respect to the dominance order; in other words,
$Q_+$ is the set of all linear combinations of the simple roots~$\alpha_i$
with non-negative integral coefficients. We set $Q_-=-Q_+$.

\begin{lemma}
\label{le:CondA}
Let $(i_1,\ldots,i_\ell)$ be a reduced word, set
$w_k=s_{i_1}\cdots s_{i_k}$ for $k\in\{1,\ldots,\ell\}$, and let
$(\nu_1,\ldots,\nu_\ell)\in w_1(Q_-)\times\cdots\times w_\ell(Q_-)$.
Assume that $\nu_1+\cdots+\nu_k\in Q_+$ for all $k\in\{1,\ldots,\ell-1\}$,
that $\nu_1+\cdots+\nu_\ell=0$, and that $(i_1,\ldots,i_\ell)$
satisfies condition (A). Then $\nu_1=\cdots=\nu_\ell=0$.
\end{lemma}
\begin{proof}
We set $\mu_0=0$ and $\mu_k=\nu_1+\cdots+\nu_k$ for
$k\in\{1,\ldots,\ell\}$. We pick $x_1$, \dots, $x_\ell$
as stated in condition (A). Then
$$\sum_{k=1}^\ell\langle x_k,\nu_k\rangle=
\sum_{k=1}^\ell\langle x_k,\mu_k-\mu_{k-1}\rangle=
\sum_{k=1}^{\ell-1}\langle x_k-x_{k+1},\mu_k\rangle.$$
From $x_k\in w_k(C)$ and $\nu_k\in w_k(Q_-)$, we deduce that
$\langle x_k,\nu_k\rangle\leq0$ for each $k\in\{1,\ldots,\ell\}$.
On the other hand, from $x_k-x_{k+1}\in C$ and $\mu_k\in Q_+$,
we deduce that $\langle x_k-x_{k+1},\mu_k\rangle\geq0$
for each $k\in\{1,\ldots,\ell-1\}$. We conclude that each
$\langle x_k,\nu_k\rangle$ is indeed zero, which implies $\nu_k=0$.
\end{proof}

In sect.~\ref{ss:CycPerm}, we defined, for each
$(\lambda,w)\in\Lambda^+\times W$, a product $\theta(\lambda,w)$
of divided powers of the root vectors $f_i$. We can then set
$v_{w\lambda}=\theta(\lambda,w)\cdot v_\lambda$; this is a vector
of weight $w\lambda$ in $V(\lambda)$. We define
$\Delta_{\lambda,w\lambda}=\Psi_\lambda(v_{w\lambda})$, usually
called a flag minor if $\lambda$ is minuscule. We denote by
$\{\varpi_i\mid i\in I\}$ the set of fundamental weights.

\begin{proposition}
\label{pr:ClusMon}
Let $(i_1,\ldots,i_\ell)$ be a reduced word and define
$x_k=\Delta_{\varpi_{i_k},s_{i_1}\cdots s_{i_k}\varpi_{i_k}}$
for each $k\in\{1,\ldots,\ell\}$.
If $(i_1,\ldots,i_\ell)$ satisfies condition (A), then any monomial
in $x_1$, \dots, $x_\ell$ belongs to the MV basis of $\mathbb C[N]$.
\end{proposition}
\begin{proof}
We choose $\bm\lambda=(\lambda_1,\ldots,\lambda_\ell)$ in $(\Lambda^+)^\ell$.
For $k\in\{1,\ldots,\ell\}$, we set $w_k=s_{i_1}\cdots s_{i_k}$.
The extremal weight vector $v_{w_k\lambda_k}\in V(\lambda_k)$ belongs
to the MV basis (\cite{BaumannKamnitzerKnutson}, Remark~2.10 and
Theorem~5.2), so $v_{w_k\lambda_k}=\langle Z_k\rangle$ where $Z_k$
is the cycle $\overline{\Gr^{\lambda_k}}\cap T_{w_k\lambda_k}$.
We set $\mu=w_1\lambda_1+\cdots+w_\ell\lambda_\ell$ and
$\mathbf Z=(Z_1,\ldots,Z_\ell)$. We adopt the convention of
sect.~\ref{ss:InterMult} and regard $\mathbf Z$ as an element
of~$\mathscr Z(\bm\lambda)_\mu$; then $\dbllangle\mathbf Z\dblrangle
=v_{w_1\lambda_1}\otimes\cdots\otimes v_{w_\ell\lambda_\ell}$ .

Let us expand this element on the MV basis of $V(\bm\lambda)$. As in
Theorem~\ref{th:TransMat}, we write
\begin{equation}
\label{eq:PfClusMon}
\dbllangle\mathbf Z\dblrangle=\sum_{\mathbf Z'\in
\mathscr Z(\bm\lambda)_\mu}a_{\mathbf Z',\mathbf Z}\;
\bigl\langle\mathbf Z'\bigr\rangle.
\end{equation}
Suppose $\mathbf Z'\in \mathscr Z(\bm\lambda)_\mu$ satisfies
$a_{\mathbf Z',\mathbf Z}\neq0$. Let
$(\nu_1,\ldots,\nu_\ell)\in\Lambda^\ell$ be the tuple
of weights such that $\mathbf Z'\in\mathscr Z(\lambda_1)_{\nu_1}
\times\cdots\times\mathscr Z(\lambda_\ell)_{\nu_\ell}$.
For each $k\in\{1,\ldots,\ell\}$, we have
$\mathscr Z(\lambda_k)_{\nu_k}\neq\varnothing$, so
$w_k^{-1}\nu_k$ is a weight of $V(\lambda_k)$, whence
$(\nu_k-w_k\lambda_k)\in w_k(Q_-)$. From
$\nu_1+\cdots+\nu_\ell=\mu$, we deduce
that~$(\nu_1-w_1\lambda_1)+\cdots+(\nu_\ell-w_\ell\lambda_\ell)=0$.
And by Corollary~\ref{co:TriangWeight}, we get
$$(\nu_1-w_1\lambda_1)+\cdots+(\nu_k-w_k\lambda_k)\in Q_+$$
for each $k\in\{1,\ldots,\ell-1\}$.
Then, assuming that $(i_1,\ldots,i_\ell)$ satisfies condition~(A)
and applying Lemma~\ref{le:CondA}, we obtain $\nu_k=w_k\lambda_k$ for
each $k\in\{1,\ldots,\ell-1\}$. In other words, none of the
inequalities given in Corollary~\ref{co:TriangWeight} is strict.
By Remark~\ref{rk:TriangStrict}, this forces $\mathbf Z'=\mathbf Z$.
Thus, the expansion~\eqref{eq:PfClusMon} contains a single term,
namely $\langle\mathbf Z\rangle$.

Set $\lambda=\lambda_1+\cdots+\lambda_\ell$ and let
$p:V(\bm\lambda)\to V(\lambda)$ be the unique morphism that maps
$v_{\lambda_1}\otimes\cdots\otimes v_{\lambda_n}$ to $v_\lambda$.
Noting that $p$ is the quotient map to the top factor in the isotypical
filtration of $V(\bm\lambda)$ and applying Remark~\ref{rk:StringP2},
we obtain that $p(\langle\mathbf Z\rangle)$ belongs to the MV basis of
$V(\lambda)$. From the equality
$v_{w_1\lambda_1}\otimes\cdots\otimes v_{w_\ell\lambda_\ell}=
\langle\mathbf Z\rangle$, we deduce that
$$\Delta_{\lambda_1,w_1\lambda_1}\cdots
\Delta_{\lambda_\ell,w_\ell\lambda_\ell}=
\bigl\langle v_\lambda^*,
p\bigl(\bm?\langle\mathbf Z\rangle\bigr)\bigr\rangle=
\Psi_\lambda\bigl(p\bigl(\langle\mathbf Z\rangle\bigr)\bigr)$$
belongs to the MV basis of $\mathbb C[N]$. The claim in the
proposition is the particular case where each $\lambda_k$ is a
multiple of $\varpi_{i_k}$.
\end{proof}

\subsection{A computation in type $D_4$}
\label{ss:BasesD4}
In~\cite{KashiwaraSaito}, Kashiwara and Saito found an example in type
$A_5$ where the singular support of a simple perverse sheaf related to
the canonical basis is not irreducible. Looking again at this situation,
Geiß, Leclerc and Schröer \cite{GeissLeclercSchroer05} computed the
dual canonical and dual semicanonical elements and found that they were
different. They also observed that a similar situation occurs in type
$D_4$. In~\cite{BaumannKamnitzerKnutson}, Kamnitzer, Knutson and the
first author observed that in both situations $A_5$ and $D_4$, the MV
basis is a third basis, different from the two other ones.

Let us have a closer look at the $D_4$ case. As usual, we label the
vertices of the Dynkin diagram from $1$ to $4$, with $2$ for the central
node. Our three bases are indexed by the crystal $B(\infty)$: given
$b\in B(\infty)$, we denote the corresponding dual semicanonical basis
element by $C(b)$, the dual canonical basis element by $C'(b)$, and the
MV basis element by $C''(b)$. Calling $b_0$ the highest weight element
in $B(\infty)$, we set
$$b_1=\bigl(\tilde f_2\bigl(\tilde f_1\tilde f_3\tilde f_4\bigr)
\tilde f_2\bigr)^2\;b_0
\quad\text{and}\quad
b_{12}=\bigl(\tilde f_2\bigr)^2\bigl(\tilde f_1\tilde f_3
\tilde f_4\bigr)^2\bigl(\tilde f_2\bigr)^2\;b_0.$$

\begin{proposition}
\label{pr:BasesD4}
The basis elements are related by the equations
$$C(b_{12})=C''(b_{12})+2C(b_1)\quad\text{and}\quad
C'(b_{12})=C''(b_{12})+C(b_1).$$
\end{proposition}

The proof is given in \cite{BaumannKamnitzerKnutson}, sect.~2.7, except
for one justification left to the present paper. We here fill the gap.

The fundamental weight $\varpi_2$ is the highest root
$\alpha_1+2\alpha_2+\alpha_3+\alpha_4$. The crystal of the representation
$V(\varpi_2)$ (the adjoint representation) is pictured below. Highest
weights are towards the left, vertices are represented as keys
$\begin{smallmatrix}p\\[2pt]q\end{smallmatrix}$ with $p$, $q$ in
$\{1,2,3,4,\overline1,\overline2,\overline3,\overline4\}$, and
operators $\tilde f_1$, $\tilde f_2$, $\tilde f_3$ and $\tilde f_4$
are indicated by dashed, solid, dotted and dash-dotted arrows, respectively.

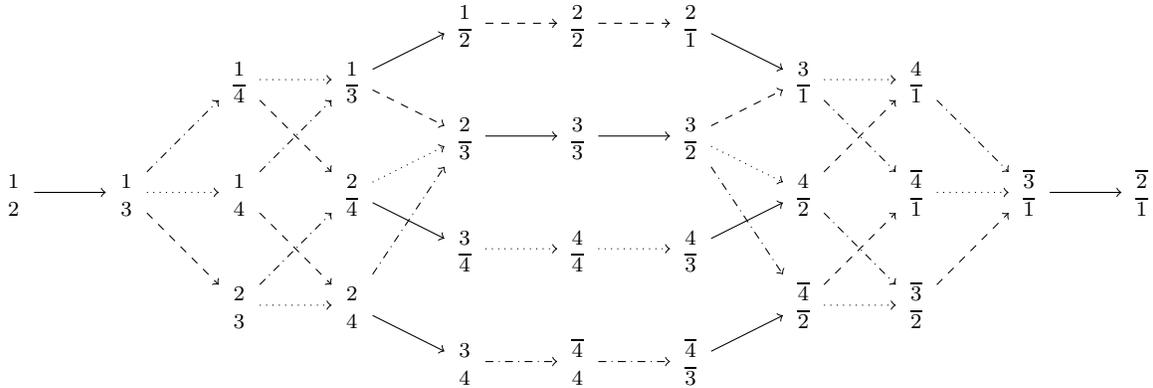
\begin{center}
\begin{tikzpicture}
\begin{scope}[scale=1.5]
\node (A2) at (0,0) {$\begin{smallmatrix}\vphantom{\overline1}1\\[2pt]
  \vphantom{\overline1}2\end{smallmatrix}$};
\node (B2) at (1,0) {$\begin{smallmatrix}\vphantom{\overline1}1\\[2pt]
  \vphantom{\overline1}3\end{smallmatrix}$};
\node (C4) at (2,1) {$\begin{smallmatrix}\vphantom{\overline1}1\\[2pt]
  \overline4\end{smallmatrix}$};
\node (C3) at (2,0) {$\begin{smallmatrix}\vphantom{\overline1}1\\[2pt]
  \vphantom{\overline1}4\end{smallmatrix}$};
\node (C1) at (2,-1) {$\begin{smallmatrix}\vphantom{\overline1}2\\[2pt]
  \vphantom{\overline1}3\end{smallmatrix}$};
\node (D1) at (3,1) {$\begin{smallmatrix}\vphantom{\overline1}1\\[2pt]
  \overline3\end{smallmatrix}$};
\node (D3) at (3,0) {$\begin{smallmatrix}\vphantom{\overline1}2\\[2pt]
  \overline4\end{smallmatrix}$};
\node (D4) at (3,-1) {$\begin{smallmatrix}\vphantom{\overline1}2\\[2pt]
  \vphantom{\overline1}4\end{smallmatrix}$};
\node (E1) at (4,1.5) {$\begin{smallmatrix}\vphantom{\overline1}1\\[2pt]
  \overline2\end{smallmatrix}$};
\node (E2) at (4,.5) {$\begin{smallmatrix}\vphantom{\overline1}2\\[2pt]
  \overline3\end{smallmatrix}$};
\node (E3) at (4,-.5) {$\begin{smallmatrix}\vphantom{\overline1}3\\[2pt]
  \overline4\end{smallmatrix}$};
\node (E4) at (4,-1.5) {$\begin{smallmatrix}\vphantom{\overline1}3\\[2pt]
  \vphantom{\overline1}4\end{smallmatrix}$};
\node (F1) at (5,1.5) {$\begin{smallmatrix}\vphantom{\overline1}2\\[2pt]
  \overline2\end{smallmatrix}$};
\node (F2) at (5,.5) {$\begin{smallmatrix}\vphantom{\overline1}3\\[2pt]
  \overline3\end{smallmatrix}$};
\node (F3) at (5,-.5) {$\begin{smallmatrix}\vphantom{\overline1}4\\[2pt]
  \overline4\end{smallmatrix}$};
\node (F4) at (5,-1.5) {$\begin{smallmatrix}\overline4\\[2pt]
  \vphantom{\overline1}4\end{smallmatrix}$};
\node (G1) at (6,1.5) {$\begin{smallmatrix}\vphantom{\overline1}2\\[2pt]
  \overline1\end{smallmatrix}$};
\node (G2) at (6,.5) {$\begin{smallmatrix}\vphantom{\overline1}3\\[2pt]
  \overline2\end{smallmatrix}$};
\node (G3) at (6,-.5) {$\begin{smallmatrix}\vphantom{\overline1}4\\[2pt]
  \overline3\end{smallmatrix}$};
\node (G4) at (6,-1.5) {$\begin{smallmatrix}\overline4\\[2pt]
  \overline3\end{smallmatrix}$};
\node (H1) at (7,1) {$\begin{smallmatrix}\vphantom{\overline1}3\\[2pt]
  \overline1\end{smallmatrix}$};
\node (H3) at (7,0) {$\begin{smallmatrix}\vphantom{\overline1}4\\[2pt]
  \overline2\end{smallmatrix}$};
\node (H4) at (7,-1) {$\begin{smallmatrix}\overline4\\[2pt]
  \overline2\end{smallmatrix}$};
\node (I4) at (8,1) {$\begin{smallmatrix}\vphantom{\overline1}4\\[2pt]
  \overline1\end{smallmatrix}$};
\node (I3) at (8,0) {$\begin{smallmatrix}\overline4\\[2pt]
  \overline1\end{smallmatrix}$};
\node (I1) at (8,-1) {$\begin{smallmatrix}\overline3\\[2pt]
  \overline2\end{smallmatrix}$};
\node (J2) at (9,0) {$\begin{smallmatrix}\overline3\\[2pt]
  \overline1\end{smallmatrix}$};
\node (K2) at (10,0) {$\begin{smallmatrix}\overline2\\[2pt]
  \overline1\end{smallmatrix}$};
\draw[->,solid] (A2)--(B2);
\draw[->,dashed] (B2)--(C1);
\draw[->,dotted] (B2)--(C3);
\draw[->,dashdotted] (B2)--(C4);
\draw[->,dashed] (C4)--(D3);
\draw[->,dashed] (C3)--(D4);
\draw[->,dotted] (C1)--(D4);
\draw[->,dotted] (C4)--(D1);
\draw[->,dashdotted] (C3)--(D1);
\draw[->,dashdotted] (C1)--(D3);
\draw[->,dashed] (D1)--(E2);
\draw[->,dotted] (D3)--(E2);
\draw[->,dashdotted] (D4)--(E2);
\draw[->,solid] (D1)--(E1);
\draw[->,solid] (D3)--(E3);
\draw[->,solid] (D4)--(E4);
\draw[->,dashed] (E1)--(F1);
\draw[->,solid] (E2)--(F2);
\draw[->,dotted] (E3)--(F3);
\draw[->,dashdotted] (E4)--(F4);
\draw[->,dashed] (F1)--(G1);
\draw[->,solid] (F2)--(G2);
\draw[->,dotted] (F3)--(G3);
\draw[->,dashdotted] (F4)--(G4);
\draw[->,dashed] (G2)--(H1);
\draw[->,dotted] (G2)--(H3);
\draw[->,dashdotted] (G2)--(H4);
\draw[->,solid] (G1)--(H1);
\draw[->,solid] (G3)--(H3);
\draw[->,solid] (G4)--(H4);
\draw[->,dashed] (H4)--(I3);
\draw[->,dashed] (H3)--(I4);
\draw[->,dotted] (H1)--(I4);
\draw[->,dotted] (H4)--(I1);
\draw[->,dashdotted] (H3)--(I1);
\draw[->,dashdotted] (H1)--(I3);
\draw[->,dashed] (I1)--(J2);
\draw[->,dotted] (I3)--(J2);
\draw[->,dashdotted] (I4)--(J2);
\draw[->,solid] (J2)--(K2);
\end{scope}
\end{tikzpicture}
\end{center}

If we endow the weight lattice $\Lambda$ with its usual basis
$(\varepsilon_1,\varepsilon_2,\varepsilon_3,\varepsilon_4)$, then the
weight of the element $\begin{smallmatrix}p\\[2pt]q\end{smallmatrix}$
is simply $\varepsilon_p+\varepsilon_q$, with the convention that
$\varepsilon_{\overline\imath}=-\varepsilon_i$ for $i\in\{1,2,3,4\}$.
The crystal contains four elements of weight zero, namely
$\begin{smallmatrix}\vphantom{\overline1}2\\[2pt]
\overline2\end{smallmatrix}$,
$\begin{smallmatrix}\vphantom{\overline1}3\\[2pt]
\overline3\end{smallmatrix}$,
$\begin{smallmatrix}\vphantom{\overline1}4\\[2pt]
\overline4\end{smallmatrix}$ and
$\begin{smallmatrix}\overline4\\[2pt]
\vphantom{\overline1}4\end{smallmatrix}$.

We set $\bm\lambda=(\varpi_2,\varpi_2)$ and look at the
tensor square $V(\bm\lambda)=V(\varpi_2)^{\otimes2}$. As in
sect.~\ref{ss:InterMult}, its MV basis consists of symbols
$\langle\mathbf Z\rangle$, where $\mathbf Z=(Z_1,Z_2)$ is a pair
in $\mathscr Z(\varpi_2)\times\mathscr Z(\varpi_2)$. In addition,
$V(\bm\lambda)$ is endowed with the tensor product basis. In the
matter of notation, we simply use the keys
$\begin{smallmatrix}p\\[2pt]q\end{smallmatrix}$ to denote MV cycles,
taking advantage that $\mathscr Z(\varpi_2)$ is isomorphic to the
crystal pictured above.

We claim that
\begin{equation}
\label{eq:CalcD4}
\begin{split}
\Bigl\langle\begin{smallmatrix}\vphantom{\overline1}2\\[2pt]
\overline3\end{smallmatrix}\Bigr\rangle\otimes
\Bigl\langle\begin{smallmatrix}\vphantom{\overline1}3\\[2pt]
\overline2\end{smallmatrix}\Bigr\rangle=
2\,&\Bigl\langle\Bigl(\begin{smallmatrix}\vphantom{\overline1}1\\[2pt]
\vphantom{\overline1}2\end{smallmatrix},
\begin{smallmatrix}\overline2\\[2pt]
\overline1\end{smallmatrix}\Bigr)\Bigr\rangle+
\Bigl\langle\Bigl(\begin{smallmatrix}\vphantom{\overline1}1\\[2pt]
\vphantom{\overline1}3\end{smallmatrix},
\begin{smallmatrix}\overline3\\[2pt]
\overline1\end{smallmatrix}\Bigr)\Bigr\rangle+
\Bigl\langle\Bigl(\begin{smallmatrix}\vphantom{\overline1}2\\[2pt]
\vphantom{\overline1}3\end{smallmatrix},
\begin{smallmatrix}\overline3\\[2pt]
\overline2\end{smallmatrix}\Bigr)\Bigr\rangle+
\Bigl\langle\Bigl(\begin{smallmatrix}\vphantom{\overline1}1\\[2pt]
\vphantom{\overline1}4\end{smallmatrix},
\begin{smallmatrix}\overline4\\[2pt]
\overline1\end{smallmatrix}\Bigr)\Bigr\rangle\\[2pt]
&\qquad
+\Bigl\langle\Bigl(\begin{smallmatrix}\vphantom{\overline1}1\\[2pt]
\overline4\end{smallmatrix},
\begin{smallmatrix}\vphantom{\overline1}4\\[2pt]
\overline1\end{smallmatrix}\Bigr)\Bigr\rangle+
\Bigl\langle\Bigl(\begin{smallmatrix}\vphantom{\overline1}2\\[2pt]
\vphantom{\overline1}4\end{smallmatrix},
\begin{smallmatrix}\overline4\\[2pt]
\overline2\end{smallmatrix}\Bigr)\Bigr\rangle+
\Bigl\langle\Bigl(\begin{smallmatrix}\vphantom{\overline1}2\\[2pt]
\overline4\end{smallmatrix},
\begin{smallmatrix}\vphantom{\overline1}4\\[2pt]
\overline2\end{smallmatrix}\Bigr)\Bigr\rangle+
\Bigl\langle\Bigl(\begin{smallmatrix}\vphantom{\overline1}1\\[2pt]
\overline3\end{smallmatrix},
\begin{smallmatrix}\vphantom{\overline1}3\\[2pt]
\overline1\end{smallmatrix}\Bigr)\Bigr\rangle+
\Bigl\langle\Bigl(\begin{smallmatrix}\vphantom{\overline1}2\\[2pt]
\overline3\end{smallmatrix},
\begin{smallmatrix}\vphantom{\overline1}3\\[2pt]
\overline2\end{smallmatrix}\Bigr)\Bigr\rangle.
\end{split}
\end{equation}

Let $p:V(\varpi_2)^{\otimes2}\to V(2\varpi_2)$ be the unique morphism
that maps $v_{\varpi_2}\otimes v_{\varpi_2}$ to $v_{2\varpi_2}$. Applying
$\Psi_{2\varpi_2}\circ p$ to the equality \eqref{eq:CalcD4}, we obtain
the equation
$$C''(b_{13})\,C''(b_{14})=2\,C''(b_1)+\sum_{i=2}^8C''(b_i)+C''(b_{12})$$
asserted without proof in \cite{BaumannKamnitzerKnutson}. Thus,
establishing \eqref{eq:CalcD4} will complete the proof of
Proposition~\ref{pr:BasesD4}. Actually, an inspection of the proof in
\textit{loc.~cit.} reveals that it is enough to justify that the coefficient
in front of $\Bigl\langle\Bigl(
\begin{smallmatrix}\vphantom{\overline1}1\\[2pt]
\vphantom{\overline1}2\end{smallmatrix},
\begin{smallmatrix}\overline2\\[2pt]
\overline1\end{smallmatrix}
\Bigr)\Bigr\rangle$ is strictly larger than one.

We will use Theorem~\ref{th:TransMat} to prove this fact.
Here the group $G^\vee$ is $\SO_8$. For $(i,j)\in\{1,\ldots,8\}$, we
denote by $E_{i,j}$ the matrix of size $8\times8$ with zeros everywhere
except for a one at position~$(i,j)$. For each coroot
$\alpha^\vee\in\Phi^\vee$, we define a subgroup
$x_{\alpha^\vee}:\mathbb C\to G^\vee$ by the following formulas,
where $I$ is the identity matrix, $a\in\mathbb C$, and $i$, $j$ are
elements in $\{1,2,3,4\}$ such that $i<j$.
\begin{xalignat*}2
x_{(\varepsilon_i-\varepsilon_j)^\vee}(a)&=I+a(E_{i,j}-E_{9-j,9-i})&
x_{(\varepsilon_i+\varepsilon_j)^\vee}(a)&=I+a(E_{i,9-j}-E_{j,9-i})\\[4pt]
x_{(\varepsilon_j-\varepsilon_i)^\vee}(a)&=I+a(E_{9-i,9-j}-E_{j,i})&
x_{(-\varepsilon_i-\varepsilon_j)^\vee}(a)&=I+a(E_{9-i,j}-E_{9-j,i})
\end{xalignat*}
For each root $\alpha$, we define a map
$\chi_\alpha:\mathbb C^{10}\to G^\vee\bigl(\mathbb C\bigl[z,z^{-1}\bigr]\bigr)$
by the formula
$$\chi_\alpha(\mathbf a)(z)=
\Biggl(\prod_{k=1}^8x_{\beta_k^\vee}(a_k)\Biggr)\
x_{\alpha^\vee}(a_9+za_{10})\ z^\alpha$$
where $\mathbf a$ stands for the tuple
$(a_1,\ldots,a_{10})\in\mathbb C^{10}$ and where
$\beta_1^\vee$, \dots, $\beta_8^\vee$ are the coroots $\beta^\vee$
such that $\langle\beta^\vee,\alpha\rangle=1$.
We specify the enumeration in our cases of interest as follows.

\begin{center}
\setlength\tabcolsep{.5em}
\begin{tabular}{>{$\scriptstyle}c<{$}>{$\scriptstyle}c<{$}
>{$\scriptstyle}c<{$}>{$\scriptstyle}c<{$}>{$\scriptstyle}c<{$}
>{$\scriptstyle}c<{$}>{$\scriptstyle}c<{$}>{$\scriptstyle}c<{$}
>{$\scriptstyle}c<{$}}
\toprule
\bm\alpha&\beta_1^\vee&\beta_2^\vee&\beta_3^\vee&\beta_4^\vee&\beta_5^\vee&
\beta_6^\vee&\beta_7^\vee&\beta_8^\vee\\
\midrule
\bm{\varepsilon_1+\varepsilon_2}&
(\varepsilon_1-\varepsilon_3)^\vee&(\varepsilon_1-\varepsilon_4)^\vee&
(\varepsilon_1+\varepsilon_4)^\vee&(\varepsilon_1+\varepsilon_3)^\vee&
(\varepsilon_2-\varepsilon_3)^\vee&(\varepsilon_2-\varepsilon_4)^\vee&
(\varepsilon_2+\varepsilon_4)^\vee&(\varepsilon_2+\varepsilon_3)^\vee\\[3pt]
\bm{-\varepsilon_1-\varepsilon_2}&
(-\varepsilon_2-\varepsilon_3)^\vee&(-\varepsilon_2-\varepsilon_4)^\vee&
(\varepsilon_4-\varepsilon_2)^\vee&(\varepsilon_3-\varepsilon_2)^\vee&
(-\varepsilon_1-\varepsilon_3)^\vee&(-\varepsilon_1-\varepsilon_4)^\vee&
(\varepsilon_4-\varepsilon_1)^\vee&(\varepsilon_3-\varepsilon_1)^\vee\\[3pt]
\bm{\varepsilon_2-\varepsilon_3}&
(\varepsilon_2-\varepsilon_1)^\vee&(\varepsilon_2-\varepsilon_4)^\vee&
(\varepsilon_2+\varepsilon_4)^\vee&(\varepsilon_1+\varepsilon_2)^\vee&
(-\varepsilon_1-\varepsilon_3)^\vee&(-\varepsilon_3-\varepsilon_4)^\vee&
(\varepsilon_4-\varepsilon_3)^\vee&(\varepsilon_1-\varepsilon_3)^\vee\\[3pt]
\bm{\varepsilon_3-\varepsilon_2}&
(\varepsilon_3-\varepsilon_1)^\vee&(\varepsilon_3-\varepsilon_4)^\vee&
(\varepsilon_3+\varepsilon_4)^\vee&(\varepsilon_1+\varepsilon_3)^\vee&
(-\varepsilon_1-\varepsilon_2)^\vee&(-\varepsilon_2-\varepsilon_4)^\vee&
(\varepsilon_4-\varepsilon_2)^\vee&(\varepsilon_1-\varepsilon_2)^\vee\\[3pt]
\bottomrule
\end{tabular}
\end{center}

We now define two charts on $\BDConv_2^{\bm\lambda}$, both with
$\mathbb C^{22}$ as domain:
\begin{align*}
\upphi_1:(x_1,x_2,\mathbf a,\mathbf b)&\mapsto\bigl(x_1,x_2;
\bigl[\chi_{\varepsilon_1+\varepsilon_2}(\mathbf a)(z-x_1),\,
\chi_{-\varepsilon_1-\varepsilon_2}(\mathbf b)(z-x_2)\bigr]\bigr),\\[8pt]
\upphi_2:(x_1,x_2,\mathbf a',\mathbf b')&\mapsto\bigl(x_1,x_2;
\bigl[\chi_{\varepsilon_2-\varepsilon_2}(\mathbf a')(z-x_1),\,
\chi_{\varepsilon_3-\varepsilon_2}(\mathbf b')(z-x_2)\bigr]\bigr).
\end{align*}

One can then compute the transition map between these two charts.
(The calculations were actually carried out with the help of the
computer algebra system \textsc{Singular}~\cite{Singular}.) One finds
the variables $a'_1$, \ldots, $b'_{10}$ as rational functions in
$x_2-x_1$, $a_1$, \ldots, $b_{10}$. We denote by $f$ the l.c.m.\
of the denominators.

Recall the notation used in sect.~\ref{ss:InterMult}.
In the chart $\upphi_1$, the cycle $\mathcal Y\Bigl(
\begin{smallmatrix}\vphantom{\overline1}1\\[2pt]
\vphantom{\overline1}2\end{smallmatrix},
\begin{smallmatrix}\overline2\\[2pt]
\overline1\end{smallmatrix}\Bigr)$ is defined by the equations
$a_1=\cdots=a_{10}=x_2-x_1=0$, so the ideal in
$R=\mathbb C[x_1,x_2,a_1,\ldots,a_{10},b_1,\ldots,b_{10}]$ of
$$V=\upphi_1^{-1}\Bigl(\mathcal Y\Bigl(
\begin{smallmatrix}\vphantom{\overline1}1\\[2pt]
\vphantom{\overline1}2\end{smallmatrix},
\begin{smallmatrix}\overline2\\[2pt]
\overline1\end{smallmatrix}\Bigr)\Bigr)$$
is
$$\mathfrak p=(a_1,\ldots,a_{10},x_2-x_1).$$

In the chart $\upphi_2$, the cycle
$\mathcal X\Bigl(\begin{smallmatrix}\vphantom{\overline1}2\\[2pt]
\overline3\end{smallmatrix},
\begin{smallmatrix}\vphantom{\overline1}3\\[2pt]
\overline2\end{smallmatrix}\Bigr)$ is defined by the equations
$a'_2=a'_3=a'_4=a'_8=a'_9=a'_{10}=b'_2=b'_3=b'_4=b'_8=0$.
Since the zero locus of $f$ contains the locus where the transition
map between the charts is not defined, the ideal $\mathfrak q$ of the
subvariety
$$X=\upphi_1^{-1}\Bigl(\mathcal X\Bigl(
\begin{smallmatrix}\vphantom{\overline1}2\\[2pt]
\overline3\end{smallmatrix},
\begin{smallmatrix}\vphantom{\overline1}3\\[2pt]
\overline2\end{smallmatrix}\Bigr)\Bigr)$$
is the preimage in $R$ of the ideal
$\mathfrak q_f=(a'_2,a'_3,a'_4,a'_8,a'_9,a'_{10},b'_2,b'_3,b'_4,b'_8)$
of the localized ring $R_f$. \textsc{Singular} gives the following
expression:
\begin{multline*}
\mathfrak q=(
a_1a_4+a_2a_3,\;
a_1a_6-a_2a_5,\;
a_3a_6+a_4a_5,\;
a_1a_7-a_3a_5,\;
a_2a_7+a_4a_5,\;
a_1a_8-a_4a_5,\\
a_2a_8-a_4a_6,\;
a_3a_8-a_4a_7,\;
a_5a_8+a_6a_7,\;
a_9,\;
a_{10},\;
a_1b_4+a_2b_3+a_3b_2+a_4b_1,\\
a_5b_4+a_6b_3+a_7b_2+a_8b_1,\;
a_3b_6+a_4b_5+a_7b_2+a_8b_1-(x_2-x_1),\;
a_2b_7-a_3b_6+a_6b_3-a_7b_2,\\
a_1b_8+a_3b_6+a_5b_4+a_7b_2,\;
a_2b_8-a_4b_6+a_6b_4-a_8b_2,\;
a_3b_8-a_4b_7+a_7b_4-a_8b_3
).
\end{multline*}

We observe that $\mathfrak q\subseteq\mathfrak p$, hence $V\subseteq X$.

Let $a$ and $x$ be two indeterminates. Let $B$ be the field
$\mathbb C(x,b_1,\ldots,b_{10})$. Extract the last seven equations
from $\mathfrak q$ and remove the term $x_2-x_1$ in the third one:
we then deal with seven linear equations with coefficients in $B$
in the eight variables $a_1$, \dots, $a_8$. This system has a
non-zero solution $(c_1,\ldots,c_8)\in B^8$. We can then define an
algebra morphism $u:R/\mathfrak q\to B[a]/(a^2)$ by
\begin{alignat*}3
&u(x_1)=u(x_2)=x,&\qquad&u(b_i)=b_i&\quad&
\text{for $i\in\{1,\ldots,10\}$},\\[2pt]
&u(a_9)=u(a_{10})=0,&&u(a_i)=c_ia&&\text{for $i\in\{1,\ldots,8\}$}.
\end{alignat*}
The ring $B[a]/(a^2)$ is local with maximal ideal $(a)$ and the preimage
of this ideal by $u$ is the ideal $\mathfrak p/\mathfrak q$ of
$R/\mathfrak q$.

Let $A$ be the localization of $R/\mathfrak q$ at $\mathfrak p/\mathfrak q$.
Then $u$ extends to an algebra morphism $\overline u:A\to B[a]/(a^2)$.
By construction, the kernel of $\overline u$ contains $x_2-x_1$ but not
all $a_1$, \dots, $a_8$. Therefore $x_2-x_1$ does not generate the maximal
ideal of $A$. Since $A$ is the local ring $\mathscr O_{V,X}$ of $X$ along
$V$, this means that the order of vanishing of $x_2-x_1$ along $V$ is
larger than one. In other words, the multiplicity of
$\mathcal Y\Bigl(
\begin{smallmatrix}\vphantom{\overline1}1\\[2pt]
\vphantom{\overline1}2\end{smallmatrix},
\begin{smallmatrix}\overline2\\[2pt]
\overline1\end{smallmatrix}\Bigr)$ in the
intersection product~$\mathcal X\Bigl(
\begin{smallmatrix}\vphantom{\overline1}2\\[2pt]
\overline3\end{smallmatrix},
\begin{smallmatrix}\vphantom{\overline1}3\\[2pt]
\overline2\end{smallmatrix}\Bigr)\cdot
\BDConv_2^{\bm\lambda}\bigl|_\Delta$ is larger than one. Applying
Theorem~\ref{th:TransMat}, we conclude that in~\eqref{eq:CalcD4}
the coefficient in front of
$\Bigl\langle\Bigl(
\begin{smallmatrix}\vphantom{\overline1}1\\[2pt]
\vphantom{\overline1}2\end{smallmatrix},
\begin{smallmatrix}\overline2\\[2pt]
\overline1\end{smallmatrix}
\Bigr)\Bigr\rangle$ is strictly larger than one.

Pierre Baumann,
Institut de Recherche Mathématique Avancée,
Université de Strasbourg et CNRS UMR 7501,
7 rue René Descartes,
67084 Strasbourg Cedex,
France.\\
\texttt{p.baumann@unistra.fr}
\medskip

Stéphane Gaussent,
Institut Camille Jordan,
Université de Lyon, UJM et CNRS UMR 5208,
23 rue du Docteur Paul Michelon,
42023 Saint-Étienne Cedex 2,
France.\\
\texttt{stephane.gaussent@univ-st-etienne.fr}
\medskip

Peter Littelmann,
Universität zu Köln,
Mathematisches Institut,
Weyertal 86--90,
50931 Köln,
Germany.\\
\texttt{peter.littelmann@math.uni-koeln.de}
\end{document}